\documentclass[11pt, reqno]{amsart}
\usepackage{amsfonts,amsmath,latexsym,amssymb,verbatim,amsbsy,amsthm,mathabx}
\usepackage{dsfont,bm}
\numberwithin{equation}{section}
\usepackage{enumitem}
\usepackage{booktabs}
\usepackage{array}
\usepackage{placeins}
\usepackage{float}
\usepackage{graphicx}
\usepackage{subcaption}
\usepackage{tikz}

\usepackage[margin=1in]{geometry}
\usepackage[dvipsnames]{xcolor}
\usepackage{aliascnt}

\usepackage[colorlinks=true, pdfstartview=FitV, linkcolor=RoyalBlue, citecolor=BrickRed, urlcolor=RoyalBlue]{hyperref}

\definecolor{labelkey}{rgb}{0,0,1}

\newtheorem{theorem}{Theorem}[section]

\newaliascnt{lemma}{theorem}
\newtheorem{lemma}[lemma]{Lemma}
\aliascntresetthe{lemma}

\newaliascnt{proposition}{theorem}
\newtheorem{proposition}[proposition]{Proposition}
\aliascntresetthe{proposition}

\newaliascnt{corollary}{theorem}
\newtheorem{corollary}[corollary]{Corollary}
\aliascntresetthe{corollary}

\theoremstyle{definition}
\newaliascnt{definition}{theorem}

\aliascntresetthe{definition}

\theoremstyle{remark}
\newaliascnt{remark}{theorem}
\newtheorem{remark}[remark]{Remark}
\aliascntresetthe{remark}

\theoremstyle{plain}
\newaliascnt{claim}{theorem}

\aliascntresetthe{claim}

\def \e {\varepsilon}

\def \cA {\mathcal{A}}
\def \cB {\mathcal{B}}

\def \cE {\mathcal{E}}
\def \cF {\mathcal{F}}

\def \cL {\mathcal{L}}

\def \cN {\mathcal{N}}

\def \cP {\mathcal{P}}

\def \cV {\mathcal{V}}
\def \cW {\mathcal{W}}

\newcommand{\N}{\ensuremath{\mathbb{N}}}

\newcommand{\R}{\ensuremath{\mathbb{R}}}

\newcommand{\E}{\ensuremath{\mathbb{E}}}
\renewcommand{\S}{\ensuremath{\mathbb{S}}}

\def \one {{\mathds{1}}}

\def \Cov {\operatorname{Cov}}
\def \rMF {\mathrm{MF}}
\def \tr {\operatorname{tr}}

\DeclareMathOperator{\Id}{Id}

\DeclareMathOperator{\Law}{Law}

\renewcommand{\geq}{\geqslant}
\renewcommand{\ge}{\geqslant}
\renewcommand{\leq}{\leqslant}
\renewcommand{\le}{\leqslant}

\def \gap {\operatorname{gap}}

\def \barX {\Bar{X}}

\def \dY  {\, \mathrm{d}Y}

\def \dt  {\, \mathrm{d}t}

\def \ds  {\, \mathrm{d}s}

\def \dB  {\, \mathrm{d} B}

\def \dd  {\mathrm{d}}

\def \P {\mathbb{P}}

\def \hX {\widehat{X}}

\newcommand{\push}{\mathbin{\#}}
\def \barB {\overline{B}}
\begin{document}

\title[Variance-Feedback Noise in Interacting Diffusions]{Variance-Feedback Noise in Interacting Diffusions: Shape Dynamics and Mean-Field Limits}

\author{Ben Hambly}
\address{Mathematical Institute, University of Oxford, Oxford, OX2 6GG, UK}
\email{ben.hambly@maths.ox.ac.uk}

\author{Vinh Nguyen}
\address{Department of Mathematics, University of California, Berkeley, CA 94720, USA}
\email{vnguyen26@berkeley.edu}

\date{\today}
\thanks{{\bf Acknowledgments:} V.N. was partially supported by an AMS--Simons travel grant.}

\begin{abstract}
We study a system of $N$ attractive interacting diffusions where the variance feeds back into the noise. That is the common and idiosyncratic Brownian forcing has amplitude $\kappa(S_t^N)^q$, with $S_t^N$ the empirical variance of the system, $\kappa>0$ and $q\geq 0$. With the interaction kernel given by $K(z)=|z|^p z$, for $p\geq 0$, feedback through the empirical variance gives the system critical homogeneity at $q=(p+2)/4$.  For this exponent an intrinsic time change separates the variance from the normalized centered configuration: the latter evolves autonomously on a unit sphere and, when it is not purely driven by common noise and the sphere has positive dimension,
is a reversible elliptic diffusion with an explicit invariant Gibbs measure. In this case, for $p>0$, the stationary averaged dissipation gives an exact finite-particle trichotomy between consensus, recurrent excursions through arbitrarily small and large variance scales, and finite-time explosion. At critical homogeneity we also derive an autonomous unit-variance mean-field equation and its constrained free-energy dissipation identity. Across the subcritical, critical, and supercritical regimes, we establish well-posedness and boundary classifications and prove conditional propagation of chaos on intervals where the limiting solution is unique and nonexplosive. We connect the sharp finite-particle coercivity constants to a mean-field variational problem.  Finite-size It\^o corrections enlarge the universal coercive consensus region. We solve the linear model, where $p=0$, completely, including invariant laws, explosion probabilities, noncommuting limits, and an $N^{-1/2}$ transition window.
\end{abstract}
\maketitle
\tableofcontents

\section{Introduction}\label{sec:intro}
For $N\ge2$, let $B^0,B^1,\ldots,B^N$ be independent Brownian motions in $\R^n$, and let $\rho\in[0,1]$.  We study the system of interacting diffusion processes
\begin{equation}\label{eq:particle}
 \dd X_t^i
 = \frac1N\sum_{k=1}^N K(X_t^k-X_t^i)\,\dd t
 +\psi(S_t^N)\Bigl(\rho\,\dd B_t^0+\sqrt{1-\rho^2}\,\dd B_t^i\Bigr),
 \qquad 1\le i\le N,
\end{equation}
where the empirical variance is defined by
\begin{equation*}
 S_t^N:=\frac1{2N^2}\sum_{i,k=1}^N|X_t^i-X_t^k|^2. 
\end{equation*}
The interaction kernel and the noise strength function are respectively given by 
\begin{equation*}
 K(z)=|z|^p z \quad \forall z\in\R^n, \quad p\ge0 \qquad \text{ and } \qquad \psi(s)=\kappa s^q\quad \forall s\ge0,\quad q\ge0,\quad \kappa>0.
\end{equation*}
For $q=0$, we use the convention that $s^0=1$ for all $s\geq 0$, so the noise amplitude is the constant $\kappa$.

A basic question concerning such a system is whether the particles reach relative consensus, that is $S_t^N\to 0$ as $t\to \infty$, or equivalently $\max_{i,j}|X_t^i-X_t^j|\to 0$ as $t\to \infty$ for fixed $N$. We use \textit{consensus} in this sense throughout the paper. The convergence of all particles to a common finite point additionally requires the convergence of the center, an important distinction under pure common noise.
The novel feature that distinguishes our model and gives it a rich structure is the feedback in its stochastic forcing. The noise coefficient is a function of a summary statistic of the total particle configuration
\begin{equation*}
 \kappa(S_t^N)^q\Bigl(\rho\,\dd B_t^0+\sqrt{1-\rho^2}\,\dd B_t^i\Bigr).
\end{equation*}
A simple interpretation is that the noise coefficient controls mobility or responsiveness.  For example, in an opinion-dynamics model, a tightly clustered population is relatively settled and experiences weak fluctuations, whereas a more dispersed population is more susceptible to common news and individual shocks.  Equivalently, in a crowding interpretation, particles that are tightly packed have less freedom to move, so diffuse more slowly, whereas a dispersed configuration permits larger fluctuations.

The noise contains both a common Brownian motion $B^0$, shared by all particles, and independent idiosyncratic Brownian motions $B^i$.  The parameter $\rho$ interpolates between these two sources of randomness, while their common amplitude is determined endogenously by the empirical variance of the whole configuration. Thus the noise is not prescribed independently of the system: dispersion feeds back into the strength of both the common and idiosyncratic fluctuations.  This feedback is degenerate at consensus when $q>0$, is non-Lipschitz there when $0<q<1/2$, and can become strong enough at large variance to compete with or dominate the attractive interaction.  These three features, common noise, idiosyncratic noise, and global variance-dependent intensity, drive the phenomena studied in the paper. We aim to classify the behavior of this particle system for the full range of values of the parameters $(p,q)$.

\subsection{The centered model, main results and approach}

The basic structural observation is that the common Brownian motion disappears after centering.  Writing
\[
\bar X_t^N=\frac1N\sum_{i=1}^N X_t^i,
\qquad
Y_t^i=X_t^i-\bar X_t^N,
\]
the variance satisfies the exact scalar stochastic differential equation
\begin{equation}\label{eq:intro-variance}
 \dd S_t^N
 =\left[-D_{p,N}(Y_t)+c_N(S_t^N)^{2q}\right]\dd t
 +a_N(S_t^N)^{q+1/2}\dd\beta_t,
\end{equation}
where
\[
 D_{p,N}(Y):=\frac1{N^2}\sum_{i,k=1}^N|Y^i-Y^k|^{p+2},\qquad
 c_N:=n\kappa^2(1-\rho^2)\Big(1-\frac1N\Big),\qquad
 a_N:=2\kappa\sqrt{\frac{1-\rho^2}{N}},
\]
and $\beta$ is a one-dimensional Brownian motion given explicitly, on the set $\{S_t^N>0\}$,
by
\begin{equation}\label{eq:intro-beta}
 \dd\beta_t
 =
 \frac{1}{\sqrt{N S_t^N}}
 \sum_{i=1}^N Y_t^i\cdot\dd B_t^i.
\end{equation}
Note that we can define $\beta$ on the whole lifetime by using the integrand in \eqref{eq:intro-beta} on the set $\{S_t^N>0\}$, setting it to zero on the set $\{S_t^N=0\}$, and adding $\one_{\{S_t^N=0\}}\dd W_t$ for an independent one-dimensional Brownian motion $W$. The resulting continuous local martingale has quadratic variation $t$ and preserves \eqref{eq:intro-variance}, including when the system leaves consensus after a visit.  When $\rho=1$, the martingale coefficient $a_N$ vanishes and the choice of $\beta$ is immaterial.

For $p>0$, \eqref{eq:intro-variance} is generally not closed because the normalized dissipation $D_{p,N}(Y)/(S^N)^{(p+2)/2}$ depends on the shape of the particle cloud. To make this terminology precise, we regard the centered configuration $Y=(Y^1,\ldots,Y^N)$ as a vector in the zero-sum subspace of $(\R^n)^N$.  Its Euclidean radius is
\[ R_t^N:=|Y_t|=\Bigl(\sum_{i=1}^N|Y_t^i|^2\Bigr)^{1/2}=\sqrt{N S_t^N},
\]
and, whenever $R_t^N>0$, its shape is the normalized centered configuration $U_t:=Y_t/R_t^N$, which lies on the unit sphere of that subspace. Thus $S_t^N=(R_t^N)^2/N$ records the scale, while $U_t$ records the configuration, modulo translation and dilation. The balance for critical homogeneity occurs at $q=(p+2)/4$.
At this exponent, radius and shape evolve on the same intrinsic time scale.  After the corresponding time change, the normalized centered configuration becomes an autonomous diffusion on the relative sphere. For $\rho<1$, the sphere has dimension $m_N-1=n(N-1)-1$, so the angular diffusion is elliptic and reversible with an explicit invariant Gibbs measure when $m_N\ge2$. The radius of the normalized sphere is always one. If $m_N=1$, it consists of two points and has no nontrivial angular diffusion.  The corresponding stationary average of the normalized dissipation determines the radial Lyapunov exponent.

This observation yields our first main result,
\autoref{thm:critical-shape-gibbs}: for $p>0$, $\rho<1$, and positive
initial variance at critical homogeneity, the sign of the averaged
radial drift determines asymptotic consensus, global recurrence across
variance scales, or almost-sure finite-time explosion.
The challenge is the boundary case of zero averaged drift. This requires more than ergodic
averaging, which gives only sublinear growth of the logarithmic radius.
We solve a centered Poisson equation for the angular generator and
write the logarithmic radius as a continuous martingale plus a bounded
corrector. Its quadratic variation is comparable to intrinsic time.
Brownian recurrence and occupation estimates then give both oscillation
of the logarithmic radius and divergence of the inverse clock.
Consequently, the original solution is global and
\[
 \liminf_{t\to\infty}S_t^N=0,\qquad
 \limsup_{t\to\infty}S_t^N=\infty.
\]
Above critical homogeneity, a reciprocal-radius transform and
one-dimensional boundary tests identify $m_N=2$ as the threshold for
accessibility of explosion and give two-sided incomplete-Gamma bounds
for its probability.

The same normalization connects the particle configuration geometry with the mean-field
equation. For a probability law $\nu$ on $\R^n$, define its
$p$-dissipation by
\[
 D_p(\nu):=\iint_{\R^n\times\R^n}|x-y|^{p+2}\,\nu(\dd x)\nu(\dd y).
\]
The normalized particle dissipation equals $D_p$ evaluated at the
centered unit-variance empirical shape law.
\autoref{prop:gibbs-empirical-bridge} shows that its minimum over
$N$-particle shapes converges to the infimum of $D_p$ over centered
unit-variance probability laws.
At critical homogeneity, \autoref{thm:critical-mf-shape} gives an
autonomous evolution for the normalized mean-field law $\tilde{\nu}_\tau$,
with
\[
 \frac{\dd}{\dd\tau}\log s_{t(\tau)}
 =c-D_p(\tilde{\nu}_\tau),\qquad c=n\kappa^2(1-\rho^2),
\]
where $t(\tau)$ is the inverse intrinsic clock.
For sufficiently regular solutions, the normalized flow satisfies the
constrained free-energy dissipation identity
\eqref{eq:mf-critical-free-energy-dissipation}.
These results reduce the nonlinear critical mean-field problem to the
long-time behavior of this shape flow. They establish convergence of
the sharp dissipation minima. Convergence of the Gibbs-averaged
thresholds and a complete classification of the normalized mean-field
dynamics remain open. The finite-size coercive comparison is given in
\autoref{cor:critical-finite-stabilization}.

The well-posedness results in \autoref{thm:particle-allq} and
\autoref{thm:MV-allq} provide the setting for the particle approximation,
including the degenerate boundary at consensus and the possibility of
explosion. The first difficulty is the regularity of the diffusion
coefficient as a function of the law. We use the centered $L^2$ radius,
which is $1$-Lipschitz in the 2-Wasserstein distance, $\cW_2$. In terms of this radius
$r=\sqrt{s}$, the amplitude is $\kappa r^{2q}$ and is locally Lipschitz
at zero for $q\ge1/2$. For $p>0$, $\rho<1$, and $0<q<1/2$, moment generation and interpolation
show that diffusion dominates dissipation at sufficiently small
positive variance. This yields a positive lower bound for solutions
started with positive variance and permits localization away from
consensus. The linear case follows from its closed variance equation.

The second difficulty is to prove propagation of chaos under the
natural $(p+2)$-moment assumption. We construct the interacting particles
together with copies of the McKean--Vlasov solution that are independent
conditionally on the common noise. For each index $i$, the particle
$X^{i,N}$ and its reference copy $\hX^i$ have the same initial value
and are driven by the same common Brownian motion $B^0$ and the same
idiosyncratic Brownian motion $B^i$. This synchronous coupling reduces
the comparison to estimating their averaged squared separation
$N^{-1}\sum_{i=1}^N|X_t^{i,N}-\hX_t^i|^2$ up to suitable stopping
times. For $p>0$, however, $|K(z)|^2=|z|^{2p+2}$, so a direct
second-moment estimate of the empirical force fluctuation would require
the stronger $(2p+2)$-moment assumption. We retain the non-positive interaction term supplied by monotonicity and
truncate only the empirical force fluctuation.
The bounded part admits the usual sampling estimate, while H\"older's
inequality and the dissipative moment bound control its tail.
Convergence of the independent reference radius then removes the
stopping: with high probability, exit of the interacting radius from
the localization window forces a non-vanishing coupling error. \autoref{thm:poc} consequently gives qualitative conditional propagation
of chaos on intervals where the limiting solution is unique and
nonexplosive, without assuming global existence of every finite system. Stronger moments yield $N^{-1}$ mean-square coupling estimates on
compact radius windows. For $q\ge1/2$, two additional moments give an
$O(N^{-1})$ exit probability before a fixed nonexplosive mean-field
horizon. These quantitative statements are given in
\autoref{thm:poc-quant} and
Propositions~\ref{prop:poc-local-rate}--\ref{prop:poc-exit-rate}.

Finally, convergence on fixed time intervals does not determine the
interaction between the large-population and long-time limits.
We resolve this question explicitly for $p=0$ in \autoref{sec:p0},
where the variance is a closed one-dimensional diffusion.
Scale and speed calculations give invariant laws and explosion
probabilities, and their Gamma representations describe the
large-$N$ transition. At $q=1/2$ and the mean-field critical noise strength, an $N^{-1}$
stabilizing correction produces a relaxation time of order $N$ and we show that 
the large-$N$ and long-time limits do not commute.
For $q>1/2$, the exact explosion probability has a Gaussian transition
window of width $N^{-1/2}$ in the initial variance around the
deterministic mean-field threshold.

The principal regimes for $\rho<1$ and positive initial variance are summarized in \autoref{tab:phase-summary}. \autoref{fig:homogeneity-diagram} separates the two homogeneity boundaries from the noise-strength and initial-data criteria within each regime.

\begin{table}[ht]
\centering
\footnotesize
\renewcommand{\arraystretch}{1.08}
\begin{tabular}{@{}p{0.18\textwidth}p{0.75\textwidth}@{}}
\toprule
Regime & Main behavior \\
\midrule
$0\le q<1/2$ & Particle solutions are nonexplosive. At consensus we use absorbing continuation for $0<q<1/2$, consensus is not absorbing for $q=0$. Positive-variance mean-field solutions are global and uniformly separated from zero. \\
$1/2\le q<(p+2)/4$ & Attraction controls large variance. Both the finite-particle and mean-field solutions are global. \\
$q=(p+2)/4$, $p>0$ & The normalized shape is autonomous.  The finite system has an exact consensus/scale-recurrence/explosion trichotomy. The mean-field problem reduces to an autonomous unit-variance shape flow. \\
$q=1/2$, $p=0$ &
The finite-particle variance is geometric Brownian motion.
Its logarithmic drift determines consensus, growth, or recurrence
across variance scales. \\
$q>(p+2)/4$ & For finite $N$, explosion is inaccessible when $n(N-1)\le2$ and has positive probability from every positive-variance state when $n(N-1)>2$.  \\
\bottomrule
\end{tabular}
\caption{Phase regimes for $\rho<1$ and positive initial variance.}
\label{tab:phase-summary}
\end{table}

\begin{figure}[ht]
\centering
\begin{minipage}[t]{0.49\textwidth}
\centering
\textbf{Finite $N$}\par\smallskip
\begin{tikzpicture}[x=0.92cm,y=2.25cm,draw=black,text=black,font=\scriptsize]
\fill[blue!5] (0,0) rectangle (4,0.5);
\fill[blue!12] (0,0.5)--(4,0.5)--(4,1.5)--cycle;
\fill[blue!3] (0,0.5)--(4,1.5)--(4,1.9)--(0,1.9)--cycle;
\draw[->] (0,0)--(4.3,0) node[right] {$p$};
\draw[->] (0,0)--(0,2.03) node[above] {$q$};
\draw[dashed] (0,0.5)--(4,0.5);
\draw[very thick] (0,0.5)--(4,1.5)
 node[pos=0.58,above,sloped] {$q=(p+2)/4$};
\draw (-0.06,0.5)--(0.06,0.5) node[left=3pt] {$1/2$};
\node[below left] at (0,0) {$0$};
\node[align=center,text width=3.3cm] at (2,0.23)
 {Nonexplosive\\For $q>0$, consensus\\may require selection};
\node[align=center] at (2.75,0.70) {Global};
\node[align=center,text width=3.3cm] at (2,1.64)
 {Explosion accessible\\exactly when $m_N>2$};
\fill (0,0.5) circle (1.3pt);
\end{tikzpicture}\par\smallskip
\begin{minipage}{0.94\linewidth}\footnotesize
On the critical line with $p>0$, the sign of $\Gamma_N$ gives consensus, scale recurrence, or explosion.
\end{minipage}
\end{minipage}\hfill
\begin{minipage}[t]{0.49\textwidth}
\centering
\textbf{Mean field}\par\smallskip
\begin{tikzpicture}[x=0.92cm,y=2.25cm,draw=black,text=black,font=\scriptsize]
\fill[blue!5] (0,0) rectangle (4,0.5);
\fill[blue!12] (0,0.5)--(4,0.5)--(4,1.5)--cycle;
\fill[blue!3] (0,0.5)--(4,1.5)--(4,1.9)--(0,1.9)--cycle;
\draw[->] (0,0)--(4.3,0) node[right] {$p$};
\draw[->] (0,0)--(0,2.03) node[above] {$q$};
\draw[dashed] (0,0.5)--(4,0.5);
\draw[very thick] (0,0.5)--(4,1.5)
 node[pos=0.58,above,sloped] {$q=(p+2)/4$};
\draw (-0.06,0.5)--(0.06,0.5) node[left=3pt] {$1/2$};
\node[below left] at (0,0) {$0$};
\node[align=center,text width=3.3cm] at (2,0.23)
 {Global, uniform\\separation from zero};
\node[align=center] at (2.75,0.70) {Global};
\node[align=center,text width=3.3cm] at (2,1.64)
 {Local theory,\\sufficient global barrier};
\fill (0,0.5) circle (1.3pt);
\end{tikzpicture}\par\smallskip
\begin{minipage}{0.94\linewidth}\footnotesize
On the critical line with $p>0$, shape dynamics and conditional energy dissipation are available, full long-time classification is open.
\end{minipage}
\end{minipage}
\caption{Schematic $(p,q)$ regimes for $\rho<1$ and positive initial variance. The solid line is critical homogeneity and the dashed line is the Lipschitz boundary $q=1/2$. At their intersection $p=0$, the complete linear theory applies. The diagram records the scope of the results. Further noise-strength and initial-data conditions determine behavior within the critical and supercritical regimes. At $q=0$ the noise is constant and consensus is not absorbing.}
\label{fig:homogeneity-diagram}
\end{figure}
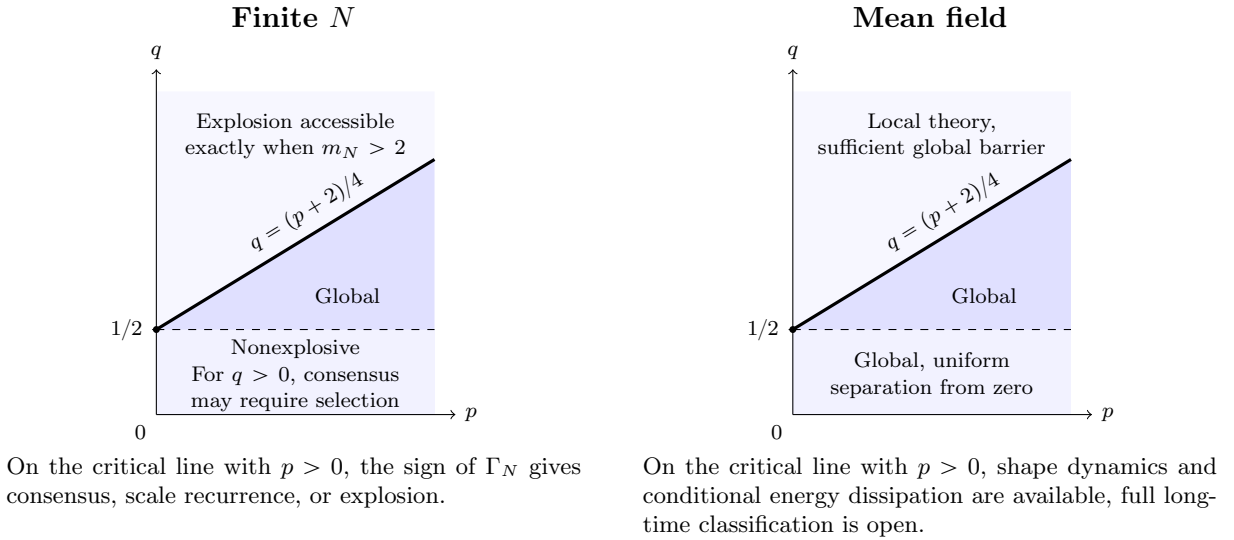
\subsection{Literature overview}
The natural comparison is with first-order consensus and opinion dynamics.  Deterministic first-order models include the classical averaging mechanism of DeGroot and continuous opinion systems with state-dependent interactions, see \cite{DeGroot,NguyenShvydkoy2023}.  Stochastic first-order consensus models have been studied with additive communication noise, multiplicative measurement noise, and noise intensities depending on relative states. Representative works include \cite{LiWuZhang2014,ZongLiZhang2018}.  In noisy Hegselmann--Krause-type models, prescribed diffusion can produce order--disorder transitions and affect the stability of consensus, see \cite{ChazelleJiuLiWang2017,WangLiEChazelle2017}.  Particularly close to the present setting is the continuous-time model of Chen, Nikolaev, and Pr\"omel \cite{ChenNikolaevPromel2025}, which contains both idiosyncratic and environmental noise and develops the corresponding conditional McKean--Vlasov theory.  The distinction here is that the strength of both noise components is the nonlinear global statistic $\kappa(S_t^N)^q$.  It is therefore neither a fixed diffusion amplitude nor a local pairwise measurement intensity.  This global feedback couples the stochastic forcing directly to the collective radius and creates boundary behavior at consensus, finite-size stabilization, logarithmic-scale recurrence, and finite-time explosion.

Second-order alignment systems form a related but different model class.  In the classical Cucker--Smale framework one evolves positions and velocities simultaneously, with velocity alignment modulated by spatial separation.  Its deterministic, kinetic, and mean-field theory includes \cite{CS2007a,CS2007b,HT2008,MT2011,MT2014,Sbook}.  Stochastic second-order extensions have treated multiplicative perturbations, random communication, common and individual noises, and stochastic mean-field limits, see \cite{ChoiSalem2019,FriesenKutoviy2020,HaJungRockner2021,HuangLvWang2026,HuangZhang2020,Rosello2019}.  Recent work also develops nonlinear velocity alignment and quantitative propagation-of-chaos questions in this second-order setting \cite{CFP2025,NS2022,NguyenShvydkoyTan2026}.  These results provide important neighboring techniques and comparisons for stochastic alignment.

The mean-field description of \eqref{eq:particle} belongs to the McKean--Vlasov theory originating with McKean \cite{McKean1966}. Classical references for propagation-of-chaos and the mean-field include \cite{BCC2011,Golse2016chap,JabinMFreview,Meleard1996,Sznitman}.  Because the particle system contains a common Brownian motion, the limiting law is naturally conditional on the common noise and the limiting evolution is, before centering, a stochastic nonlinear Fokker--Planck equation. Conditional McKean--Vlasov equations and common-noise particle limits have been developed in \cite{CoghiFlandoli2016,CoghiGess2019,HamblyLedger2017,HamblySojmark2019, HammersleySiskaSzpruch2021,LackerShkolnikovZhang2023,Nikolaev2025}, see also \cite{HuangWang2019,HuangWang2022}.  Quantitative propagation of chaos for mean-field diffusions has recently been sharpened in several directions, including entropy and uniform-in-time estimates \cite{Lacker2023,LackerLeFlem2023}.  In the present model the common-noise structure simplifies after centering, but the variance-dependent diffusion remains singular or superlinear in precisely the regimes where standard Lipschitz mean-field theory is least informative.

There is also a useful contrast with phase transitions for weakly interacting diffusions and McKean--Vlasov equations.  In classical examples the transition is typically expressed through bifurcation of stationary states or minimizers of a free-energy functional, see \cite{CarrilloGvalaniPavliotisSchlichting2020,Dawson1983,DelgadinoGvalaniPavliotis2021,DesaiZwanzig1978,GomesKalliadasisPavliotisYatsyshin2019,Tugaut2014}.  Here the transition mechanism is dynamical and is generated by the variance feedback in the noise.  The relevant questions are whether consensus is accessible, whether the finite-size It\^o correction stabilizes the radius, whether the variance repeatedly visits arbitrarily small and large scales, and whether the particle cloud explodes in finite time.

Related forms of state-dependent noise occur in consensus-based
optimization and sampling. In consensus-based optimization, typically $p=0$ and the noise
amplitude depends on a particle's displacement from a weighted ensemble
mean \cite{PinnauTotzeckTseMartin2017,CarrilloChoiTotzeckTse2018}.
Consensus-based sampling instead uses the weighted ensemble covariance
in its diffusion coefficient \cite{CarrilloHoffmannStuartVaes2022},
and thus already incorporates a global dispersion statistic.
The present model combines the isotropic scalar amplitude
$\kappa(S_t^N)^q$ with a homogeneous pairwise attraction.
Varying $q$ allows us to classify how the competition between these
terms changes the accessibility of consensus and explosion.
Another related construction appears in volatility-stabilized market
models \cite{FernholzKaratzas2005,Pal2011,Shkolnikov2013}.
Their analysis relates independent squared Bessel processes through a
common time change involving total market capitalization, and identifies
the normalized market weights as a Wright--Fisher diffusion
\cite{Pal2011}.
Here normalization uses the centered Euclidean radius, and at critical
homogeneity the resulting shape diffusion evolves on the relative sphere
with a Gibbs law determined by the interaction potential.

\subsection{Outline}
The rest of the paper is organized as follows.  \autoref{sec:particle-identities} develops the center--shape decomposition and coercive identities.  \autoref{sec:allq-particle} and \autoref{sec:mf} treat the finite-particle and mean-field dynamics, respectively.  \autoref{sec:poc} proves propagation of chaos, \autoref{sec:long-time} studies long-time behavior and \autoref{sec:p0} gives the explicit linear theory.

\section{Finite-particle structure and dissipation}\label{sec:particle-identities}

Define
\begin{equation*}
 \barX_t^N:=\frac1N\sum_{i=1}^N X_t^i,
 \qquad Y_t^i:=X_t^i-\barX_t^N,
 \qquad \barB_t^N:=\frac1N\sum_{i=1}^N B_t^i.
\end{equation*}
Then, the relative dynamics $Y$ satisfy $\sum_{i=1}^N Y_t^i=0$, and the empirical variance is given by
\begin{equation*}
 S_t^N=\frac1N\sum_{i=1}^N|Y_t^i|^2.
\end{equation*}
We will write $S_t^N(u) = \frac1N\sum_{i=1}^N|u_t^i|^2$ for any vector $u\in (\R^n)^N $  with $\sum_i^N u^i=0$, so that $S_t^N=S_t^N(Y)$.
The dynamics of the system can be decomposed into that for the center of mass $\bar{X}$ and the relative dynamics for $Y$, capturing the evolution of the shape of the system about $\bar{X}$. 
This translation-invariant decomposition underlies the finite-particle analysis.

\begin{proposition}\label{prop:decomposition}  
Up to the lifetime of a solution of \eqref{eq:particle},
\begin{align}
 \dd\barX_t^N
 &=\psi(S_t^N)\left(\rho\,\dd B_t^0+\sqrt{1-\rho^2}\,\dd\barB_t^N\right),\label{eq:centerSDE}\\
 \dd Y_t^i
 &=\frac1N\sum_{k=1}^N K(Y_t^k-Y_t^i)\,\dd t
 +\sqrt{1-\rho^2}\,\psi(S_t^N)
 \left(\dd B_t^i-\dd\barB_t^N\right).
 \label{eq:relativeSDE}
\end{align}
Hence the relative dynamics, and therefore the variance equation, contain no $B^0$ term.
\end{proposition}

\begin{proof}
Summing \eqref{eq:particle} over $i$, the interaction term vanishes because
\[
 \sum_{i,k=1}^N K(X^k-X^i)=0
\]
by oddness of $K$.  This proves \eqref{eq:centerSDE}.  Subtracting \eqref{eq:centerSDE} from the equation for $X^i$ gives \eqref{eq:relativeSDE}.  Translation invariance of pairwise differences gives $K(X^k-X^i)=K(Y^k-Y^i)$.  The final assertion follows directly from \eqref{eq:relativeSDE}. 
\end{proof}
As $K(z)=|z|^pz$, the relative dynamics yield a scalar semimartingale for the empirical variance.  

\begin{proposition}\label{prop:variance-general}
The empirical variance satisfies
\begin{equation}\label{eq:SNpower}
 \dd S_t^N
 =\left[-D_{p,N}(Y_t)+c_N(S_t^N)^{2q}\right]\dd t+\dd M_t^N,
\end{equation}
where 
\begin{equation}\label{eq:MN}
 M_t^N :=\frac{2\kappa\sqrt{1-\rho^2}}N
 \sum_{i=1}^N\int_0^t (S_s^N)^q Y_s^i\cdot\dd B_s^i
\end{equation}
is a continuous local martingale with
\begin{equation}\label{eq:bracket-power}
 \dd\langle M^N\rangle_t
 =\frac{4\kappa^2(1-\rho^2)}N(S_t^N)^{2q+1}\,\dd t = a_N^2(S_t^N)^{2q+1}\,\dd t.
\end{equation}
 Moreover, for $\tau_R:=\inf\{t\ge0:S_t^N\ge R\}$, $M^N_{\cdot\wedge\tau_R}$ is square-integrable for every $R<\infty$. 
\end{proposition}

\begin{proof}
Apply It\^o's formula to $N^{-1}\sum_{i=1}^N|Y^i|^2$.  For the drift, symmetrization gives
\begin{align*}
 \frac2{N^2}\sum_{i,k=1}^NY^i\cdot K(Y^k-Y^i)
 &=\frac1{N^2}\sum_{i,k=1}^N(Y^i-Y^k)\cdot K(Y^k-Y^i)
 =-\frac1{N^2}\sum_{i,k=1}^N|Y^i-Y^k|^{p+2}.
\end{align*}
The covariance matrix of $ B^i-\barB^N$ is $(1-1/N)\Id_n$, so the normalized sum of the It\^o corrections is
\[
 n\kappa^2(1-\rho^2)\Bigl(1-\frac1N\Bigr)(S^N)^{2q}\,\dd t.
\]
The martingale term is \eqref{eq:MN}.  Since $\sum_{i=1}^NY^i=0$, all cross terms coming from the projection onto the zero-sum subspace cancel and
\begin{align*}
 \dd\langle M^N\rangle_t
 &=\frac{4\kappa^2(1-\rho^2)}{N^2}(S_t^N)^{2q}
 \sum_{i=1}^N|Y_t^i|^2\,\dd t
 =\frac{4\kappa^2(1-\rho^2)}N (S_t^N)^{2q+1}\,\dd t.
\end{align*}
As $\psi(s)=\kappa s^q$ is locally bounded, for every $t,R<\infty$, continuity of $S^N$ gives
\[
 \E\langle M^N\rangle_{t\wedge\tau_R}
 \le \frac{4\kappa^2(1-\rho^2)}N\,R\,t\,\sup_{0\le s\le R}s^{2q}<\infty.
\]
Thus $M^N_{\cdot\wedge\tau_R}$ is a square-integrable martingale.
\end{proof}

The next estimate quantifies the coercivity of the $p$-dissipation, including the finite-size and dimension corrections used later. We will write $D_{p,N}=D_{p,N}(Y)$ for convenience of notation. 

\begin{lemma}\label{lem:D-bounds}
Let $Y=(Y^1,\dots,Y^N)$ be centered, set $\bar{p}:=(p+2)/2$ and let $d_N:=\min\{n,N-1\}$. 
Then
\begin{equation}\label{eq:D-upper-lower}
 \cA_{p,n,N}(S^N)^{\bar{p}}\le D_{p,N}(Y)
 \le 2^{\bar{p}}N^{p/2}(S^N)^{\bar{p}}, 
\end{equation}
where
\begin{equation}\label{eq:AeffN}
 \cA_{p,n,N}:= 2^{\bar{p}}
 \begin{cases}
 \left(1+\frac{1}{N-1}\right)^{p/2} ,&0\le p<2,\\[1mm]
 \max\left\{\left(1+\frac{1}{N-1}\right)^{p/2},\,
 (1+d_N^{-1})^{{\bar{p}}/2}\right\},&p\ge2,
 \end{cases}
\end{equation}
The constant in the lower bound of \eqref{eq:D-upper-lower}
is sharp whenever $n\ge N-1$. Moreover, $D_{0,N}=2S^N$, and for $p=2$,
\begin{equation}\label{eq:D2N-exact-lower}
D_{2,N}\ge4\left(1+\frac1{d_N}\right)(S^N)^2.
\end{equation}
\end{lemma}

\begin{proof}
The upper bound follows from $|Y^i-Y^k|^2\le2NS^N$: 
\[
 D_{p,N}\le(2NS^N)^{p/2}\frac1{N^2}\sum_{i,k=1}^N|Y^i-Y^k|^2
 =2^{\bar{p}} N^{p/2}(S^N)^{\bar{p}}.
\]
For the lower bound we note that diagonal pairs do not contribute.  Since
\[
 \frac1{N(N-1)}\sum_{i\ne k}|Y^i-Y^k|^2
 =\frac{2N}{N-1}S^N,
\]
applying Jensen's inequality on the $N(N-1)$ off-diagonal pairs, we have for all $p\geq 0$
\begin{align*}
 D_{p,N}
 &=\frac{N-1}{N}\frac1{N(N-1)}\sum_{i\ne k}|Y^i-Y^k|^{p+2}
 \ge\frac{N-1}{N}\Big(\frac{2N}{N-1}S^N\Big)^{\bar{p}}
 =2^{\bar{p}}\Big(1+\frac{1}{N-1}\Big)^{p/2}(S^N)^{\bar{p}}.
\end{align*}
For the dimension-sensitive estimate, let $X$ be uniformly distributed on the set $\{Y^1,\dots,Y^N\}$, write $\E$ for the corresponding expectation and let $\widetilde X$ be an independent copy. We can then write $\Sigma_N:=\E[X\otimes X]$, and note that
\[\tr\Sigma_N=S^N,
 \qquad \operatorname{rank}\Sigma_N\le d_N.
\]
A direct expansion gives
\begin{equation*}
 D_{2,N}=\E|X-\widetilde X|^4
 =2\E|X|^4+2(S^N)^2+4\tr(\Sigma_N^2).
\end{equation*}
Now $\E|X|^4\ge(S^N)^2$, while Cauchy--Schwarz on the nonzero eigenvalues of $\Sigma_N$ yields
$\tr(\Sigma_N^2)\ge(S^N)^2/d_N$. Thus, \eqref{eq:D2N-exact-lower} is derived.  If $p\ge2$, then ${\bar{p}}\ge2$ and the monotonicity of $L^a$ norms applied to $R=|X-\widetilde X|^2$ gives
\[
 D_{p,N}=\E R^{\bar{p}}\ge(\E R^2)^{{\bar{p}}/2}
 \ge\Big[4\Big(1+\frac1{d_N}\Big)(S^N)^2\Big]^{{\bar{p}}/2},
\]
establishing the general lower bound.  
If $n\ge N-1$, place the $N$ points at the vertices of a centered regular simplex.  All $N(N-1)$ off-diagonal squared distances are equal, so Jensen's inequality on the off-diagonal pairs is an equality.  This proves sharpness of the constant in the lower bound of \eqref{eq:D-upper-lower} in that range. 
The remaining statements follow from these identities. 
\end{proof}

We record the limiting explicit coercivity constant, with $ \bar{p}=\frac{p+2}{2}$
\begin{equation}\label{eq:Aeff-MV}
 \cA_{p,n}:=\lim_{N\to\infty} \cA_{p,n,N}=
 \begin{cases}
 2^{\bar{p}},&0\le p<2,\\[1mm]
 2^{\bar{p}}\Big(1+\dfrac1n\Big)^{{\bar{p}}/2},&p\ge2.
 \end{cases}
\end{equation}

\begin{remark}\label{rem:nonradial-extension}
The power-law interaction has several structural properties that enter at different later stages.  Monotonicity,
\[
 (x-y)\cdot\bigl(K(x)-K(y)\bigr)\ge0,
\]
gives the favorable sign in the synchronous coupling of \autoref{sec:poc}.  Coercivity of $z\cdot K(z)$ by multiples of $|z|^{p+2}$ turns the variance identity into the scalar upper and lower comparisons used in the phase diagram.  Exact homogeneity of degree $p+1$ makes the normalized dissipation scale-free and permits the critical intrinsic time change.

These observations extend beyond radial interactions.  For example, suppose $K=\nabla W$, where $W$ is even, $C^2$ away from the origin, positively homogeneous of degree $p+2$, and there are constants $0<c_{\pm}$ such that
\[
 c_-|z|^{p+2}\le z\cdot K(z)\le c_+|z|^{p+2}.
\]
Then the relative interaction remains the gradient of a homogeneous potential, the critical time change again produces an autonomous diffusion on the relative sphere, and its reversible density is the corresponding Gibbs weight.  What is lost is the explicit power-law formula for the shape functional and hence the computable constants in \autoref{lem:D-bounds}.  Without homogeneity the angular dynamics are no longer autonomous after a single radial time change. Without gradient structure there is generally no explicit reversible Gibbs measure. And without coercivity scalar variance comparison need not determine a phase boundary.  Thus radiality itself is not essential, while homogeneity, gradient structure, and quantitative coercivity play distinct roles in the sharp results below.
\end{remark}

For later use, it is convenient to separate the exact normalized shape extrema from the explicit bounds in \autoref{lem:D-bounds}.  Let
\begin{equation}\label{def:HN}
  H_N:=\Big\{y=(y^1,\ldots,y^N)\in(\R^n)^N:\ \sum_{i=1}^Ny^i=0\Big\}   
\end{equation}
and let $\S(H_N)$ be the unit sphere in $H_N$.  Since $S^N(u)=1/N$ for $u\in\S(H_N)$, define
\begin{equation}\label{eq:shape-extrema-global}
 \Theta_N(u):=\frac{D_{p,N}(u)}{(S^N(u))^{\bar{p}}}=N^{\bar{p}}D_{p,N}(u),
 \qquad
 \cA_{p,n,N}^{\sharp}:=\min_{\S(H_N)}\Theta_N,
 \qquad
 \cB_{p,n,N}^{\sharp}:=\max_{\S(H_N)}\Theta_N.
\end{equation}
Both extrema are attained.  With
\begin{equation}\label{eq:BpN}
 \mathcal{B}_{p,N}:=2^{\bar{p}}N^{p/2},
\end{equation}
\autoref{lem:D-bounds} gives the exact hierarchy
\begin{equation}\label{eq:shape-extrema-explicit-bracket}
 \cA_{p,n,N}\le\cA_{p,n,N}^{\sharp}
 \le\cB_{p,n,N}^{\sharp}\le \mathcal{B}_{p,N}.
\end{equation}
The inner two constants are the optimal uniform lower and upper bounds for the normalized dissipation. The outer two constants are explicit.

\section{Finite-particle system: well-posedness and explosion}\label{sec:allq-particle}

Two independent mechanisms enter the well-posedness analysis.  The first is \emph{escape to infinity}, controlled by the competition in \eqref{eq:SNpower} between the attractive term of order $(S^N)^{\bar{p}}$ and the It\^o correction of order $(S^N)^{2q}$.  The second is \emph{degeneracy at consensus}: when $0<q<1/2$, the map
\[
 Y\longmapsto (S^N(Y))^q
\]
is not Lipschitz but only H\"older continuous at $Y=0$.  The two cases are treated separately below.

Set $ m_N:=n(N-1)$, $c:=n\kappa^2(1-\rho^2)$ and recall $a_N:=2\kappa\sqrt{\frac{1-\rho^2}{N}}$.
Thus $m_N$ is the dimension of the zero-sum relative configuration space and noting $c_N=c(1-1/N)$,
\begin{equation}\label{eq:cn-an-relation}
 c_N=\frac{m_N}{4}a_N^2.
\end{equation}
Recall also $\bar p =(p+2)/2$, and the normalized shape extrema and the explicit upper bound from \eqref{eq:shape-extrema-global}--\eqref{eq:BpN}.  Then
$\cA_{p,n,N}(S^N)^{\bar{p}}\le D_{p,N}\le \cB_{p,N}(S^N)^{\bar{p}}$.

We now state our main result on the well-posedness of the finite particle system.

\begin{theorem}\label{thm:particle-allq}
Assume $K(z)=|z|^pz$ for $z\in\R^n$, $p\ge0$, $\psi(s)=\kappa s^q$ for $s\ge0$, $q\ge0$, and that the initial data is deterministic.

\begin{enumerate}[label=\textnormal{(\roman*)}]
\item\label{item:particle-common} \textnormal{(Pure common noise: $\rho=1$.)} For every $q\ge0$ and every initial configuration, \eqref{eq:particle} has a unique global strong solution.  The relative configuration is deterministic and independent of $q$.

\item\label{item:particle-constant} \textnormal{(Constant 
noise: $q=0$, $\rho<1$.)}  For every initial configuration, \eqref{eq:particle} has a unique global strong solution.

\item\label{item:particle-holder} \textnormal{(Sub-Lipschitz noise away from consensus: $0<q<1/2$, $\rho<1$, $S_0^N>0$.)}
If $m_N\ge2$, the strong solution is unique and global, and
\begin{equation}\label{eq:positive-variance-finiteN}
 \P\left(S_t^N>0\text{ for all }t\ge0\right)=1.
\end{equation}
If $m_N=1$ (equivalently, $n=1$ and $N=2$), consensus occurs in finite time almost surely.  For $q\ge1/4$, pathwise uniqueness holds and the solution is absorbed at consensus after this hitting time.  For $0<q<1/4$, uniqueness holds only up to the hitting time and there are infinitely many weak continuations afterwards unless a boundary selection is imposed.

\item\label{item:particle-lipschitz} \textnormal{(Locally Lipschitz noise: $q\ge1/2$, $\rho<1$.)}
For every initial configuration, there is a unique maximal strong solution up to explosion.  It is global if
\begin{equation*}
 2q<\bar{p} ,\
 \text{ or }
 2q=\bar{p} \text{ and either } p=0\ \text{or}\
 \left(p>0\ \text{ and }\ c_N\le \cA_{p,n,N}^{\sharp}+\frac12a_N^2\right).
\end{equation*}
If $2q>\bar{p} $ and $S_0^N>0$, it is global almost surely for $m_N\le2$ and explodes in finite time with positive probability for $m_N>2$.

\item\label{item:particle-consensus} \textnormal{(Consensus initial data: $0<q<1/2$, $\rho<1$, $S_0^N=0$.)}
The constant configuration is absorbing.  If $m_N\ge2$, uniqueness in law fails.  If $m_N=1$, uniqueness fails for $0<q<1/4$, while pathwise uniqueness holds for $1/4\le q<1/2$.
\end{enumerate}
\end{theorem}

\begin{proof}
Although mostly standard, we give a full proof for completeness. 

For part~\ref{item:particle-common}, \autoref{prop:decomposition} gives
\[
 \dY_t^i=\frac1N\sum_{k=1}^N K(Y_t^k-Y_t^i)\dt,
\]
so the relative system is a deterministic ODE with a locally Lipschitz vector field.  Its squared radius is $NS_t^N$, and \autoref{prop:variance-general} reduces to
\[
 \frac{\dd}{\dd t}S_t^N=-D_{p,N}(Y_t)\le0.
\]
Thus the relative configuration cannot blow up.  Once $Y$ is known, the center is the stochastic integral
\[
 \barX_t^N=\barX_0^N+\kappa\int_0^t(S_s^N)^q\dB_s^0,
\]
with a continuous deterministic integrand.  This proves global strong existence and uniqueness for every $q\ge0$.

For part~\ref{item:particle-constant}, the diffusion coefficient is constant.  The drift is locally Lipschitz and the variance equation satisfies, by \autoref{lem:D-bounds},
\[
 \dd S_t^N\le\left[-\cA_{p,n,N}(S_t^N)^{\bar{p}}+c_N\right]\dt+\dd M_t^N.
\]
Stopping at $\tau_R: =\inf\{t \,|\, S_t^N\geq R\}$ and taking expectations in the variance equation gives
\[
\E[S^N_{t\wedge\tau_R}] +\cA_{p,n,N}\E\int_{0}^{t\wedge\tau_R}(S_s^N)^{\bar{p}}\dd s \leq S_0^N + c_N t,
\]
since the martingale term vanishes and $(S^N)^{2q} = 1$ for $q=0$. In particular,
\[
\E[S^N_{t\wedge\tau_R}] \leq S_0^N + c_N t,
\]
which is independent of $R$. Since $S^N_{T\wedge\tau_R}=R$ on $\{\tau_R\le T\}$,
\[
 \P(\tau_R\le T)\le\frac{S_0^N+c_NT}{R}\rightarrow0 \quad \text{ as } R\to \infty.
\]
This rules out radial explosion. By Fatou's lemma we obtain
$\sup_{t\le T}\E S_t^N\le S_0^N+c_NT$ for every finite $T$.
The center has quadratic variation bounded by a constant multiple of $t$.  Since
\[
 \frac1N\sum_{i=1}^N|X_t^i|^2=|\barX_t^N|^2+S_t^N,
\]
the standard stopping-time criterion rules out explosion (see, for example \cite[Section 5.5]{KaratzasShreve1991}).  Local Lipschitz continuity then yields global pathwise uniqueness.

We next prove part~\ref{item:particle-holder}. On the open set $\{S^N>0\}$ the coefficient $Y\mapsto(S^N(Y))^q$ is smooth, so there is a unique strong solution up to the first hitting time of consensus or explosion.  Because $2q<1\le \bar{p} $, for any $a>0$ Young's inequality gives 
\[
c_N s^{2q}\leq a s^{\bar{p}} + C_{a}, \quad \text{ with } C_{a}:= c_N^{\bar{p} /(\bar{p} -2q)}a^{-2q/(\bar{p} -2q)}C_{p,q}.
\]
Choosing $a = \cA_{p,n,N}/2, C=C_a$, we have
\[
 -\cA_{p,n,N}s^{\bar{p}}+c_Ns^{2q}\le -a s^{\bar{p}}+C.
\]
Then we use the same stopping time argument for the variance estimate as for part ~\ref{item:particle-constant}  to conclude that explosion is impossible on every finite time interval. 

We next prove that consensus is not reached when $m_N\ge2$.
Fix $R>S_0^N$ and $0<\e<S_0^N$, and set
\[
 \tau_R:=\inf\{t\ge0:S_t^N\ge R\},\qquad
 \tau_\e:=\inf\{t\ge0:S_t^N\le\e\}.
\] The upper estimate in \autoref{lem:D-bounds} gives
\begin{equation}\label{eq:D-upper-hit}
 D_{p,N}(Y)\le 2^{\bar{p}}N^{p/2}(S^N)^{\bar{p}}.
\end{equation}
First suppose $m_N>2$.  Choose a constant $\eta$ such that
\[
 0<\eta<\frac{m_N}{2}-1
\]
and apply It\^o's formula to $f(s)=s^{-\eta}$ for $ s\in [\e, R]$.  From \eqref{eq:SNpower}, \eqref{eq:bracket-power}, and \eqref{eq:cn-an-relation}, the drift of $f(S^N)$ is, writing in terms of the generator $\cL$ of the diffusion,
\begin{align*}
 \cL f(S^N)
 &=\eta D_{p,N}(Y)(S^N)^{-\eta-1}
 +\eta\Big[-c_N+\frac{\eta+1}{2}a_N^2\Big](S^N)^{2q-\eta-1}.
\end{align*}
Using \eqref{eq:cn-an-relation}, the coefficient of the second term can be written explicitly as
\[
 -c_N+\frac{\eta+1}{2}a_N^2
 =\frac{a_N^2}{4}\bigl(2(\eta+1)-m_N\bigr)<0.
\]
Dropping this favorable term and using \eqref{eq:D-upper-hit}, we obtain that on $\{S^N\le R\}$,
\[
 \cL f(S^N)\le C_R\bigl(1+f(S^N)\bigr).
\]
Thus $f$ is a Lyapunov function, so taking expectations and using Dynkin's formula and Gr\"onwall's inequality give a bound, independent of $\e$, on the expectation of $(S^N)^{-\eta}$ up to the stopping time $T\wedge\tau_R\wedge\tau_{\e}$.  Consequently
\[
 \P\bigl(\tau_\e\le T,\ \tau_R>T\bigr)\le C_{T,R}\e^{\eta}.
\]
The non-explosion estimate proved at the beginning of part~\ref{item:particle-holder} gives $\P(\tau_R\le T)\to0$ as $R\to\infty$.  Hence, first letting $\e\to0$ at fixed $R$ and then $R\to\infty$ proves that zero is not hit in finite time.

When $m_N=2$, use the nonnegative function $f_R(s)=\log(R/s)$ on $(0,R]$. Applying It\^o's formula gives
\[
 \cL(-\log S^N)
 =\frac{D_{p,N}(Y)}{S^N}
 +\Big(-c_N+\frac12a_N^2\Big)(S^N)^{2q-1}.
\]
Since \eqref{eq:cn-an-relation} becomes $c_N=a_N^2/2$, the singular terms cancel exactly and
\begin{equation*}
 \cL(f_R(S^N))=\frac{D_{p,N}(Y)}{S^N}\le C_R \quad \text{ on } \{S^N\le R\}.
\end{equation*}
 Optional stopping gives 
 \[
 \log(R/\e)\P(\tau_\e\le T\wedge\tau_R)\le f_R(S_0^N)+C_RT. 
 \]
 Hence, as $\e\to0$,
\[
 \P\bigl(\tau_\e\le T,\ \tau_R>T\bigr)
 \le\frac{C_{T,R}}{|\log\e|}\rightarrow 0.
\]
This proves \eqref{eq:positive-variance-finiteN} for $m_N\ge2$ and therefore global pathwise uniqueness.

If $m_N=1$, write $Z_t=X_t^1-X_t^2$.  Then
\begin{equation*}
 \dd Z_t=-|Z_t|^pZ_t\dt
 +\sigma_q|Z_t|^{2q}\dd W_t,
 \qquad
 \sigma_q:=\kappa 2^{1/2-2q}\sqrt{1-\rho^2}.
\end{equation*}
On $(0,\infty)$ a straightforward calculation gives a scale density as
\begin{equation}\label{eq:exception-scale}
 \mathfrak s'(x)=
 \exp\Big(\frac{2}{\sigma_q^2(p+2-4q)}x^{p+2-4q}\Big).
\end{equation}

As $q<1/2$, the exponent $p+2-4q$ is positive.  Hence $\mathfrak s(0)>-\infty$ and $\mathfrak s(\infty)=\infty$.  The corresponding speed density is asymptotic to $Cx^{-4q}$ at zero, and
\[
 \int_0^1(\mathfrak s(x)-\mathfrak s(0))x^{-4q}\,\dd x<\infty
\]
exactly because $q<1/2$.  Thus zero is accessible, while infinity is not reached before zero. The standard hitting formula in scale therefore gives
\begin{equation*}
 \P_z(\tau_0<\infty)=1,\qquad z\ne0.
\end{equation*}
The Yamada--Watanabe criterion \cite{YamadaWatanabe1971} gives pathwise uniqueness when the diffusion coefficient is at least $1/2$-H\"older, that is, when $2q\ge 1/2$ or $q\ge 1/4$.  Since both coefficients vanish at zero, the unique continuation is then the absorbing one.  For $0<q<1/4$, pathwise uniqueness fails at zero.  A nontrivial weak solution leaving zero is constructed by the time change used below in part~\ref{item:particle-consensus}, and concatenating it after an arbitrary deterministic waiting time yields infinitely many weak continuations.  This proves the exceptional statement in part~\ref{item:particle-holder}.

For part~\ref{item:particle-lipschitz}, $2q\ge 1$ implies that $Y\mapsto(S^N(Y))^q$ is locally Lipschitz on the whole relative space, hence local strong existence and pathwise uniqueness hold up to explosion.  If $2q< \bar{p}$, localization in \eqref{eq:SNpower} and \autoref{lem:D-bounds} gives
\[
 \E S_{t\wedge\tau_R}^N
 +a\E\int_0^{t\wedge\tau_R}(S_s^N)^{\bar{p}}\ds
 \le S_0^N+Ct
\]
for constants independent of $R$, and the usual stopping argument rules out explosion.

Consider the critical case $2q=\bar{p}$.  If $p=0$, then $\bar{p} =1$ and the relative coefficients have at most linear growth, so the maximal strong solution is global for every $c_N$.  Suppose $p>0$ and put $\gamma=p/2=\bar{p} -1$.  On the set $S^N>0$, the martingale representation theorem and \eqref{eq:bracket-power} give a one-dimensional Brownian motion $\beta$ such that
\[
 \dd S_t^N=\bigl[-\theta_t+c_N\bigr](S_t^N)^{\bar{p}}\dt
 +a_N(S_t^N)^{(\bar{p} +1)/2}\dd\beta_t,
 \qquad
 \theta_t:=\frac{D_{p,N}(Y_t)}{(S_t^N)^{\bar{p}}}\ge\cA_{p,n,N}^{\sharp}.
\]
For $Z_t=(S_t^N)^{-\gamma}$, It\^o's formula yields, up to the first hitting time of zero or infinity,
\begin{equation*}
 \dd Z_t
 =\gamma\left(\theta_t-c_N+\frac{\bar{p}}{2}a_N^2\right)\dt
 -\gamma a_N\sqrt{Z_t}\,\dd\beta_t.
\end{equation*}
Let
\[
 b:=\gamma\left(\cA_{p,n,N}^{\sharp}-c_N+\frac{\bar{p}}{2} a_N^2\right).
\]
If $c_N\le\cA_{p,n,N}^{\sharp}+a_N^2/2$, then $2b\ge\gamma^2a_N^2$. We compare $Z$ with the squared-Bessel-type diffusion $U$, driven by a Brownian motion $W$,
\[
 \dd U_t=b\,\dt+\gamma a_N\sqrt{U_t}\,\dd W_t.
 \qquad U_0=Z_0,
\]
The origin is polar for $U$ under the Feller condition $2b\ge\gamma^2a_N^2$ (see, for example, \cite[Section~5.5]{KaratzasShreve1991}).  To make the comparison completely local, take $W=-\beta$ and stop at
$\tau_{\e,R}:=\inf\{t:U_t\notin(\e,R)\text{ or }Z_t\notin(\e,R)\}$.
On this interval both square-root coefficients are Lipschitz.  Since the drift of $Z$ is at least $b$, Tanaka's formula for $(U-Z)_+$ \cite[Chapter~3]{KaratzasShreve1991}, taking expectations, followed by a standard Gr\"onwall argument, gives $U_{t\wedge\tau_{\e,R}}\le Z_{t\wedge\tau_{\e,R}}$.  Letting first $\e\to 0$ and then $R\to\infty$ yields $U\le Z$ up to the lifetime.  (Under strict inequality the drift of $Z$ is positive, so $Z$ is locally a submartingale. No supermartingale argument is being used.)  Hence $S^N$ cannot explode.  If $S^N$ hits zero instead, the relative process is absorbed there by pathwise uniqueness, so this also gives a global solution.  The center is then nonexplosive because its quadratic variation is finite on every interval on which $S^N$ is bounded.  This proves the critical assertion.  The supercritical case $2q>\bar{p} $ is treated separately in \autoref{thm:supercritical-dimension}.

We finish with part~\ref{item:particle-consensus}.  In orthonormal coordinates on the relative subspace, the diffusion near consensus has the form
\begin{equation}\label{eq:relative-local-form}
 \dd Z_t=\tilde{b}(Z_t)\dt+\sigma_0|Z_t|^{\alpha}\dd W_t,
 \qquad \alpha=2q\in(0,1),
\end{equation}
where $\tilde{b}(Z)=O(|Z|^{p+1})$ and $\sigma_0>0$.  The identically zero path is one weak solution.  To construct a nontrivial one for $m_N\ge2$, start an $m_N$-dimensional Brownian motion $W$ from zero and set
\[
 A_u:=\sigma_0^{-2}\int_0^u|W_v|^{-2\alpha}\,\dd v.
\]
Since $m_N\ge2$, Brownian motion leaves the origin immediately and does not return to it, and $A_u<\infty$ for every sufficiently small $u>0$ because $\alpha<1$.  If $T_t$ is the right-continuous inverse of $A$, then $Z_t=W_{T_t}$ is a nontrivial weak solution of the driftless equation $\dd Z=\sigma_0|Z|^\alpha\dd W$.  On a bounded ball,
\[
 \frac{|\tilde{b}(z)|}{\sigma_0|z|^\alpha}\le C|z|^{p+1-\alpha}
\]
is bounded because $p+1-\alpha>0$. Define a stopping time to be the first time of leaving that ball or reaching a deterministic time horizon.  The integrand in Girsanov's Theorem is then bounded, so Novikov's condition 
holds for the stopped exponential martingale.  The stopped change of measure in \cite[Theorem~3.5.1]{KaratzasShreve1991} therefore inserts the drift $\tilde{b}$ and gives a nontrivial weak solution of \eqref{eq:relative-local-form}. Localization removes the auxiliary deterministic horizon.  Thus uniqueness in law fails.  In dimension one the same time-change is valid exactly when $\alpha<1/2$, i.e. $q<1/4$, while the Yamada--Watanabe criterion gives uniqueness for $q\ge1/4$.  This completes the proof.
\end{proof}

\begin{remark}\label{rem:allq-three-issues}
For $0<q<1/2$, the assertion of global uniqueness from positive variance when $m_N\ge2$ does not contradict nonuniqueness at exact consensus: the singular state is polar for a solution started away from it, but it admits more than one weak continuation when it is taken as initial data.  This distinction is important in both the finite-particle and mean-field theories.

Whenever the paper needs a single finite-particle process on any such nonunique consensus event, we select the absorbing continuation $Y_t\equiv0$ from the first consensus time onward.  This is a boundary selection rule, not a uniqueness statement.
\end{remark}

The proof also records a quantitative finite-horizon lower-tail estimate in the non-Lipschitz regime.  If $0<q<1/2$, $m_N>2$, $S_0^N>0$, and $0<\eta<m_N/2-1$, then for every $T,R<\infty$,
\begin{equation*}
 \P\left(\inf_{t\le T}S_t^N\le\e,\ \sup_{t\le T}S_t^N<R\right)
 \le C_{T,R,\eta}\e^\eta.
\end{equation*}
For $m_N=2$ the right-hand side is $C_{T,R}/|\log\e|$.  Together with the upper-radius tightness used in \autoref{thm:particle-allq}, these estimates imply that $\inf_{t\le T}S_t^N$ is bounded away from zero in probability for every fixed $T$.  They do not give, and one should not expect, a deterministic almost-sure floor of the form $cN^{-\alpha}$: the relative diffusion is nondegenerate away from consensus and can make arbitrarily deep excursions with positive probability.

The exceptional relative dimension $m_N=1$ admits a complete boundary classification.

\begin{corollary}\label{cor:two-particle}
Assume $n=1$, $N=2$, $\rho<1$, $S_0^2>0$, and $0<q<1/2$.  Then consensus occurs in finite time almost surely.  For $1/4\le q<1/2$, the unique strong solution is absorbed at consensus. For $0<q<1/4$, the solution is unique only up to the consensus time and requires a boundary selection thereafter.
\end{corollary}

At critical homogeneity we shall use polar coordinates in the relative configuration.  The following non-attainment statement makes that reduction legitimate independently of the sign of the eventual radial drift.

\begin{lemma}\label{lem:critical-nonattainment}
Assume $p>0$, $\rho<1$, $q=(p+2)/4$, and $S_0^N>0$.  Let $\tau_N^{\rm expl}$ be the lifetime of the maximal particle solution.  Then
\begin{equation}\label{eq:critical-zero-polar}
 \P\left(S_t^N>0\text{ for every }0\le t<\tau_N^{\rm expl}\right)=1.
\end{equation}
In particular, consensus cannot be reached before explosion.
\end{lemma}

\begin{proof}
Fix $T<\infty$ and $R>S_0^N$, and let
\[
 \tau_R:=\inf\{t:S_t^N\ge R\},
 \qquad
 \tau_\e:=\inf\{t:S_t^N\le\e\}.
\]
At criticality $2q-1=p/2$.  Applying It\^o's formula with the function
$f_R(s)=\log(R/s)$ for $s\in [\e,R]$ and using \eqref{eq:SNpower}--\eqref{eq:bracket-power} gives the drift
\begin{equation*}
 \cL f_R(S^N)
 =\frac{D_{p,N}}{S^N}
 +\Big(-c_N+\frac12a_N^2\Big)(S^N)^{p/2}.
\end{equation*}
By the upper bound in \autoref{lem:D-bounds},
$D_{p,N}\le \cB_{p,N}(S^N)^{1+p/2}$.  Hence on $\{S^N\le R\}$,
\[
 \cL f_R(S^N)\le
 \Big(\cB_{p,N}+\Big|-c_N+\frac12a_N^2\Big|\Big)R^{p/2}
 =:C_R.
\]
Optional stopping at $T\wedge\tau_\e\wedge\tau_R$ therefore yields
\[
 \log(R/\e)\,
 \P(\tau_\e\le T\wedge\tau_R)
 \le \log(R/S_0^N)+C_RT.
\]
Letting $\e\to0$ shows that zero is not reached before $T\wedge\tau_R$.  Since $T$ and $R$ are arbitrary, zero cannot be reached before the relative radius leaves every bounded set.  The center has zero drift and, while the relative radius is bounded, the quadratic variation is also bounded. Hence, the center cannot explode first. Therefore \eqref{eq:critical-zero-polar} follows. 
\end{proof}

At critical homogeneity, the radial and angular dynamics have the same intrinsic time scale.  This produces an exact skew-product representation which resolves the shape dependence left open by scalar comparison.

\begin{theorem}\label{thm:critical-shape-gibbs}
Assume $p>0$, $\rho<1$, $q=(p+2)/4$, and $S_0^N>0$.  Let $H_N$ be the relative space \eqref{def:HN}, $m_N=\dim H_N=n(N-1)$, and  define the homogeneous interaction potential for $y \in H_N$ by
\begin{equation*}
 \cV_{p,N}(y):=\frac1{N(p+2)}\sum_{1\le i<j\le N}|y^i-y^j|^{p+2}.
\end{equation*}
For $u\in\S(H_N)$, the normalized $p$-dissipation in \eqref{eq:shape-extrema-global} satisfies
\begin{equation*}
 \Theta_N(u)=2(p+2)N^{p/2}\cV_{p,N}(u).
\end{equation*}
For $m_N\ge2$, define $ \varsigma_N:=\kappa\sqrt{1-\rho^2}\,N^{-(p+2)/4}$ and
\begin{equation}\label{eq:Gibbs-shape-law}
 \Pi_N^{\rm sh}(\dd u)
 :=\frac1{Z_N^{\rm sh}}\exp\left(-\frac{2\cV_{p,N}(u)}{\varsigma_N^2}\right)
 \sigma_{m_N-1}(\dd u),
\end{equation}
where $\sigma_{m_N-1}$ is surface measure on $\S(H_N)$ and $Z_N^{\rm sh}$ is the normalization factor. For $m_N=1$ let $\Pi_N^{\rm sh}$ be the normalized counting measure on the two-point sphere $\S(H_N)$.  Set
\begin{equation*}
 \overline{\Theta}_N:=\int_{\S(H_N)}\Theta_N(u)\,\Pi_N^{\rm sh}(\dd u),
\end{equation*}
so that $\overline{\Theta}_N=2^{p+1}$  when $m_N=1$, and define
\begin{equation}\label{eq:critical-exact-index}
 \Gamma_N:=c_N-\frac12a_N^2-\overline{\Theta}_N.
\end{equation}
Then the sign of $\Gamma_N$ gives the exact critical trichotomy:
\begin{align}
 \Gamma_N<0
 &\quad\Rightarrow\quad
 \tau_N^{\rm expl}=\infty\ \text{ and }\ S_t^N\rightarrow 0\quad\text{a.s.},\nonumber\\
 \Gamma_N=0
 &\quad\Rightarrow\quad
 \tau_N^{\rm expl}=\infty,\qquad
 \liminf_{t\to\infty}S_t^N=0,\quad
 \limsup_{t\to\infty}S_t^N=\infty\quad\text{a.s.},
 \label{eq:critical-Gamma-zero}\\
 \Gamma_N>0
 &\quad\Rightarrow\quad
 \tau_N^{\rm expl}<\infty\quad\text{a.s.}.\nonumber
\end{align}
Moreover
\begin{equation}\label{eq:Gibbs-theta-bracket}
 \cA_{p,n,N}
 \le\cA_{p,n,N}^{\sharp}
 \le\overline{\Theta}_N
 \le\cB_{p,n,N}^{\sharp}
 \le \cB_{p,N}.
\end{equation}
\end{theorem}

\begin{proof}
Identify $H_N$ isometrically with $\R^{m_N}$. Since $S^N(y)=|y|^2/N$ and $q=(p+2)/4$, the process $Y$ in the relative space $H_N$, satisfies, up to its lifetime, the SDE
\begin{equation*}
  \dd Y_t=-\nabla\cV_{p,N}(Y_t)\,\dd t
  +\varsigma_N|Y_t|^{(p+2)/2}\,\dd W_t^{H},
\end{equation*}
where the projected idiosyncratic Brownian motion, $W^H$, is a standard Brownian motion on $H_N$.
The interaction potential $\cV_{p,N}$ is homogeneous of degree $p+2$.  For $R_t:=|Y_t|>0$ and $U_t:=Y_t/R_t$, Euler's identity in these polar variables reads, for every $R>0$ and $u\in\S(H_N)$,
\[
 \nabla\cV_{p,N}(Ru)
 =R^{p+1}\bigl((p+2)\cV_{p,N}(u)u+\nabla_{\S}\cV_{p,N}(u)\bigr).
\]
Below it is applied with $R=R_t$ and $u=U_t$.
\autoref{lem:critical-nonattainment} shows that $R_t>0$ throughout the lifetime of the maximal solution.  Hence it is valid to use polar coordinates up to explosion.

Introduce the intrinsic clock
\begin{equation*}
 A_t:=\int_0^tR_s^p\,\dd s,
 \qquad T_\tau:=\inf\{t\ge0:A_t>\tau\}.
\end{equation*}
Write $\widehat U_\tau:=U_{T_\tau}$ and $\ell_\tau:=\log R_{T_\tau}$.  It\^o's formula for the Euclidean polar decomposition, followed by the time change, gives
\begin{align}
 \dd\widehat U_\tau
 &=-\nabla_{\S}\cV_{p,N}(\widehat U_\tau)\,\dd\tau
 +\varsigma_N(I-\widehat U_\tau\otimes\widehat U_\tau)\,\dd\widehat W_\tau
 -\frac{m_N-1}{2}\varsigma_N^2\widehat U_\tau\,\dd\tau,
 \label{eq:critical-angular-SDE}\\
 \dd\ell_\tau
 &=\Big[-(p+2)\cV_{p,N}(\widehat U_\tau)
 +\frac{m_N-2}{2}\varsigma_N^2\Big]\dd\tau
 +\varsigma_N\,\dd\widehat\beta_\tau.
 \label{eq:critical-log-radius}
\end{align}
In particular, the time-changed process
$ \widehat W_\tau:=\int_0^{T_\tau}R_s^{p/2}\,\dd W_s^H$
is a standard Brownian motion on $H_N$, and  thus
$\widehat\beta_\tau:=\int_0^\tau\widehat U_s\cdot\dd\widehat W_s$, is
the one-dimensional Brownian motion driving the 
logarithmic radius $\ell=\log R$ in \eqref{eq:critical-log-radius}. Also 
$M_\tau^{\rm ang}:=\varsigma_N\int_0^\tau
 (I-\widehat U_s\otimes\widehat U_s)\,\dd\widehat W_s$
is the angular martingale in \eqref{eq:critical-angular-SDE}.  Their cross-variation vanishes because
$\widehat U_s^{\mathsf T}(I-\widehat U_s\otimes\widehat U_s)=0$.
 Equation \eqref{eq:critical-angular-SDE} is autonomous.  For $f\in C^2(\S(H_N))$, its generator is
\begin{equation}\label{eq:angular-generator}
 \cL_N^{\rm ang}f
 =-\nabla_{\S}\cV_{p,N}\cdot\nabla_{\S}f
 +\frac{\varsigma_N^2}{2}\Delta_{\S}f.
\end{equation}
For $m_N\ge2$ this is uniformly elliptic on the compact connected sphere, and the unique invariant probability is precisely the reversible Gibbs measure \eqref{eq:Gibbs-shape-law}.  The intrinsic representation extends to every $\tau\ge0$: the angular equation is a globally defined diffusion on a compact sphere (a constant sign when $m_N=1$), and the drift of $\ell$ is bounded. Thus $\ell$ cannot diverge at a finite intrinsic time. Conversely, the inverse clock has derivative $e^{-p\ell_\tau}>0$, and the reconstructed radius remains finite and positive on every compact intrinsic interval. Uniqueness of the original relative SDE identifies this reconstruction with its maximal solution. A finite endpoint of the intrinsic clock along that solution would therefore be extendible, contradicting maximality. This justifies the limits as $\tau\to\infty$ below.

Since $\widehat\beta$ in \eqref{eq:critical-log-radius} is a standard Brownian motion, $\widehat\beta_\tau/\tau\to0$ almost surely as $\tau\to\infty$.  
Combining this with the ergodic theorem for the angular diffusion yields
\begin{equation}\label{eq:critical-log-radius-LLN}
 \frac{\ell_\tau}{\tau}\rightarrow
 \gamma_N^{\rm rad}:=- (p+2)\int_{\S(H_N)}\cV_{p,N}\,\dd\Pi_N^{\rm sh}
 +\frac{m_N-2}{2}\varsigma_N^2
 \qquad\text{a.s.}
\end{equation}
For $m_N=1$, the shape has only two possible signs and $\cV_{p,N}$ is even, so the same limit holds with the constant shape value.

The relationship with \eqref{eq:critical-exact-index} is exact.  Since
\[
 c_N-\frac12a_N^2=\frac{m_N-2}{N}\kappa^2(1-\rho^2),
\]
and $N^{-p/2}\Theta_N(u)=2(p+2)\cV_{p,N}(u)$, we have
\begin{equation}\label{eq:lambda-Gamma-relation}
 2\gamma_N^{\rm rad}=N^{-p/2}\Gamma_N.
\end{equation}
Finally, the inverse time change satisfies
\begin{equation}\label{eq:critical-physical-time}
 T_\tau=\int_0^\tau e^{-p\ell_s}\,\dd s.
\end{equation}
If $\Gamma_N>0$, then $\gamma_N^{\rm rad}>0$ and \eqref{eq:critical-log-radius-LLN} implies $\ell_\tau\ge(\gamma_N^{\rm rad}/2)\tau$ for all sufficiently large $\tau$.  Hence \eqref{eq:critical-physical-time} converges as $\tau\to\infty$, while $R_{T_\tau}\to\infty$.  The original process therefore explodes in finite time almost surely.  If $\Gamma_N<0$, then $\gamma_N^{\rm rad}<0$, so $T_\tau\to\infty$ and $R_{T_\tau}\to0$.  Thus the original solution is global and $S_t^N=R_t^2/N\to0$ almost surely.

It remains to treat the boundary case $\Gamma_N=0$.  First suppose $m_N\ge2$ and define
\[
 g_N(u):=-(p+2)\cV_{p,N}(u)+\frac{m_N-2}{2}\varsigma_N^2.
\]
By \eqref{eq:critical-log-radius-LLN} and \eqref{eq:lambda-Gamma-relation},
$\int_{\S(H_N)} g_N\,\dd\Pi_N^{\rm sh}=0$.  The generator $\cL_N^{\rm ang}$ is uniformly elliptic on the compact connected sphere.  Hence the centered Poisson equation
\begin{equation*}
 \cL_N^{\rm ang}\phi_N=-g_N,
 \qquad \int_{\S(H_N)}\phi_N\,\dd\Pi_N^{\rm sh}=0,
\end{equation*}
has a solution $\phi_N\in C^{2,\alpha}(\S(H_N))$ for some $\alpha\in(0,1)$. Indeed, $\cV_{p,N}$ is $C^{2,\alpha}$ for any $0<\alpha<\min\{p,1\}$, which suffices for elliptic regularity and It\^o's formula. Infinite differentiability at particle collisions is not required.  Applying It\^o's formula to $\phi_N(\widehat U_\tau)$ in \eqref{eq:critical-angular-SDE} and using \eqref{eq:critical-log-radius} gives
\begin{equation}\label{eq:critical-zero-martingale-decomp}
 \ell_\tau
 =\ell_0+\phi_N(\widehat U_0)-\phi_N(\widehat U_\tau)+M_\tau,
\end{equation}
where
\[
 M_\tau
 :=\varsigma_N\widehat\beta_\tau
 +\varsigma_N\int_0^\tau
 \nabla_{\S}\phi_N(\widehat U_s)\cdot\dd\widehat W_s.
\]
The radial and angular martingales are orthogonal, and therefore
\begin{equation}\label{eq:critical-zero-bracket}
 \varsigma_N^2\tau
 \le \langle M\rangle_\tau
 =\varsigma_N^2\int_0^\tau
 \bigl(1+|\nabla_{\S}\phi_N(\widehat U_s)|^2\bigr)\,\dd s
 \le C_N\tau
\end{equation}
for a deterministic finite constant $C_N$.  By the Dambis--Dubins--Schwarz theorem 
$M_\tau=B_{\langle M\rangle_\tau}$ for a one-dimensional Brownian motion $B$.  Since the bracket tends to infinity and the corrector in \eqref{eq:critical-zero-martingale-decomp} is bounded,
\begin{equation}\label{eq:critical-zero-oscillation-intrinsic}
 \liminf_{\tau\to\infty}\ell_\tau=-\infty,
 \qquad
 \limsup_{\tau\to\infty}\ell_\tau=\infty
 \qquad\text{a.s.}
\end{equation}
Moreover, if $Q_\tau:=\langle M\rangle_\tau$ and $Q^{-1}$ denotes its inverse, then \eqref{eq:critical-zero-bracket} implies $(Q^{-1})'(u)\ge C_N^{-1}$ for almost every $u$.  With
$C:=|\ell_0|+2\|\phi_N\|_\infty$,
\begin{align*}
 \int_0^\infty e^{-p\ell_\tau}\,\dd\tau
 &\ge e^{-pC}\int_0^\infty e^{-pM_\tau}\,\dd\tau\\
 &=e^{-pC}\int_0^\infty e^{-pB_u}(Q^{-1})'(u)\,\dd u\\
 &\ge \frac{e^{-pC}}{C_N}\int_0^\infty e^{-pB_u}\,\dd u=\infty
 \qquad\text{a.s.}
\end{align*}
The last integral diverges because Brownian motion spends infinite Lebesgue time in $(-\infty,0]$.  Thus $T_\tau\to\infty$ in \eqref{eq:critical-physical-time}, there is no finite-time explosion.  Combining this with \eqref{eq:critical-zero-oscillation-intrinsic} and $S_{T_\tau}^N=e^{2\ell_\tau}/N$ proves \eqref{eq:critical-Gamma-zero}.

For $m_N\le2$, the zero-drift case cannot occur: $c_N-a_N^2/2=(m_N-2)a_N^2/4\le0$ and $\overline\Theta_N>0$, so $\Gamma_N<0$. In particular, the two-point sphere at $m_N=1$ belongs only to the consensus alternative.

To prove \eqref{eq:Gibbs-theta-bracket}, observe that $\Theta_N(U)=D_{p,N}(U)/(S^N(U))^{\bar{p}}$ is exactly the normalized dissipation.  Its minimum and maximum are $\cA_{p,n,N}^{\sharp}$ and
$\cB_{p,n,N}^{\sharp}$ by definition.  \autoref{lem:D-bounds} gives
$\cA_{p,n,N}\le\cA_{p,n,N}^{\sharp}$ and
$\cB_{p,n,N}^{\sharp}\le \cB_{p,N}$.  Integrating against $\Pi_N^{\rm sh}$ gives the result. 
\end{proof}

The compact angular diffusion also yields a quantitative fixed-$N$ mixing statement.  It does not by itself give a useful uniform large-$N$ rate, but it makes precise the time scale on which the Gibbs average enters the radial dynamics.

\begin{corollary}\label{cor:angular-mixing}
Assume the hypotheses of \autoref{thm:critical-shape-gibbs} and $m_N\ge2$.  Let $\gap_N>0$ be the spectral gap of $-\cL_N^{\rm ang}$ in $L^2(\Pi_N^{\rm sh})$.  If the initial angular law has density $h_0\in L^2(\Pi_N^{\rm sh})$, then for every $f\in L^2(\Pi_N^{\rm sh})$, writing $\Pi_N^{\rm sh}f:=\int_{\S(H_N)} f\,\dd\Pi_N^{\rm sh}$
\begin{equation}\label{eq:angular-L2-mixing}
 \left|\E f(\widehat U_\tau)-\Pi_N^{\rm sh}f\right|
 \le e^{-\gap_N\tau}
 \|h_0-1\|_{L^2(\Pi_N^{\rm sh})}
 \left\|f-\Pi_N^{\rm sh}f\right\|_{L^2(\Pi_N^{\rm sh})}.
\end{equation}
If $\widehat U_0\sim\Pi_N^{\rm sh}$, then
\begin{equation}\label{eq:angular-time-average-variance}
 \operatorname{Var}_{\mathbb P}\left(
 \frac1\tau\int_0^\tau\Theta_N(\widehat U_s)\,\dd s\right)
 \le
 \frac{2\left\|\Theta_N-\overline{\Theta}_N \right\|^2_{L^2(\Pi_N^{\rm sh})}}
 {\gap_N\tau}.
\end{equation}
For deterministic $\widehat U_0=u_0$ and every $\tau_0>0$, the law of $\widehat U_{\tau_0}$ has a bounded continuous, and hence $L^2(\Pi_N^{\rm sh})$, density, so \eqref{eq:angular-L2-mixing} applies for $\tau\ge\tau_0$.
\end{corollary}

\begin{proof}
The generator \eqref{eq:angular-generator} is symmetric in $L^2(\Pi_N^{\rm sh})$, uniformly elliptic, and defined on a compact connected manifold.  Its resolvent is compact, so the first nonzero eigenvalue $\gap_N$ is strictly positive and the semigroup satisfies (see for example \cite[Chapters~1 and~4]{BakryGentilLedoux2014}),
\[
 \|P_\tau f-\Pi_N^{\rm sh}f\|_2
 \le e^{-\gap_N\tau}\|f-\Pi_N^{\rm sh}f\|_2.
\]
Let $ g:=f-\Pi_N^{\rm sh}f$. If the initial law is $h_0\Pi_N^{\rm sh}$, invariance of $\Pi_N^{\rm sh}$ gives
\[
 \E f(\widehat U_\tau)-\Pi_N^{\rm sh}f
 =\int_{\S(H_N)} h_0\,P_\tau g\,\dd\Pi_N^{\rm sh} =\int_{\S(H_N)} (h_0-1)P_\tau g\,\dd\Pi_N^{\rm sh}.
\]
Cauchy--Schwarz and the $L^2$ decay estimate give \eqref{eq:angular-L2-mixing}.

Under stationarity, the Markov property gives
\[
 \Cov_{\mathbb P}\bigl(f(\widehat U_0),f(\widehat U_t)\bigr)
 =\langle g,P_tg\rangle_{L^2(\Pi_N^{\rm sh})}.
\]
Hence
\[
 \left|\Cov_{\mathbb P}\bigl(f(\widehat U_0),f(\widehat U_t)\bigr)\right|
 \le e^{-\gap_Nt}\|g\|_2^2
 =e^{-\gap_Nt}\left\|f-\Pi_N^{\rm sh}f\right\|^2_{L^2(\Pi_N^{\rm sh})}.
\]
Finally, stationarity and Fubini's theorem yield
\begin{align*}
 \operatorname{Var}_{\mathbb P}\Big(\frac1\tau\int_0^\tau f(\widehat U_s)\,\dd s\Big)
 &=\frac{2}{\tau^2}\int_0^\tau(\tau-t)
 \Cov_{\mathbb P}\bigl(f(\widehat U_0),f(\widehat U_t)\bigr)\,\dd t \le \frac{2\left\|f-\Pi_N^{\rm sh}f\right\|^2_{L^2(\Pi_N^{\rm sh})}}{\gap_N\tau}.
\end{align*}
Taking $f=\Theta_N$ proves \eqref{eq:angular-time-average-variance}. For $\tau_0>0$, uniform ellipticity on the compact sphere gives a bounded continuous transition density. This suffices for the last assertion.
\end{proof}

The mixing estimate in \autoref{cor:angular-mixing} is for fixed $N$,
no positive lower bound on $\gap_N$ uniform in $N$
is established here.
The Gibbs representation has an exact empirical-measure interpretation which also clarifies what can, and cannot, be inferred as $N\to\infty$.

For $m\ge1$, let $\cP_m(\R^n)$ denote the space of probability measures with finite $m$-th moment, and $\cW_m$ denote the $m$-Wasserstein distance.

\begin{proposition}\label{prop:gibbs-empirical-bridge}
For $u=(u^1,\ldots,u^N)\in\S(H_N)$ define
\begin{equation*}
 \nu_u^N:=\frac1N\sum_{i=1}^N\delta_{\sqrt N\,u^i}.
\end{equation*}
Then the following hold.
\begin{enumerate}[label=\textnormal{(\roman*)}]
\item \textnormal{(Exact empirical-shape identity.)}
The law $\nu_u^N$ satisfies
\begin{equation}\label{eq:theta-empirical-identity}
 \int_{\R^n} x\,\nu_u^N(\dd x)=0,\qquad
 \int_{\R^n}|x|^2\nu_u^N(\dd x)=1,\qquad
 \Theta_N(u)=D_p(\nu_u^N).
\end{equation}
Hence,
\begin{equation}\label{eq:gibbs-mean-empirical}
 \overline\Theta_N
 =\int_{\S(H_N)}D_p(\nu_u^N)\,\Pi_N^{\rm sh}(\dd u).
\end{equation}

\item \textnormal{(Sharp mean-field variational limit.)}
Define
\begin{equation}\label{eq:sharp-mf-shape-constant}
 \cA_{p,n}^{\rMF,\sharp}
 :=\inf\left\{D_p(\nu):
 \nu\in\cP_{p+2}(\R^n),\ \int_{\R^n} x\nu(\dd x)=0,\ \int_{\R^n}|x|^2\nu(\dd x)=1\right\}.
\end{equation}
The infimum is attained and, as $N\to\infty$,
\begin{equation}\label{eq:sharp-minima-largeN}
 \cA_{p,n,N}^{\sharp}\rightarrow
 \cA_{p,n}^{\rMF,\sharp}.
\end{equation}
Moreover,
\begin{equation}\label{eq:sharp-mf-vs-explicit}
 \cA_{p,n}^{\rMF,\sharp}\ge\cA_{p,n},
\end{equation}
with equality exactly for $p=0$ and $p=2$, and strict inequality for $p>0$, $p\ne2$.  In particular,
\begin{equation}\label{eq:sharp-mf-p2}
 \cA_{2,n}^{\rMF,\sharp}=4\Big(1+\frac1n\Big).
\end{equation}

\item \textnormal{(Fixed-$N$ zero-temperature limit.)}
For fixed $N$, replace $\varsigma_N$ in \eqref{eq:Gibbs-shape-law} by $\e>0$.  As $\e\to0$, the resulting Gibbs measures concentrate on the minimizers of $\Theta_N$, and their
means converge to $\cA_{p,n,N}^{\sharp}$.
\end{enumerate}
\end{proposition}

\begin{proof}
Since $\sum_{i=1}^N u^i=0$ and $\sum_{i=1}^N|u^i|^2=1$, the first two identities in \eqref{eq:theta-empirical-identity} are immediate.  The last one is the exact scaling computation
\[
 D_p(\nu_u^N)
 =\frac1{N^2}\sum_{i,j=1}^N
 \bigl|\sqrt N(u^i-u^j)\bigr|^{p+2}
 =N^{(p+2)/2}D_{p,N}(u)
 =\Theta_N(u),
\]
because $S^N(u)=1/N$.  This proves \eqref{eq:gibbs-mean-empirical} and shows that
$\cA_{p,n,N}^{\sharp}$ is exactly the minimum of $D_p$ over centered, unit-variance, equal-weight $N$-atomic laws.

For (ii) we first note that the mean-field infimum is attained.  For $p=0$ the functional is identically $D_0=2$ on the admissible set.  If $p>0$ and $(\nu_j)$ is a minimizing sequence, conditional Jensen gives
\[
 \int_{\R^n} |x|^{p+2}\nu_j(\dd x)
 \le \iint_{\R^n\times\R^n} |x-y|^{p+2}\nu_j(\dd x)\nu_j(\dd y)
 =D_p(\nu_j).
\]
Thus the minimizing sequence has a uniformly bounded $(p+2)$-moment.  Since $p+2>2$, this sequence is tight with uniformly integrable first and second moments.  Any weak limit remains centered with variance one, and lower semicontinuity of the nonnegative continuous interaction cost gives a minimizer.
The lower bound
$\cA_{p,n,N}^{\sharp}\ge\cA_{p,n}^{\rMF,\sharp}$ is therefore immediate.  

For the reverse asymptotic inequality, fix $\nu$ admissible in \eqref{eq:sharp-mf-shape-constant}.  Let $X_1,X_2,\ldots$ be i.i.d. with law $\nu$, put
\[
 \bar X_N:=\frac1N\sum_{i=1}^NX_i,
 \qquad
 s_N:=\frac1N\sum_{i=1}^N|X_i-\bar X_N|^2,
 \qquad
 Z_i^N:=\frac{X_i-\bar X_N}{\sqrt{s_N}},
\]
on the event $s_N>0$, and let $\nu_N=N^{-1}\sum_{i=1}^N\delta_{Z_i^N}$.  Since $\nu$ is nondegenerate and has variance one, $s_N\to1$ almost surely. In particular, $s_N>0$ almost surely for all sufficiently large $N$, so the exceptional finitely many indices may be ignored.  The strong law for empirical measures in $\cW_{p+2}$, together with $\bar X_N\to0$ and $s_N\to1$, gives $\nu_N\to\nu$ in $\cW_{p+2}$ almost surely, see the standard characterization of Wasserstein convergence in \cite{Villani}. Hence
$D_p(\nu_N)\to D_p(\nu)$ because the $p$-dissipation functional is continuous under $\cW_{p+2}$ convergence.  Each $\nu_N$ is centered, has variance one, and is an equal-weight $N$-atomic law, so
\[
 \limsup_{N\to\infty}\cA_{p,n,N}^{\sharp}
 \le D_p(\nu).
\]
Taking the infimum over admissible $\nu$ proves \eqref{eq:sharp-minima-largeN}.  The explicit bound
$\cA_{p,n}\le\cA_{p,n}^{\rMF,\sharp}$ is obtained directly from the same moment inequalities used at finite $N$: for every centered unit-variance law,
\[
 D_p(\nu)\ge 2^{\bar{p}},
\]
and, when $p\ge2$,
\[
 D_p(\nu)\ge D_2(\nu)^{\bar{p}/2}
 \ge \bigl[4(1+1/n)\bigr]^{\bar{p}/2}.
\]
Here the last inequality follows from
$D_2(\nu)=2\E|X|^4+2+4\tr(\Sigma^2)$,
$\E|X|^4\ge1$, and $\tr(\Sigma^2)\ge1/n$.
For $p=0$, $D_0=2$ at unit variance.  For $p=2$, these inequalities give
$D_2(\nu)\ge4(1+1/n)$, and equality is attained by any centered isotropic law supported on the unit sphere.

It remains to identify when the explicit lower bound can be an equality.  Let $X,Y$ be i.i.d. with an admissible nondegenerate law and set $\Delta:=|X-Y|^2$.  If $0<p<2$, equality in Jensen's inequality
$\E \Delta^{\bar{p}}\ge(\E \Delta)^{\bar{p}}$ would force $\Delta$ to be almost surely constant.  If $p>2$, equality in the $L^{\bar{p}}$ step
$\E \Delta^{\bar{p}}\ge(\E \Delta^2)^{\bar{p}/2}$ has the same consequence.  Such a constant-distance law cannot exist in $\R^n$: if the law has an atom then $\P(\Delta=0)>0$ while $\E \Delta=2$. If it is non-atomic, $n+2$ i.i.d. samples would almost surely be distinct and pairwise equidistant, which is impossible in $\R^n$.  Since the variational infimum is attained, the inequalities are therefore strict for every $p>0$, $p\ne2$.  This proves \eqref{eq:sharp-mf-vs-explicit} and \eqref{eq:sharp-mf-p2} without appealing to later mean-field estimates.

The fixed-$N$ zero-temperature statement is the compact-space Laplace principle applied to the continuous potential $\Theta_N$: outside any neighborhood of its compact minimizer set there is a strictly positive energy gap, so the mass of the Gibbs measure there is exponentially small in $\e^{-2}$.  Bounded continuity of $\Theta_N$ then gives convergence of the Gibbs mean. 
\end{proof}

\begin{remark}\label{rem:gibbs-largeN-free-energy}
\autoref{prop:gibbs-empirical-bridge} separates two different large-$N$ questions.  The minima have a clean mean-field limit, namely \eqref{eq:sharp-minima-largeN}.  The actual critical threshold, however, contains the Gibbs average \eqref{eq:gibbs-mean-empirical}.  For the particle model,
\[
 \frac{2\cV_{p,N}(u)}{\varsigma_N^2}
 =\frac{N}{(p+2)\kappa^2(1-\rho^2)}\,D_p(\nu_u^N),
\]
while the dimension of the sphere is also of order $N$.  Energy and shape entropy therefore compete on the same exponential scale.  A limit theorem for $\overline\Theta_N$ should consequently be formulated as a mean-field free-energy problem for the empirical shape laws under the spherical reference measure, rather than as a zero-temperature argument.  Establishing the corresponding large deviation principle and identifying its minimizers would give a direct answer to whether the finite-particle critical threshold converges to a mean-field shape criterion. 
\end{remark}

The exact Gibbs criterion immediately yields sharp comparison regions and a quantitative finite-size correction.

\begin{corollary}\label{cor:critical-finite-stabilization}
Assume the hypotheses of \autoref{thm:critical-shape-gibbs}.
\begin{enumerate}[label=\textnormal{(\roman*)}]
\item \textnormal{(Sharp shape-extremum criteria.)}
\begin{align*}
 c_N<\cA_{p,n,N}^{\sharp}+\frac12a_N^2
 &\quad\Rightarrow \quad S_t^N\to 0\quad\text{a.s.},\\
 c_N>\cB_{p,n,N}^{\sharp}+\frac12a_N^2
 &\quad\Rightarrow\quad \tau_N^{\rm expl}<\infty\quad\text{a.s.}
\end{align*}
At $c_N=\cA_{p,n,N}^{\sharp}+a_N^2/2$, the solution is global. It reaches consensus if
$\overline\Theta_N>\cA_{p,n,N}^{\sharp}$ and is scale-recurrent if equality holds.
The explosion criterion is optimal among criteria using only a uniform upper bound on the normalized shape dissipation.  In particular,
$c_N>\cB_{p,N}+a_N^2/2$ implies explosion, while
$c_N=\cA_{p,n,N}+a_N^2/2$ is nonexplosive.

\item \textnormal{(Optimal uniform-coercivity consensus bound.)}
For $N>1+2/n$, define
\begin{equation}\label{eq:finite-stabilization-cstar-sharp}
 c_N^{\sharp,\rm stab}
 :=\frac{\cA_{p,n,N}^{\sharp}}{1-(1+2/n)/N}.
\end{equation}
Then $c<c_N^{\sharp,\rm stab}$ implies consensus and
\begin{equation}\label{eq:sharp-stabilization-limit}
 c_N^{\sharp,\rm stab}\rightarrow \cA_{p,n}^{\rMF,\sharp}.
\end{equation}
With $\alpha_N=(1+2/n)/N$,
\begin{equation}\label{eq:sharp-stabilization-gap}
 c_N^{\sharp,\rm stab}-\cA_{p,n}^{\rMF,\sharp}
 \ge
 \cA_{p,n}^{\rMF,\sharp}\frac{\alpha_N}{1-\alpha_N}
 =\frac{\cA_{p,n}^{\rMF,\sharp}(1+2/n)}{N}+O(N^{-2}).
\end{equation}

\item \textnormal{(Explicit sufficient consensus bound.)}
For $N>1+2/n$, define
\begin{equation*}
 c_N^{\rm stab}:=\frac{\cA_{p,n,N}}{1-(1+2/n)/N}.
\end{equation*}
Then $c<c_N^{\rm stab}$ implies consensus and, for fixed $p>0$ and $n$,
$c_N^{\rm stab}>\cA_{p,n}$.  Moreover,
\begin{equation*}
 c_N^{\rm stab}=
 \begin{cases}
 2^{\bar{p}}\Big[1+\dfrac{p/2+1+2/n}{N}+O(N^{-2})\Big],&0<p<2,\\[3mm]
 \cA_{p,n}\Big[1+\dfrac{1+2/n}{N}+O(N^{-2})\Big],&p\ge2,
 \end{cases}
\end{equation*}
where the second expansion holds for all sufficiently large $N$.
\end{enumerate}
\end{corollary}

\begin{proof}
For (i), the strict implications follow from \eqref{eq:Gibbs-theta-bracket} and \autoref{thm:critical-shape-gibbs}.  For
$c_N=\cA_{p,n,N}^{\sharp}+a_N^2/2$, the transform $Z=(S^N)^{-p/2}$ dominates a squared-Bessel process of dimension two, so zero is polar for $Z$ and explosion of $S^N$ is impossible.  The exact alternatives at that boundary follow from
$\Gamma_N=\cA_{p,n,N}^{\sharp}-\overline\Theta_N\le0$.

For (ii), since $c_N=c(1-1/N)$ and $a_N^2/2=2c/(nN)$, the sharp inequality
$c_N<\cA_{p,n,N}^{\sharp}+a_N^2/2$ is equivalent to $c<c_N^{\sharp,\rm stab}$.  The convergence \eqref{eq:sharp-stabilization-limit} follows from \eqref{eq:sharp-minima-largeN}.  Since every finite-$N$ admissible empirical law is also admissible in the mean-field variational problem, $\cA_{p,n,N}^{\sharp}\ge\cA_{p,n}^{\rMF,\sharp}$.  Substituting this into \eqref{eq:finite-stabilization-cstar-sharp} gives \eqref{eq:sharp-stabilization-gap}. Similarly, for (iii), the explicit inequality
$c_N<\cA_{p,n,N}+a_N^2/2$ is equivalent to $c<c_N^{\rm stab}$.  For $0<p<2$, $\cA_{p,n,N}=2^{\bar{p}}(N/(N-1))^{p/2}$, giving the first expansion.  If $p\ge2$ and $n$ is fixed, then $d_N=n$ for all sufficiently large $N$, and the dimension-refined term $\cA_{p,n}=2^{\bar{p}}(1+1/n)^{\bar{p}/2}$ eventually dominates the pairwise term.  This gives the second expansion. 
\end{proof}

The critical radius depends on the time-dependent shape coefficient $\Theta_N(U_t)$, whose long-time average is determined by the angular Gibbs measure and bracketed by the sharp shape extrema.  At $\Gamma_N=0$, the Poisson-equation decomposition in \autoref{thm:critical-shape-gibbs} makes the logarithmic radius a nondegenerate recurrent martingale up to a bounded corrector, yielding global existence and oscillations between arbitrarily small and arbitrarily large variance scales.

Above critical homogeneity, the boundary behavior of the reciprocal radius is governed by a dimension index.  The almost-sure nonexplosive side occurs only in the exceptional low-dimensional relative spaces, while the sharp shape extrema give quantitative two-sided bounds on the explosion probability.

\begin{theorem}\label{thm:supercritical-dimension}
Assume $\rho<1$, $p\ge0$, $q>(p+2)/4$, and $S_0^N>0$.  Let $\tau_N^{\rm expl}$ denote the explosion time of the unique maximal strong solution.  Then
\begin{equation}\label{eq:supercritical-dimension-dichotomy}
 \begin{cases}
  \P(\tau_N^{\rm expl}=\infty)=1, & m_N=n(N-1)\le2,\\[1mm]
  \P(\tau_N^{\rm expl}<\infty)>0, & m_N=n(N-1)>2.
 \end{cases}
\end{equation}
Hence $m_N=2$ is the sharp boundary-accessibility threshold.  If $m_N>2$, set
\begin{equation*}
 k_N^{\rm sup}:=\frac{m_N-2}{2(2q- \bar{p})},
 \qquad
 \Xi_N(C;s):=\frac{2C}{a_N^2(2q-\bar{p})}s^{\bar{p}-2q},
\end{equation*}
and let $Q(k,x):=\Gamma(k,x)/\Gamma(k)$ denote the regularized upper incomplete gamma function.  Then the explosion probability satisfies the sharp comparison bounds
\begin{equation}\label{eq:supercritical-probability-bounds}
 Q\left(k_N^{\rm sup},\Xi_N(\cB_{p,n,N}^{\sharp};S_0^N)\right)
 \le \P(\tau_N^{\rm expl}<\infty)
 \le Q\left(k_N^{\rm sup},\Xi_N(\cA_{p,n,N}^{\sharp};S_0^N)\right).
\end{equation}
The same bounds hold with $\cB_{p,n,N}^{\sharp}$ and $\cA_{p,n,N}^{\sharp}$ replaced by $\cB_{p,N}$ and $\cA_{p,n,N}$, respectively.  For $p=0$ all four constants coincide, so \eqref{eq:supercritical-probability-bounds} is exact.
For fixed $N$ with $m_N>2$,
\begin{equation}\label{eq:supercritical-small-large-data}
 \P(\tau_N^{\rm expl}<\infty)\to0\quad(S_0^N\to0),
 \qquad
 \P(\tau_N^{\rm expl}<\infty)\to1\quad(S_0^N\to\infty).
\end{equation}
\end{theorem}

\begin{proof}
Up to the first hitting of $S^N=0$ or the explosion time, the martingale representation theorem and \eqref{eq:bracket-power} give a Brownian motion $\beta$ such that
\begin{equation*}
 \dd S_t^N
 =\left[-\theta_t(S_t^N)^{\bar{p}}+c_N(S_t^N)^{2q}\right]\dd t
 +a_N(S_t^N)^{q+1/2}\dd\beta_t,
 \qquad
 \theta_t:=\frac{D_{p,N}(Y_t)}{(S_t^N)^{\bar{p}}}.
\end{equation*}
By the extremal definitions in \eqref{eq:shape-extrema-global},
\begin{equation}\label{eq:theta-two-sided-supercritical}
 \cA_{p,n,N}^{\sharp}\le\theta_t\le \cB_{p,n,N}^{\sharp}.
\end{equation}
For $Z_t:=(S_t^N)^{1-2q}$, $W:=-\beta$, It\^o's formula, \eqref{eq:cn-an-relation}, and a change of sign of the Brownian motion give
\begin{equation}\label{eq:supercritical-Z}
 \dd Z_t
 =(2q-1)\bigl[\theta_t Z_t^{\delta}+q a_N^2-c_N\bigr]\dd t
 +(2q-1) a_N\sqrt{Z_t}\dd W_t,
\end{equation}
with $\delta:=\frac{2q-\bar{p}}{2q-1}\in(0,1].$
The drift of $Z$, before a boundary hit, is bounded above by $C_N(1+Z)$ because $0<\delta\le1$ and $\theta_t$ is bounded. Together with its square-root diffusion, this prevents $Z$ from escaping to infinity in finite time. Its hitting of zero is therefore precisely radial explosion of $S^N$.  Denote
\begin{equation*}
 b_0:=(2q-1)(q a_N^2-c_N),
 \qquad
 \sigma_0:=(2q-1) a_N,
 \qquad
 \Upsilon_N:=\frac{2b_0}{\sigma_0^2}
 =\frac{4q-m_N}{2(2q-1)}.
\end{equation*}
The last identity uses $c_N=m_Na_N^2/4$.  The term $\Upsilon_N$ is the index in Feller's test for explosion.  In particular,
\begin{equation}\label{eq:lambda-dimension}
 \Upsilon_N\ge1 \quad\Leftrightarrow\quad m_N\le2.
\end{equation}

Suppose first that $m_N\le2$.  Let $V$ solve, up to its first boundary hit,
\begin{equation}\label{eq:supercritical-lower-BES}
 \dd V_t=b_0\dd t+\sigma_0\sqrt{V_t}\dd W_t,
 \qquad V_0=Z_0.
\end{equation}
By \eqref{eq:theta-two-sided-supercritical}, the drift of $Z$ is at least $b_0$.  Localizing on $[\e,R]$ and applying the one-dimensional comparison theorem yields $Z_t\ge V_t$ up to the localization time.  Since \eqref{eq:lambda-dimension} is precisely the Feller condition $2b_0\ge\sigma_0^2$, the origin is polar for \eqref{eq:supercritical-lower-BES}.  Letting $\e\to0$ and $R\to\infty$ shows that $Z$ cannot hit zero in finite time.  Hence $S^N$ cannot explode.  If $S^N$ reaches zero, local pathwise uniqueness for $q>1/2$ makes the relative configuration absorbing, and the center is nonexplosive on every finite interval.  Therefore $\tau_N^{\rm expl}=\infty$ almost surely.

Now suppose that $m_N>2$, so $\Upsilon_N<1$.  For a constant $C>0$, let $V^C$ solve
\begin{equation*}
 \dd V_t^C=\Big[b_0+(2q-1) C(V_t^C)^{\delta}\Big]\dd t
 +\sigma_0\sqrt{V_t^C}\dd W_t,
 \qquad V_0^C=Z_0,
\end{equation*}
up to its boundary hitting time.  Its scale density is
\begin{equation*}
 (\mathfrak s_C)'(u)
 =u^{-\Upsilon_N}
 \exp\Big(-\frac{2C}{a_N^2(2q- \bar{p})}u^{\delta}\Big).
\end{equation*}
  Since $\Upsilon_N<1$, the scale increment at zero is of order $u^{1-\Upsilon_N}$, whereas the speed density is of order $u^{\Upsilon_N-1}$. Their product is integrable there, so zero is accessible. Since $0<\delta\le1$, the drift has at most linear upper growth and the diffusion has square-root growth. Localization and Gr\"onwall's inequality rule out a finite-time escape to infinity for $V^C$. The stretched exponential makes its scale finite at infinity.  The scale formula therefore gives
\begin{align}
 \P_z(\tau_0^{V^C}<\infty)
 &=\frac{\int_z^\infty u^{-\Upsilon_N}e^{-\lambda_Cu^\delta}\,\dd u}
 {\int_0^\infty u^{-\Upsilon_N}e^{-\lambda_Cu^\delta}\,\dd u} =Q\Big(\frac{1-\Upsilon_N}{\delta},\lambda_Cz^\delta\Big),
 \qquad
 \lambda_C:=\frac{2C}{a_N^2(2q- \bar{p})}.
 \label{eq:supercritical-C-hit}
\end{align}
Here
\[
 \frac{1-\Upsilon_N}{\delta}=\frac{m_N-2}{2(2q- \bar{p})}=k_N^{\rm sup},
 \qquad
 Z_0^\delta=(S_0^N)^{\bar{p}-2q}.
\]

For fixed $0<\e<Z_0<R$, all coefficients are locally Lipschitz on $[\e,R]$.  The drift bounds \eqref{eq:theta-two-sided-supercritical} and the one-dimensional SDE comparison theorem, proved via Tanaka's formula, see \cite[Proposition~5.2.18]{KaratzasShreve1991}, therefore give
\[
 V_t^{\cA_{p,n,N}^{\sharp}}\le Z_t\le V_t^{\cB_{p,n,N}^{\sharp}}
\]
up to the common localized boundary hitting time.  Letting $\e\to0$ and $R\to\infty$, a hit of zero by $V^{\cB_{p,n,N}^{\sharp}}$ forces a hit by $Z$, while a hit by $Z$ forces a hit by $V^{\cA_{p,n,N}^{\sharp}}$.  Formula \eqref{eq:supercritical-C-hit} gives \eqref{eq:supercritical-probability-bounds}, and in particular proves the positive-probability assertion in \eqref{eq:supercritical-dimension-dichotomy}.  Since $\bar{p}-2q<0$ and $Q(k,x)$ decreases continuously from $1$ to $0$ as $x$ increases from $0$ to $\infty$, the two bounds also imply \eqref{eq:supercritical-small-large-data}.
\end{proof}

\begin{remark}\label{rem:supercritical-p}
In \eqref{eq:supercritical-Z}, the interaction term
$(2q-1)\theta_t Z_t^\delta$ vanishes at $Z=0$ because $\delta>0$ and
$\theta_t$ is bounded. It therefore leaves the Feller index
$\Upsilon_N$ unchanged, and the threshold for boundary accessibility depends
only on $m_N$. The exponent $p$ still affects $\delta$ and the bounds on
$\theta_t$, and hence the explosion probabilities through
\eqref{eq:supercritical-probability-bounds}.
\end{remark}

\section{Conditional McKean--Vlasov dynamics}\label{sec:mf}

Let $\cF_t^0:=\sigma(B_\tau^0:0\le \tau\le t)$, augmented in the usual way.  The conditional McKean--Vlasov equation associated with \eqref{eq:particle} is
\begin{equation}\label{eq:MV}
 \dd X_t
 =\int_{\R^n}K(y-X_t)\,\mu_t(\dd y)\,\dd t
 +\kappa s_t^q\Bigl(\rho\,\dd B_t^0+\sqrt{1-\rho^2}\,\dd B_t\Bigr),
 \qquad \mu_t=\Law(X_t\mid\cF_t^0),
\end{equation}
where
\begin{equation*}
 s_t:=\frac12\iint_{\R^n\times\R^n}|x-y|^2\,\mu_t(\dd x)\mu_t(\dd y).
\end{equation*}
The initial $X_0$ has deterministic law $\mu_0\in \cP_2(\R^n)$ and is independent of both Brownian motions.

We introduce a little more notation. We write $f\push\nu$ for the push forward of the measure $\nu$ by a function $f$, $\tau_a(x)=x+a$ for translation by $a$ and $\mathsf{Dil}_a(x)=ax$ for dilation by a factor $a$.
We next remove the common noise from the conditional law by a random translation.

\begin{theorem}\label{thm:translation}
Let $(X,\mu)$ solve \eqref{eq:MV} on the stochastic interval $[0,T)$, where $T$ is an $\cF^0$-stopping time and $\mu$ has continuous $\cP_2$-valued paths. Assume that there are $\cF^0$-stopping times $T_j\to T$, with $T_j<T$ and $T_j\le j$, such that
\begin{equation}\label{eq:translation-localization}
 \E\Big[\sup_{t\le T_j}\int_{\R^n}|x|^2\mu_t(\dd x)
 +\int_0^{T_j}\Big(s_t^{2q}
 +\iint_{\R^n\times\R^n}|x-y|^{p+1}\mu_t(\dd x)\mu_t(\dd y)\Big)\dd t\Big]<\infty.
\end{equation}
Let
\[
 m_t:=\int_{\R^n} x\mu_t(\dd x),\qquad Y_t:=X_t-m_t,\qquad
 \nu_t:=\Law(Y_t\mid\cF_t^0).
\]
Then we have for $t\in [0,T)$ 
\begin{align}
 \dd m_t&=\kappa\rho s_t^q\,\dd B_t^0,\label{eq:mSDE}\\
 \dd Y_t&=\int_{\R^n} K(z-Y_t)\,\nu_t(\dd z)\,\dd t
 +\kappa\sqrt{1-\rho^2}\,s_t^q\,\dd B_t,
 \label{eq:centeredMV}
\end{align}
and
\begin{equation}\label{eq:translation-law}
 \mu_t=(\tau_{m_t})\push\nu_t,
\end{equation}
\begin{equation}\label{eq:scenteredMV}
 \int_{\R^n} y\nu_t(\dd y)=0,\qquad s_t=\int_{\R^n}|y|^2\nu_t(\dd y).
\end{equation}
If the centered equation has a unique solution in law, the maximal lifetime of this centered law path is a deterministic number $T_{\star}\in(0,\infty]$.
\end{theorem}

\begin{proof}
Write $b_\mu(x)=\int_{\R^n} K(y-x)\mu(\dd y)$. Up to $T_j$, the idiosyncratic integral is square-integrable by \eqref{eq:translation-localization}. Its conditional expectation given the common noise is zero, first for simple predictable integrands and then by $L^2$ approximation. The common-noise integral is $\cF^0$-adapted. Moreover, by oddness of $K$,
\[
 \int_{\R^n} b_{\mu_t}(x)\mu_t(\dd x)
 =\iint_{\R^n\times\R^n} K(y-x)\mu_t(\dd x)\mu_t(\dd y)=0.
\]
 Conditional projection of the stopped equation therefore gives
\[
 m_{t\wedge T_j}
 =m_0+\kappa\rho\int_0^{t\wedge T_j}s_u^q\,\dd B_u^0.
\]
Here stopping commutes with conditional projection because $T_j$ is an $\cF^0$-stopping time and the Brownian system is non-anticipating. Letting $j\to\infty$ proves \eqref{eq:mSDE} on $[0,T)$. Subtraction and translation invariance give \eqref{eq:centeredMV}--\eqref{eq:scenteredMV}.

For $\varphi\in C_c^2(\R^n)$, conditional It\^o's formula gives
\begin{equation}\label{eq:centered-weak-FP}
 \frac{\dd}{\dd t}\int_{\R^n}\varphi\,\dd\nu_t
 =\int_{\R^n}\Big[b_{\nu_t}\cdot\nabla\varphi
 +\frac{\kappa^2(1-\rho^2)}2s_t^{2q}\Delta\varphi\Big]\dd\nu_t
\end{equation}
in its integrated, stopped form. This equation has deterministic coefficients and deterministic initial data. Uniqueness in law of the solution forces each common-noise realization to agree with the same deterministic law path. 
\end{proof}

If there is uniqueness, the centered dynamics therefore depend on $(\kappa,\rho)$ only through the term in the effective idiosyncratic variance $\kappa^2(1-\rho^2)$. 

\begin{corollary}\label{cor:common-noise-translation}
Consider two parameter pairs $(\kappa,\rho)$ and $(\widetilde\kappa,\widetilde\rho)$ satisfying
\[
 \kappa^2(1-\rho^2)=\widetilde\kappa^2(1-\widetilde\rho^2).
\]
Whenever the centered equation is unique for the same centered initial law, the two models have exactly the same centered law path $\nu_t$, the same variance path $s_t$, and the same centered propagation-of-chaos problem.  Their conditional laws differ only through the random translations generated by their respective center equations.  In particular, every assertion about relative configurations depends on $(\kappa,\rho)$ only through the idiosyncratic variance parameter $\sigma^2=\kappa^2(1-\rho^2)$.
\end{corollary}

\begin{proof}
This is immediate from \eqref{eq:centeredMV}: after centering, the common Brownian motion disappears and the only remaining noise coefficient is $\kappa\sqrt{1-\rho^2}$.  The original conditional law is reconstructed by \eqref{eq:translation-law}. 
\end{proof}

The translated equation has a finite-variation variance identity with the same coercive structure as the particle system.

\begin{proposition}\label{prop:variance-MV}
Let $\nu\in C([0,T_{\star});\cP_2(\R^n))$ be a deterministic centered law solution of \eqref{eq:centeredMV}, where $T_{\star}\in(0,\infty]$ is deterministic, and assume
\[
 \int_0^T\int_{\R^n}|x|^{p+2}\nu_t(\dd x)\,\dd t<\infty
 \qquad\text{for every }T<T_{\star}.
\]
Set $c=n\kappa^2(1-\rho^2)$ and $\bar p =(p+2)/2$. Then $s_t=\int_{\R^n}|x|^2\nu_t(\dd x)$ is locally absolutely continuous and, for almost every $t<T_{\star}$,
\begin{equation}\label{eq:s-exact}
 s_t'=-D_p(\nu_t)+cs_t^{2q},
 \qquad
 \text{ with } D_p(\nu)=\iint_{\R^n\times\R^n}|x-y|^{p+2}\,\nu(\dd x)\nu(\dd y).
\end{equation}

Moreover,
\begin{equation}\label{eq:D-MV-sharp-lower}
 D_p(\nu_t)\ge \cA_{p,n}^{\rMF,\sharp}s_t^{\bar{p}}
 \ge\cA_{p,n}s_t^{\bar{p}},
\end{equation}
and therefore
\begin{equation}\label{eq:sbarrier-sharp}
 s_t'\le-\cA_{p,n}^{\rMF,\sharp}s_t^{\bar{p}}+cs_t^{2q}
 \le-\cA_{p,n}s_t^{\bar{p}}+cs_t^{2q}.
\end{equation}

For $p=2$ one has the exact identity
\begin{equation}\label{eq:D2-MV-identity}
 D_2(\nu_t)=2\int_{\R^n}|x|^4\nu_t(\dd x)+2s_t^2+4\tr(\Sigma_t^2),
 \qquad \Sigma_t:=\int_{\R^n} x\otimes x\,\nu_t(\dd x),
\end{equation}
and hence the sharp bound
\begin{equation}\label{eq:D2-MV-sharp}
 D_2(\nu_t)\ge4\left(1+\frac1n\right)s_t^2.
\end{equation}
Equality holds in \eqref{eq:D2-MV-sharp} for centered isotropic laws supported on a sphere.
\end{proposition}

\begin{proof}
Apply It\^o's formula to $|Y_t|^2$ in \eqref{eq:centeredMV}. Conditionally on the common noise, let $\widetilde Y_t$ be an independent copy of $Y_t$ and interpret the expectations below conditionally.  Then
\begin{align*}
 2\E\left[Y_t\cdot K(\widetilde Y_t-Y_t)\right]
 &=\E\left[(Y_t-\widetilde Y_t)\cdot K(\widetilde Y_t-Y_t)\right]=-\E|Y_t-\widetilde Y_t|^{p+2}.
\end{align*}
The It\^o correction is $n\kappa^2(1-\rho^2)s_t^{2q}=cs_t^{2q}$, which proves \eqref{eq:s-exact}.  Since
\[
 \iint_{\R^n\times\R^n}|x-y|^2\nu_t(\dd x)\nu_t(\dd y)=2s_t,
\]
Applying Jensen's inequality gives $D_p\ge2^{\bar{p}}s^{\bar{p}}$ for every $p\ge0$.  If $s_t>0$, normalize by $x\mapsto s_t^{-1/2}x$.  The resulting law is centered with unit variance, and homogeneity gives
\[
 D_p(\nu_t)=s_t^{\bar{p}}D_p\bigl((\mathsf{Dil}_{s_t^{-1/2}})\push\nu_t\bigr)
 \ge \cA_{p,n}^{\rMF,\sharp}s_t^{\bar{p}}.
\]
For $s_t=0$ the same inequality is trivial.  This proves the first inequality in \eqref{eq:D-MV-sharp-lower}, the second follows from \autoref{prop:gibbs-empirical-bridge}.  

For the explicit improvement when $p\ge2$, let $X,\widetilde X$ be independent with law $\nu$ and covariance $\Sigma_\nu$.  Expanding the fourth power and using $\E X=0$ gives \eqref{eq:D2-MV-identity}.  Since $\int_{\R^n}|x|^4\nu(\dd x)\ge s^2$ and $\tr(\Sigma_\nu^2)\ge s^2/n$, \eqref{eq:D2-MV-sharp} follows.  If $p\ge2$, then $\bar{p}\ge2$ and, with $R=|X-\widetilde X|^2$,
\[
 D_p(\nu)=\E R^{\bar{p}}\ge(\E R^2)^{\bar{p}/2}
 \ge\Big[4\Big(1+\frac1n\Big)s^2\Big]^{\bar{p}/2},
\]
which proves the explicit second inequality in \eqref{eq:D-MV-sharp-lower}. The equality in \eqref{eq:D2-MV-sharp} requires $|X|^2=s$ almost surely and $\Sigma_\nu=(s/n)\Id_n$, conditions satisfied by isotropic spherical laws. 
\end{proof}

At critical homogeneity the mean-field equation has an exact dilation symmetry.  This shows, in particular, that the size of the initial variance cannot by itself create a separate large-data phase at the critical exponent.

\begin{proposition}\label{prop:critical-mf-scaling}
Assume $q=(p+2)/4$ and let $Y$ be a centered mean-field solution of \eqref{eq:centeredMV} with deterministic law $\nu_t$, variance $s_t$, and maximal lifetime $\tau_{\star}\in(0,\infty]$.  For $a>0$ set
\[
 Y_t^{(a)}:=aY_{a^pt},
 \qquad
 \nu_t^{(a)}:=(\mathsf{Dil}_a)\push\nu_{a^pt},
\]
for $0\le t<a^{-p}\tau_{\star}$.  Then, after the Brownian rescaling
$B_t^{(a)}:=a^{-p/2}B_{a^pt}$, the process $Y^{(a)}$ solves the same centered McKean--Vlasov equation with initial law $(\mathsf{Dil}_a)\push\nu_0$.  Its variance and maximal lifetime satisfy
\begin{equation}\label{eq:critical-mf-scaling}
 s_t^{(a)}=a^2s_{a^pt},
 \qquad
 \tau_{\star}^{(a)}=a^{-p}\tau_{\star}.
\end{equation}
Thus finite-time explosion, global existence, and convergence to consensus are invariant under dilation of the centered initial law. Only the spatial and temporal scales change.
\end{proposition}

\begin{proof}
The interaction field is homogeneous of degree $p+1$:
\[
 b_{(\mathsf{Dil}_a)\push\nu}(ax)
 :=\int_{\R^n} K(y-ax)\,(\mathsf{Dil}_a\push\nu)(\dd y)
 =a^{p+1}b_\nu(x).
\]
Since $B^{(a)}$ is a standard Brownian motion and $2q=(p+2)/2=1+p/2$,
\[
 a\,a^{p/2}s^q\,\dd B_t^{(a)}
 =a^{1+p/2}s^q\,\dd B_t^{(a)}
 =(a^2s)^q\,\dd B_t^{(a)}.
\]
Applying the time change $u=a^pt$ to \eqref{eq:centeredMV} therefore gives exactly the equation for $Y^{(a)}$.  The variance identity in \eqref{eq:critical-mf-scaling} is immediate, and the lifetime relation follows because $|Y_t^{(a)}|=a|Y_{a^pt}|$. 
\end{proof}

At critical homogeneity, variance normalization and the intrinsic clock yield an autonomous mean-field shape equation with a constrained free-energy structure.

\begin{theorem}\label{thm:critical-mf-shape}
Assume $\rho<1$, $q=(p+2)/4$, and let $\nu_t$ be the deterministic centered law of a positive-variance mean-field solution with locally bounded $(p+2)$-moment on its maximal lifetime.  Set
\[
 \sigma^2:=\kappa^2(1-\rho^2)=\frac cn,
 \qquad
 b_{\tilde{\nu}}(x):=\int_{\R^n} K(y-x)\tilde{\nu}(\dd y).
\]
Then the following hold.
\begin{enumerate}[label=\textnormal{(\roman*)}]
\item \textnormal{(Autonomous shape dynamics.)}
Define
\begin{equation*}
 \tau(t):=\int_0^t s_u^{p/2}\,\dd u,
 \qquad
 \tilde{\nu}_\tau:=\bigl(\mathsf{Dil}_{s_{t(\tau)}^{-1/2}}\bigr)\push\nu_{t(\tau)},
\end{equation*}
where $t(\tau)$ is the inverse clock, and set
\begin{equation*}
 \Theta(\tilde{\nu}):=D_p(\tilde{\nu}).
\end{equation*}
Then $\tilde{\nu}_\tau$ is centered with unit variance, and
$U_\tau:=s_{t(\tau)}^{-1/2}Y_{t(\tau)}$ satisfies
\begin{equation}\label{eq:mf-critical-shape-SDE}
 \dd U_\tau
 =\Big[b_{\tilde{\nu}_\tau}(U_\tau)
 +\frac12\bigl(\Theta(\tilde{\nu}_\tau)-c\bigr)U_\tau\Big]\dd\tau
 +\sigma\,\dd\widetilde B_\tau,
 \qquad \tilde{\nu}_\tau=\Law(U_\tau),
\end{equation}
where $\widetilde B$ is an $n$-dimensional Brownian motion in the time-changed filtration.  Equivalently, $\tilde{\nu}$ is the weak solution to the following Fokker-Planck equation
\begin{equation}\label{eq:mf-critical-shape-PDE}
 \partial_\tau\tilde{\nu}
 =-\nabla\cdot\Big[
 \Big(b_{\tilde{\nu}}+\frac12(\Theta(\tilde{\nu})-c)x\Big)\tilde{\nu}\Big]
 +\frac{\sigma^2}{2}\Delta\tilde{\nu}.
\end{equation}

\item \textnormal{(Exact radial reconstruction and phase criterion.)}
The variance is recovered from
\begin{equation}\label{eq:mf-critical-log-variance}
 \frac{\dd}{\dd\tau}\log s_{t(\tau)}
 =c-\Theta(\tilde{\nu}_\tau),
 \qquad
 t(\tau)=\int_0^\tau s_{t(a)}^{-p/2}\,\dd a.
\end{equation}
Assume the normalized flow exists for all intrinsic times.  If
\begin{equation}\label{eq:mf-critical-average-align}
 \liminf_{\tau\to\infty}
 \frac1\tau\int_0^\tau\Theta(\tilde{\nu}_a)\,\dd a>c,
\end{equation}
then the original solution is global and $s_t\to0$.  If $p>0$ and
\begin{equation}\label{eq:mf-critical-average-explode}
 \limsup_{\tau\to\infty}
 \frac1\tau\int_0^\tau\Theta(\tilde{\nu}_a)\,\dd a<c,
\end{equation}
then the physical clock has a finite limit and $s_t\to\infty$ at that time. Hence, the original solution explodes in finite physical time.

\item \textnormal{(Constrained free-energy dissipation.)}
For a centered unit-variance law $\tilde{\nu}=f\,\dd x$, define
\begin{equation}\label{eq:mf-critical-free-energy}
 \cF(\tilde{\nu})
 :=\frac{D_p(\tilde{\nu})}{2(p+2)}
 +\frac{\sigma^2}{2}\int_{\R^n}f\log f\,\dd x.
\end{equation}
Assume that $f_\tau>0$ is $C^{1,2}$. On each compact intrinsic-time interval $I$, assume that $f_\tau$, $\partial_\tau f_\tau$, and its first two spatial derivatives are rapidly decreasing in $x$, uniformly in time, and that $|\log f_\tau(x)|$ has at most polynomial growth, uniformly in time. Assume also that
\[
\int_I\int_{\R^n}\left(|x|^{2p+2}+|\nabla\log f_\tau(x)|^2\right)
 f_\tau(x)\,\dd x\,\dd\tau<\infty.
\]
Then $\cF(\tilde{\nu}_\tau)$ is locally absolutely continuous. With
\[
 V_{\tilde{\nu}}(x):=\frac1{p+2}\int_{\R^n}|x-y|^{p+2}\tilde{\nu}(\dd y),
 \qquad
 \lambda_{\tilde{\nu}}:=\frac12(\Theta(\tilde{\nu})-c),
\]
for almost every intrinsic time one has
\begin{equation}\label{eq:mf-critical-free-energy-dissipation}
 \frac{\dd}{\dd\tau}\cF(\tilde{\nu}_\tau)
 =-\int_{\R^n}\Big|
 \nabla V_{\tilde{\nu}_\tau}
 +\frac{\sigma^2}{2}\nabla\log f_\tau
 -\lambda_{\tilde{\nu}_\tau}x
 \Big|^2 f_\tau(x)\,\dd x\le0.
\end{equation}

\end{enumerate}
\end{theorem}

\begin{proof}
Write $Y_t=\sqrt{s_t}\,U_t$ and let $\tilde{\nu}_t$ denote the law of $U_t$.  Homogeneity of $K$ gives
\[
 b_{\nu_t}(\sqrt{s_t}u)=s_t^{(p+1)/2}b_{\tilde{\nu}_t}(u),
 \qquad
 D_p(\nu_t)=s_t^{(p+2)/2}D_p(\tilde{\nu}_t).
\]
At critical homogeneity $2q=(p+2)/2$, \autoref{prop:variance-MV} therefore reads
\[
 s_t'=s_t^{1+p/2}\bigl(c-D_p(\tilde{\nu}_t)\bigr).
\]
Since $s_t$ is deterministic and positive before the lifetime, It\^o's formula gives
\[
 \dd U_t
 =s_t^{p/2}\Big[b_{\tilde{\nu}_t}(U_t)
 +\frac12\bigl(D_p(\tilde{\nu}_t)-c\bigr)U_t\Big]\dd t
 +\sigma s_t^{p/4}\dd B_t.
\]
The time change $\dd\tau=s_t^{p/2}\dd t$ turns the martingale into $\sigma\dd\widetilde B_\tau$ and proves \eqref{eq:mf-critical-shape-SDE}.  The normalization gives $\int_{\R^n} x\tilde{\nu}_\tau(\dd x)=0$ and $\int_{\R^n}|x|^2\tilde{\nu}_\tau(\dd x)=1$, so $\Theta(\tilde{\nu}_\tau)=D_p(\tilde{\nu}_\tau)$.  Dividing the variance equation by $s_t$ and by $\dd\tau/\dd t=s_t^{p/2}$ gives \eqref{eq:mf-critical-log-variance}. The inverse-clock formula follows immediately.

If \eqref{eq:mf-critical-average-align} holds, there are $\e>0$ and $\tau_0$ such that
$\log s_{t(\tau)}\le C-\e\tau$ for $\tau\ge\tau_0$.  Hence $s_{t(\tau)}\to0$, while the inverse clock in \eqref{eq:mf-critical-log-variance} diverges, so physical time tends to infinity.  If $p>0$ and \eqref{eq:mf-critical-average-explode} holds, then
$\log s_{t(\tau)}\ge-C+\e\tau$ for large $\tau$. Consequently,
$\int_0^\infty s_{t(a)}^{-p/2}\dd a<\infty$ and $s_{t(\tau)}\to\infty$, which is finite-time explosion.

For the last assertion, $b_{\tilde{\nu}}=-\nabla V_{\tilde{\nu}}$, so \eqref{eq:mf-critical-shape-PDE} can be written
\[
 \partial_\tau f
 =\nabla\cdot\Big(
 f\Big[\nabla V_{\tilde{\nu}}+\frac{\sigma^2}{2}\nabla\log f-\lambda_{\tilde{\nu}} x\Big]
 \Big).
\]
Let $G_{\tilde{\nu}}:=\nabla V_{\tilde{\nu}}+(\sigma^2/2)\nabla\log f$.  Symmetrization and integration by parts give
\[
 \int_{\R^n} x\cdot\nabla V_{\tilde{\nu}}\,f\,\dd x=\frac12D_p(\tilde{\nu}),
 \qquad
 \frac{\sigma^2}{2}\int_{\R^n} x\cdot\nabla\log f\,f\,\dd x
 =-\frac{n\sigma^2}{2}=-\frac c2.
\]
Oddness gives $\int_{\R^n} G_{\tilde{\nu}} f\,\dd x=0$, so the mean constraint requires no additional multiplier. Moreover, $\int_{\R^n} x\cdot G_{\tilde{\nu}} f\,\dd x=\lambda_{\tilde{\nu}}$.  Since the variance is one,
$\int_{\R^n} x\cdot(G_{\tilde{\nu}}-\lambda_{\tilde{\nu}} x)f\,\dd x=0$.  Differentiating \eqref{eq:mf-critical-free-energy} along the flow and integrating by parts yields
\[
 \frac{\dd}{\dd\tau}\cF(\tilde{\nu}_\tau)
 =-\int_{\R^n} G_{\tilde{\nu}}\cdot(G_{\tilde{\nu}}-\lambda_{\tilde{\nu}} x)f\,\dd x
 =-\int_{\R^n}|G_{\tilde{\nu}}-\lambda_{\tilde{\nu}} x|^2f\,\dd x,
\]
which is \eqref{eq:mf-critical-free-energy-dissipation}.  The calculation is justified first with spatial cutoffs under the stated decay assumptions. In order to pass to less regular limits we would require an energy chain rule to get equality in \eqref{eq:mf-critical-free-energy-dissipation}. Lower semicontinuity alone yields only the corresponding integrated energy inequality.
\end{proof}

\begin{remark}\label{rem:critical-mf-shape}
For $p>0$, the mean-field classification at critical homogeneity is
governed by the long-time average of $D_p(\tilde{\nu}_\tau)$ in
\eqref{eq:mf-critical-log-variance}. Under the regularity and decay
assumptions of part (iii) of
\autoref{thm:critical-mf-shape}, a stationary normalized density satisfies
\[
 f(x)=\frac{1}{Z}\exp\Big(
 -\frac{2}{\sigma^2}V_{\tilde{\nu}}(x)
 +\frac{\lambda_{\tilde{\nu}}}{\sigma^2}|x|^2\Big),\quad
 \int_{\R^n}x f(x)\dd x=0,
 \qquad
 \int_{\R^n}|x|^2f(x)\dd x=1,
\]
where $\tilde{\nu}(\dd x)=f(x)\dd x$ and $Z$ is the normalizing constant.
If a trajectory exists for all intrinsic times and
$D_p(\tilde{\nu}_\tau)\to D_\infty$, then
\autoref{thm:critical-mf-shape} gives consensus when $D_\infty>c$ and
finite-time explosion in physical time when $D_\infty<c$.
Convergence to a stationary shape and the classification of the boundary
case $D_\infty=c$ remain open in the general nonlinear setting.
\end{remark}

For $p=2$, the difference between $D_2(\nu_t)$ and its sharp lower bound $4(1+1/n)s_t^2$ is the sum of two explicitly nonnegative quantities. The first measures the variation of the squared distance from the center. The second measures the deviation of the covariance from isotropy.

\begin{corollary}\label{cor:p2-defect}
Assume $p=2$ and let $\nu_t$ satisfy the hypotheses of \autoref{prop:variance-MV}. Write
\[
 \Sigma_t:=\int_{\R^n} x\otimes x\,\nu_t(\dd x),\qquad
 \operatorname{Var}_{\nu_t}(|x|^2):=\int_{\R^n} (|x|^2-s_t)^2\nu_t(\dd x),
\]
and $\|A\|_{\rm HS}^2:=\tr(A^{\mathsf T}A)$ for the Hilbert--Schmidt norm. At the critical homogeneity $q=1$, for almost every $t$,
\begin{equation}\label{eq:p2-defect}
 \begin{split}
 s_t'= &-\Big(4\Big(1+\frac1n\Big)-c\Big)s_t^2
 -2\operatorname{Var}_{\nu_t}(|x|^2)
 -4\left\|\Sigma_t-\frac{s_t}{n}\Id_n\right\|_{\rm HS}^2.
 \end{split}
\end{equation}
In particular, at the sharp coercive noise level $c=4(1+1/n)$,
\[
 -s_t'=2\operatorname{Var}_{\nu_t}(|x|^2)
 +4\left\|\Sigma_t-\frac{s_t}{n}\Id_n\right\|_{\rm HS}^2\ge0
\]
for almost every $t$. 
\end{corollary}

Note, here $-s_t'$ is precisely the instantaneous rate of decrease of the variance. At any time where \eqref{eq:p2-defect} holds, $s_t'=0$ at the sharp coercive noise level if and only if both defects vanish: $|x|^2=s_t$ for $\nu_t$-almost every $x$, and $\Sigma_t=(s_t/n)\Id_n$. Thus the equality condition is a statement about the law at that time, not a claim that a spherical law remains stationary under the evolution. At a positive-variance time when $\nu_t$ has a density, the spherical condition is impossible. 

\begin{proof}
Subtract $4(1+1/n)s^2$ from the exact identity \eqref{eq:D2-MV-identity} and use
\[
 \int_{\R^n}|x|^4\nu(\dd x)-s^2=\operatorname{Var}_{\nu}(|x|^2),
 \qquad
 \tr(\Sigma^2)-\frac{s^2}{n}
 =\left\|\Sigma-\frac{s}{n}\Id_n\right\|_{\rm HS}^2.
\]
Consequently,
\[
 D_2(\nu_t)=4\Big(1+\frac1n\Big)s_t^2
 +2\operatorname{Var}_{\nu_t}(|x|^2)
 +4\left\|\Sigma_t-\frac{s_t}{n}\Id_n\right\|_{\rm HS}^2.
\]
Inserting this formula into $s_t'=-D_2(\nu_t)+cs_t^2$ proves \eqref{eq:p2-defect}. 
\end{proof}

\begin{remark}\label{rem:closure}
For $p=0$, $D_0(\nu)=2s$ and the variance equation closes.
For $p>0$, the dissipation is generally not determined by the variance
alone. At $p=2$, \eqref{eq:p2-defect} separates the variance contribution
from radial fluctuations and covariance anisotropy, but the identity
does not provide closed evolution equations for these two defects.
\end{remark}
The superlinear attraction also yields a coercive moment estimate and instantaneous generation of centered moments.

\begin{lemma}\label{lem:moment-generation}
Let $\nu_t$ be a centered solution and write $M_k(t):=\int_{\R^n}|x|^k\nu_t(\dd x)$ for $k\ge2$.  As long as the indicated moments are finite,
\begin{equation}\label{eq:moment-diff-sharp}
 M_k'(t)
 \le-a_k M_k(t)^{1+p/k}
 +b_k M_k(t)^{1+(4q-2)/k},
\end{equation}
where
\begin{equation*}
 a_k:=k 2^{1-k},
 \qquad
 b_k:=\frac{k(k+n-2)}2\kappa^2(1-\rho^2).
\end{equation*}
Consequently, if $4q<p+2$, every initially finite centered $k$th moment is bounded uniformly on the maximal lifetime of the solution.  If in addition $p>0$, then the drift comes down from infinity: for every $k\ge2$ and every $\tau>0$ there is $C_{k,\tau}<\infty$, independent of any higher initial moment, such that
\begin{equation}\label{eq:moment-generation}
 \sup_{\tau\le t<\tau_{\star}}M_k(t)\le C_{k,\tau},
\end{equation}
where $\tau_{\star}$ is the maximal lifetime.
In particular, if $p>0$ and $4q<p+2$, initial data in $\cP_{p+2}$ acquire every finite centered moment at each positive time before $\tau_{\star}$. 
\end{lemma}

\begin{proof}
For $\Phi_k(x)=|x|^{k-2}x$, symmetrization and the strong monotonicity inequality
\[
 (\Phi_k(x)-\Phi_k(y))\cdot(x-y)\ge2^{2-k}|x-y|^k
\]
give
\begin{align*}
 k\E\Big[|Y|^{k-2}Y\cdot\int_{\R^n} K(z-Y)\nu_t(\dd z)\Big]
 &=-\frac k2\E\left[(\Phi_k(Y)-\Phi_k(\widetilde Y))\cdot K(Y-\widetilde Y)\right]\\
 &\le-k2^{1-k}\E|Y-\widetilde Y|^{k+p}\\
 &\le-k2^{1-k}M_{k+p}(t),
\end{align*}
where the last step is Jensen's inequality conditional on $Y$ and uses $\E\widetilde Y=0$. Applying It\^o to $|Y|^k$, the It\^o correction is
\[
 \frac{k(k+n-2)}2\kappa^2(1-\rho^2)s_t^{2q}M_{k-2}(t).
\]
Now taking expectations in the SDE for $|Y|^k$, and using Jensen's inequality to see
\[
 M_{k+p}\ge M_k^{1+p/k},\qquad
 s=M_2\le M_k^{2/k},\qquad
 M_{k-2}\le M_k^{(k-2)/k},
\]
proves \eqref{eq:moment-diff-sharp}.  If $4q<p+2$, the exponent of the negative term is strictly larger than that of the positive term, so the right-hand side is negative above a deterministic level. This gives the stated uniform bound for finite $M_k(0)$.

When $p>0$, the exponent $1+p/k$ for the negative term is strictly greater than one.  Above a sufficiently large level $R_k$, \eqref{eq:moment-diff-sharp} implies
\[
 M_k'\le-\frac{a_k}{2}M_k^{1+p/k}.
\]
Comparison with this scalar ODE gives, for a finite initial $k$-th moment,
\begin{equation}\label{eq:moment-bound}
  M_k(t)\le \max\Big\{R_k,
 \Big(\frac{2k}{a_k p\,t}\Big)^{k/p}\Big\}.   
\end{equation}
To justify the estimate for initial laws that need not have a finite $k$-th moment, approximate the centered initial law in $\cW_{p+2}$ by  centered laws with bounded support, but keep the original interaction $K(z)=|z|^pz$.   The local construction in \autoref{lem:local-MV}, whose proof uses only the $(p+2)$-moment, gives the corresponding solutions on every compact admissible radius window.  For each approximation, apply the preceding calculation to smooth regularizations of $|x|^k$ and then remove the state-space localization.  The scalar comparison bound \eqref{eq:moment-bound} for times $t\ge\tau>0$ is uniform in the approximating initial $k$-th moment.  The uniform $(p+2)$-moment estimate gives compactness, monotonicity and local pathwise uniqueness identify the limit on each such window, and Fatou's lemma transfers the bound to the limiting solution.  Restarting over successive radius windows proves \eqref{eq:moment-generation} up to the maximal lifetime.  A branch that remains at consensus trivially satisfies the estimate. After any departure, the same argument starts on a positive-radius window. 
\end{proof}

We shall also use the fact that a positive-variance mean-field solution cannot collapse to a Dirac mass at a finite time while the solution remains regular.

\begin{lemma}\label{lem:mf-nonattainment}
Let $\nu_t$ be the deterministic centered law of a strong solution of \eqref{eq:centeredMV} on $[0,T]$, with $\sup_{t\le T}\int_{\R^n}|x|^{p+2}\nu_t(\dd x)<\infty$.  If $s_0>0$, then
\begin{equation}\label{eq:mf-finite-time-positive}
 s_t>0\qquad\text{for every }0\le t\le T.
\end{equation}
By continuity, therefore, $\inf_{0\le t\le T}s_t>0$.
\end{lemma}

\begin{proof}
For the prescribed deterministic law path $\nu_t$, write
\[
 b_t(x):=\int_{\R^n} K(y-x)\,\nu_t(\dd y),
 \qquad
 \sigma_t:=\kappa\sqrt{1-\rho^2}\,s_t^q.
\]
The map $x\mapsto b_t(x)$ is locally Lipschitz, locally uniformly in time after localization: indeed
$|DK(z)|\le C_p|z|^p$ for $p>0$ (and $K(z)=z$ for $p=0$), while the assumed moment bound controls $\int|y|^p\nu_t(\dd y)$ on compact time intervals.

On a product probability space take two independent initial variables $Y_0,Z_0$ with law $\nu_0$, but drive the two equations with the same Brownian motion $B$.  Both marginals solve
\[
 \dd Y_t=b_t(Y_t)\,\dd t+\sigma_t\,\dd B_t.
\]
After subtracting the common additive path $M_t:=\int_0^t\sigma_u\,\dd B_u$, the processes $\widetilde Y:=Y-M$ and $\widetilde Z:=Z-M$ solve the same nonautonomous random ordinary differential equation
\[
 \dot y_t=b_t(y_t+M_t).
\]
The random ODE has an injective flow.  Indeed, suppose two trajectories meet at some $t_0\le T$.  By continuity they remain in a ball $B_R$ on $[0,t_0]$ for some finite $R$.  On that ball,
\[
 |b_t(x)-b_t(z)|\le L_R(t)|x-z|,
 \qquad
 L_R(t)\le C_{p,R}\Big(1+\int_{\R^n} |y|^p\nu_t(\dd y)\Big),
 \qquad \int_0^{t_0}L_R(t)\,\dd t<\infty.
\]
Since the additive noises are identical and $Y_{t_0}=Z_{t_0}$,
\[
 |Y_t-Z_t|
 \le \int_t^{t_0} L_R(u)|Y_u-Z_u|\,\dd u,
 \qquad 0\le t\le t_0.
\]
Backward Gr\"onwall gives $Y_t=Z_t$ on $[0,t_0]$, hence $Y_0=Z_0$.  Therefore, pathwise,
\[
 Y_0\ne Z_0\quad\Rightarrow\quad Y_t\ne Z_t
 \qquad\text{for every }0\le t\le T.
\]
Since $s_0>0$, the law $\nu_0$ is not a Dirac mass and hence $\P(Y_0\ne Z_0)>0$.  If $s_t=0$ at some finite time, then $\nu_t=\delta_0$, so every coupling of two $\nu_t$-distributed random variables would satisfy $Y_t=Z_t$ almost surely, a contradiction.  This proves \eqref{eq:mf-finite-time-positive}. Continuity of $s$ yields the compact-time lower bound. 
\end{proof}

The following local construction gives the approximation step used in the classification theorem.

\begin{lemma}\label{lem:local-MV}
Let $h=p+2$ and let the centered initial law belong to $\cP_h(\R^n)$, and let $J\subset[0,\infty)$ be a compact neighborhood of the initial radius $r_0=s_0^{1/2}$.  Suppose either $\rho=1$, or $q=0$, or $r\mapsto r^{2q}$ is Lipschitz on $J$.  If $q\ge1/2$, $J$ may contain $0$, and if $0<q<1/2$, require $J\Subset(0,\infty)$.  Then the centered equation \eqref{eq:centeredMV} has a unique strong solution up to the first exit of $r_t=s_t^{1/2}$ from $J$.  The solution can be continued as long as the radius remains in a compact set on which $r^{2q}$ is Lipschitz and its $h$-th moment stays finite.
\end{lemma}

\begin{proof}
For $\nu\in\cP_2(\R^n)$, set
\[
 \mathcal R(\nu):=\Big(\int_{\R^n}|x-\bar\nu|^2\nu(\dd x)\Big)^{1/2},
 \qquad \bar\nu:=\int_{\R^n} x\nu(\dd x).
\]
Then, for $\nu,\widetilde\nu\in\cP_2(\R^n)$,
\begin{equation}\label{eq:radius-W2-local}
 |\mathcal R(\nu)-\mathcal R(\widetilde\nu)|\le\cW_2(\nu,\widetilde\nu).
\end{equation}
Indeed, for any coupling $(U,V)$ of $\nu,\widetilde\nu$, orthogonal projection onto centered random variables gives
$|\|U-\E U\|_2-\|V-\E V\|_2|\le\|U-V\|_2$, and taking the infimum over couplings proves \eqref{eq:radius-W2-local}.

If $p=0$, then $K(z)=z$ and the interaction field $(x,\nu)\mapsto\int_{\R^n}(z-x)\nu(\dd z)$ is globally Lipschitz with respect to $|\cdot|+\cW_2$.  Extend $r\mapsto\kappa\sqrt{1-\rho^2}\,r^{2q}$ from $J$ to a bounded globally Lipschitz function $g_J$ on $[0,\infty)$, with $g_J\equiv0$ when $\rho=1$ and $g_J\equiv\kappa\sqrt{1-\rho^2}$ when $q=0$.  By \eqref{eq:radius-W2-local}, the diffusion coefficient $\nu\mapsto g_J(\mathcal R(\nu))$ is globally Lipschitz in $\cW_2$.  Standard McKean--Vlasov Picard iteration therefore gives the claimed unique strong solution and continuation directly.

It remains to consider $p>0$, so $h=p+2>2$.  For $R\ge1$ define
\[
 K_R(z):=\min\{|z|^p,R^p\}z.
\]
The map $K_R$ is odd, monotone, and globally Lipschitz.  Extend $r\mapsto\kappa\sqrt{1-\rho^2}\,r^{2q}$ from $J$ to a bounded globally Lipschitz function $g_J$ on $[0,\infty)$. When $\rho=1$ take $g_J\equiv0$, and when $q=0$ take $g_J\equiv\kappa\sqrt{1-\rho^2}$.

Consequently, the truncated coefficients
\[
 b_R(x,\nu):=\int_{\R^n} K_R(z-x)\nu(\dd z),
 \qquad
 \sigma_J(\nu):=g_J(\mathcal R(\nu))
\]
are globally Lipschitz in $(x,\nu)$ with respect to $|\cdot|+\cW_2$.  Picard iteration in $C([0,T];\cP_2)$, coupled synchronously to the driving Brownian motion, therefore gives a unique strong solution $Y^R$ for every $R$.

The estimates needed to remove the drift truncation are uniform in $R$.  Put
$H_h(y)=(1+|y|^2)^{h/2}$.  Since $h=p+2\ge2$, the function $H_h$ is convex.  Since $K_R(z)=a_R(|z|)z$ with $a_R\ge0$, symmetrization and convexity of $H_h$ show that the interaction contribution to $\frac{\dd}{\dd t}\E H_h(Y_t^R)$ is nonpositive.  Moreover
$\|D^2H_h(y)\|\le C_h(1+|y|^{h-2})$, while $g_J$ is bounded.  Hence
\begin{equation}\label{eq:local-uniform-hmoment}
 \sup_R\sup_{t\le T}\E H_h(Y_t^R)\le C_{T,J}\bigl(1+\E H_h(Y_0)\bigr).
\end{equation}
Also $|K_R(z)|\le|z|^{p+1}$, so with $a=h/(p+1)>1$,
\[
 \sup_R\E\int_0^T|b_R(Y_t^R,\Law(Y_t^R))|^a\dd t<\infty.
\]
Together with the bounded diffusion coefficient, this gives tightness of the laws of $Y^R$ in $C([0,T];\R^n)$ by Aldous' tightness criterion \cite{Aldous1978}.  The moment bound \eqref{eq:local-uniform-hmoment} gives uniform integrability of the nonlinear drift.  
The same bound also gives uniform integrability of the squared norm.  Therefore, along every convergent subsequence, local uniform convergence $K_R\to K$ and truncation of the region $|Y|+|\widetilde Y|>L$, followed by $L\to\infty$, identify the limiting drift as
$\int K(z-Y_t)\nu_t(\dd z)$.  Weak convergence together with uniform integrability of the second moments yields convergence in $\cW_2$ at each time, so \eqref{eq:radius-W2-local} identifies the diffusion coefficient.  Thus every subsequential limit solves the centered McKean--Vlasov equation with the extended coefficient $g_J$.

It remains to identify the limit and obtain a strong solution.  If $Y$ and $Z$ are two solutions driven by the same Brownian motion, symmetrization gives a nonpositive interaction contribution to $\frac{\dd}{\dd t}\E|Y_t-Z_t|^2$, while \eqref{eq:radius-W2-local} and the Lipschitz property of $g_J$ give
\[
 \frac{\dd}{\dd t}\E|Y_t-Z_t|^2\le C_J\E|Y_t-Z_t|^2.
\]
Hence pathwise uniqueness holds.  Applying the Gy\"ongy--Krylov argument \cite{GyongyKrylov1996} to the joint laws of $(Y^R,B)$ shows that $Y^R$ converges in probability on the original stochastic basis. Its limit is adapted to the given Brownian filtration and is therefore a strong solution.  Until the radius exits $J$, the extension $g_J$ agrees with the original coefficient.  Restarting the same construction at deterministic times proves the continuation statement. 
\end{proof}

These estimates lead to the mean-field well-posedness classification.

\begin{theorem}\label{thm:MV-allq} 
Let $\mu_0\in\cP_{p+2}(\R^n)$ be deterministic, and set
\[
 s_0:=\frac12\iint_{\R^n\times\R^n}|x-y|^2\,\mu_0(\dd x)\mu_0(\dd y),
 \qquad c=n\kappa^2(1-\rho^2).
\]
Let $\bar p =(p+2)/2$, let $\cA_{p,n}$ be defined by \eqref{eq:Aeff-MV}, and let $\cA_{p,n}^{\rMF,\sharp}$ be the sharp variational constant in \eqref{eq:sharp-mf-shape-constant}.

\begin{enumerate}[label=\textnormal{(\roman*)}]
\item\label{item:mf-common} \textnormal{(Pure common noise $\rho=1$.)}  For every $q\ge0$ there is a unique global strong solution in the finite-$(p+2)$-moment class.

\item\label{item:mf-constant} \textnormal{(Constant noise $q=0$.)}  For every $\rho\in[0,1]$ there is a unique global strong solution.

\item\label{item:mf-holder} \textnormal{(The non-Lipschitz range $0<q<1/2$ away from consensus.)}  Assume $\rho<1$ and $s_0>0$.  Then there is a unique global strong solution, its centered conditional law is deterministic, and there exists a deterministic constant $\underline s>0$ such that
\begin{equation}\label{eq:mf-positive-lower}
 \inf_{t\ge0}s_t\ge\underline s>0.
\end{equation}
No additional moment assumption is required.

\item\label{item:mf-lipschitz} \textnormal{(The locally Lipschitz range $q\ge1/2$.)}  There is a unique maximal strong solution up to its explosion time.  It is global if at least one of the following holds:
\begin{align*}
 &2q< \bar{p},\\
 &2q=\bar{p}\quad\text{and}\quad \bigl(p=0\ \text{or}\ c\le\cA_{p,n}^{\rMF,\sharp}\bigr),\\
 &2q>\bar{p}\quad\text{and}\quad 0\le s_0\le s_{\rm c}:=
 \Big(\frac{\cA_{p,n}^{\rMF,\sharp}}{c}\Big)^{1/(2q- \bar{p})},\qquad(c>0).
\end{align*}
In the last case $s_t\le s_{\rm c}$ for all $t\ge0$.

\item\label{item:mf-consensus} \textnormal{(Consensus data.)}  If $s_0=0$ and $q\ge1/2$, the centered solution is uniquely $\nu_t=\delta_0$.  If $q=0$ and $\rho<1$, constant idiosyncratic noise immediately spreads a Dirac mass.  For $0<q<1/2$, uniqueness for arbitrary consensus data fails in general: for $p=0$, \autoref{thm:nonunique} gives infinitely many solutions from a Dirac mass.
\end{enumerate}
\end{theorem}

\begin{proof}
Except for the pure-common-noise case, which is a deterministic relative equation and is treated directly below, local strong existence and pathwise uniqueness follow from \autoref{lem:local-MV}.  In the H\"older range $0<q<1/2$ with $s_0>0$, choose the initial radius window compactly inside $(0,\infty)$. For $q=0$ or $q\ge1/2$ no lower cutoff is needed.  The continuation statement in that lemma shows that the only possible obstructions are loss of the required positive-variance window when $0<q<1/2$ and divergence of the $(p+2)$-moment.  The estimates below exclude precisely these obstructions in the regimes claimed in the theorem.

Part~\ref{item:mf-common} follows from \autoref{thm:translation}, the centered equation has no stochastic term.  Monotonicity of $K$ yields global moment control and synchronous coupling gives uniqueness. The center is then the unique stochastic integral in \eqref{eq:mSDE}.  For part~\ref{item:mf-constant}, the diffusion coefficient is constant and \eqref{eq:sbarrier-sharp} gives $s_t'\le-\cA_{p,n}^{\rMF,\sharp}s_t^{\bar{p}}+c$.  Hence the second moment is bounded on bounded intervals, while the same moment calculation as in \autoref{lem:moment-generation} controls the $(p+2)$-moment.  This rules out explosion and proves global uniqueness.

Consider part~\ref{item:mf-holder}.  Since $s_0>0$, \autoref{lem:local-MV} gives local existence and uniqueness until a hypothetical first time at which $s$ reaches zero or the $(p+2)$-moment diverges.  Because $2q<1\le \bar{p}$, \eqref{eq:sbarrier-sharp} implies a global scalar upper bound on $s$ and \autoref{lem:moment-generation} controls the initially finite $(p+2)$-moment, so finite-time moment explosion is impossible.

If $p=0$, then $D_0(\nu)=2s$ and
\[
 s'=-2s+cs^{2q}.
\]
For $0<q<1/2$ the positive term dominates near zero, and the positive solution started at $s_0>0$ never reaches zero. Indeed, it stays above $\min\{s_0,(c/2)^{1/(1-2q)}\}$.

Now suppose $p>0$ and argue by contradiction.  If $T_0<\infty$ were the first time with $s_t\to0$, fix $\tau=T_0/2$.  Choose $m> (p+2-4q)/(1-2q)$, we have
\[
 \sup_{\tau\le t<T_0}\int_{\R^n}|x|^m\nu_t(\dd x)<\infty,
\]
since $4q<p+2$ and by applying \autoref{lem:moment-generation}.
For $k=p+2$, let $\eta:=(m-k)/(m-2) > 2q$, then interpolation between $L^2$ and $L^m$ yields
\begin{equation*}
 \int_{\R^n}|x|^k\nu_t(\dd x)\le C_{\tau,T_0}s_t^{\eta},
 \qquad
D_p(\nu_t)\le2^k\int_{\R^n}|x|^k\nu_t(\dd x)
 \le C_{\tau,T_0}s_t^{\eta}.
\end{equation*}
Hence, for all sufficiently small $s_t$,
\[
 s_t'=-D_p(\nu_t)+cs_t^{2q}
 \ge-C_{\tau,T_0}s_t^{\eta}+cs_t^{2q}
 \ge\frac c2s_t^{2q}>0,
\]
which prevents a positive trajectory from crossing into zero.  This contradiction proves positivity for all time.

To upgrade positivity to the uniform bound \eqref{eq:mf-positive-lower}, apply \autoref{lem:moment-generation} with the same $m$ on $[1,\infty)$.  The moment bound there is uniform in time, so the preceding interpolation gives $D_p(\nu_t)\le C s_t^{\eta}$ for all $t\ge1$ with a fixed $C$.  Choose $\delta>0$ so that $Cs^{\eta}\le(c/2)s^{2q}$ on $(0,\delta]$.  Then $s'\ge(c/2)s^{2q}$ whenever $t\ge1$ and $s\le\delta$, and therefore
\[
 \inf_{t\ge1}s_t\ge\min\{s_1,\delta\}>0.
\]
Continuity and positivity on $[0,1]$ complete the proof of \eqref{eq:mf-positive-lower}.

For uniqueness in part~\ref{item:mf-holder}, let $Y,Z$ be two centered solutions driven by the same Brownian motion and with the same initial variable. Symmetrization with an independent copy of the joint coupling gives
\begin{align*}
 &\E (Y-Z)\cdot\left(\int_{\R^n} K(y-Y)\nu_Y(\dd y)-\int_{\R^n} K(y-Z)\nu_Z(\dd y)\right)\\
 &\qquad=-\frac12\E\left[((Y-\widetilde Y)-(Z-\widetilde Z))\cdot
 \bigl(K(Y-\widetilde Y)-K(Z-\widetilde Z)\bigr)\right]\le0.
\end{align*}
Apply the lower-barrier argument above separately to each of the two solutions.  Since they have the same positive initial variance, both $s_Y$ and $s_Z$ stay in compact subsets of $(0,\infty)$ on every bounded time interval.  Consequently
\[
 |s_Y^q-s_Z^q|\le C_T|s_Y-s_Z|
 \le C_T(\E|Y-Z|^2)^{1/2}.
\]
It\^o's formula gives $\frac{\dd}{\dd t}\E|Y_t-Z_t|^2\le C_T\E|Y_t-Z_t|^2$, hence $Y-Z\equiv0$ by Gr\"onwall's inequality.

For part~\ref{item:mf-lipschitz}, the map from the centered $L^2$ radius to the amplitude is $r\mapsto\kappa r^{2q}$, which is locally Lipschitz at $r=0$ exactly when $2q\ge1$.  Thus the maximal strong solution is unique up to explosion.  By \eqref{eq:sbarrier-sharp},
\[
 s_t'\le f(s_t):=-\cA_{p,n}^{\rMF,\sharp}s_t^{\bar{p}}+cs_t^{2q}.
\]
If $2q<\bar{p}$, then $f(s)\le-a s^{\bar{p}}+C$. If $2q=\bar{p}$ and $p=0$, then $q=1/2$ and the closed equation $s'=(-2+c)s$ is global (although $s$ may grow exponentially when $c>2$).  If $2q=\bar{p}$, $p>0$, and $c\le\cA_{p,n}^{\rMF,\sharp}$, then $f\le0$.  In every case just listed, $s$ is bounded on each finite time interval. Once such a bound is fixed, retaining the factor $s^{2q}$ in the exact $k$th-moment It\^o formula gives for $k=p+2$,
\[
 M_k'\le-a_k M_k^{1+p/k}+C_T M_k^{(k-2)/k},
\]
so the $(p+2)$-moment is bounded and explosion is impossible.  If $2q>\bar{p}$ and $s_0\le s_{\rm c}$, the scalar solution of $\dot u=f(u)$ starting from $s_0$ remains in $[0,s_{\rm c}]$.  Scalar comparison gives $s_t\le u_t\le s_{\rm c}$, and the same bounded-variance moment estimate gives global continuation.  The case $s_0=0$ is included by uniqueness at the origin.

Part~\ref{item:mf-consensus} follows from the same local-Lipschitz argument when $q\ge1/2$, from the nonzero constant diffusion when $q=0$, and from the explicit $p=0$ construction in \autoref{thm:nonunique} for the failure of the blanket uniqueness when $0<q<1/2$. 
\end{proof}

\begin{remark}\label{rem:mf-consensus-selection-general}
The singular initial condition $s_0=0$ should be distinguished from the positive-variance theory.  For the finite particle system, part~\ref{item:particle-consensus} of \autoref{thm:particle-allq} proves nonuniqueness in law at exact consensus for $m_N\ge2$ for every $p\ge0$ when $0<q<1/2$. The higher-order attractive drift does not remove the local entrance-law mechanism.  At the mean-field level, \autoref{sec:p0} gives a complete waiting-time family when $p=0$.  For $p>0$, constructing and classifying nontrivial entrance laws from $\delta_0$ is a separate nonlinear problem and is not addressed here.

This distinction causes no ambiguity in \autoref{thm:poc}: that theorem assumes a unique limiting solution on the time interval in question.  Hence it applies directly to every positive-variance initial law covered by part~\ref{item:mf-holder} of \autoref{thm:MV-allq}, but it does not assert propagation of chaos from a nonunique consensus datum without first specifying a selection.  When $p=0$, \autoref{rem:selection} shows explicitly that exact microscopic consensus and vanishing positive-variance regularization select different mean-field branches.
\end{remark}

In the H\"older regime, the preceding argument yields a uniform quantitative separation from consensus.

\begin{corollary}\label{cor:noncollapse-general}
Assume $\rho<1$, $\mu_0\in\cP_{p+2}(\R^n)$, and $s_0>0$.  If $0\le q<1/2$, then the unique positive-variance mean-field solution is global and
\begin{equation*}
 0<\underline s\le s_t\le\overline s<\infty,
 \qquad t\ge0,
\end{equation*}
for deterministic constants $\underline{s}, \overline{s}$ depending on the data and parameters. In particular, the mean-field system cannot reach consensus from positive variance.  If $q=0$ and $s_0=0$, constant idiosyncratic noise instantaneously spreads the Dirac law: $s_t>0$ for every $t>0$, and for each $\tau>0$ there is $\underline s_\tau>0$ such that $\inf_{t\ge\tau}s_t\ge\underline s_\tau$.
\end{corollary}

\begin{proof}
For $0<q<1/2$ the lower bound is \eqref{eq:mf-positive-lower}. The upper bound follows from scalar comparison in \eqref{eq:sbarrier-sharp}.  For $q=0$ and $s_0>0$, the same proof applies with $cs^{2q}=c$: when $p>0$, use moment generation and interpolation with any exponent $\eta>0$, while for $p=0$ the closed equation is $s'=-2s+c$.

If $q=0$ and $s_0=0$, the exact identity \eqref{eq:s-exact} gives $s'(0)=c>0$, so $s_t>0$ for all sufficiently small $t>0$.  Once positive, \autoref{lem:mf-nonattainment} prevents a finite-time return to zero.  Applying the preceding positive-variance lower-barrier argument from time $\tau/2$ onward gives $\inf_{t\ge\tau}s_t>0$ for every $\tau>0$. 
\end{proof}

\section{Conditional propagation of chaos}\label{sec:poc}

Propagation of chaos originates in the classical mean-field theory of McKean, Sznitman, and M\'el\'eard \cite{McKean1966,Meleard1996,Sznitman}. Quantitative approaches for interacting diffusions include \cite{BCC2011,Lacker2023,LackerLeFlem2023}, while common-noise variants are developed in \cite{CoghiFlandoli2016,Nikolaev2025}.  We next compare the interacting system with conditionally independent copies of the McKean--Vlasov solution.

Throughout this section, whenever $X^{1,N},\ldots,X^{N,N}$ are coupled synchronously with conditionally independent McKean--Vlasov copies $\hX^1,\ldots,\hX^N$, we define

\[
\mu_t^N:=\frac1N\sum_{i=1}^N\delta_{X_t^{i,N}}, \quad \text{ and }
\]
\begin{equation}\label{def:eEN}
  e_i(t):=X_t^{i,N}-\hX_t^i,
 \qquad
 \cE_N(t):=\frac1N\sum_{i=1}^N|e_i(t)|^2.   
\end{equation}

\begin{theorem}\label{thm:poc}
Let $\mu_0\in\cP_{p+2}(\R^n)$ be deterministic and let $X_0^1,X_0^2,\dots$ be i.i.d. with law $\mu_0$, independent of all Brownian motions.  For each $N$, let $(X^{1,N},\ldots,X^{N,N})$ denote the solution of \eqref{eq:particle} with initial data $(X_0^1,\ldots,X_0^N)$.  Fix $T<\infty$ and suppose that the mean-field equation \eqref{eq:MV} has a unique nonexplosive strong solution on $[0,T]$.  Assume in addition that one of the following regularity alternatives holds:
\begin{enumerate}[label=\textnormal{(\roman*)}]
\item\label{item:poc-constant} $q=0$,
\item\label{item:poc-lipschitz} $q\ge1/2$,
\item\label{item:poc-holder} $0<q<1/2$ and $\inf_{0\le t\le T}s_t>0$.
\end{enumerate}
Then, as $N\to\infty$,
\begin{equation}\label{eq:particle-exit-before-mf}
 \P(\tau_N^{\rm expl}\le T)\rightarrow0,
\end{equation}
and, for every $\e>0$, as $N\to\infty$,
\begin{equation}\label{eq:pocW2}
 \P\Big(
 \tau_N^{\rm expl}>T,
 \ \sup_{0\le t\le T}\cW_2(\mu_t^N,\mu_t)>\e
 \Big)\rightarrow 0.
\end{equation}
Equivalently, with any arbitrary post-explosion extension of $\mu^N$, as $N\to\infty$,
$\sup_{t\le T}\cW_2(\mu_t^N,\mu_t)\to0$ in probability.  Thus the result applies on every interval strictly before a mean-field blow-up time.
\end{theorem}

Note, for $\rho<1$ and $s_0>0$, condition~\ref{item:poc-holder} follows from \autoref{thm:MV-allq}. For $\rho=1$ and $s_0>0$ it follows directly from the deterministic relative equation on every finite interval.  For the finite system, use the absorbing consensus convention stated in \autoref{sec:allq-particle} whenever a boundary selection is needed.

\begin{proof}
Let $\hX^1,\dots,\hX^N$ be conditionally i.i.d. copies of the McKean--Vlasov solution, driven by the same common Brownian motion $B^0$, the same idiosyncratic Brownian motions $B^i$, and with the same initial variables $X_0^i$.  Define
\[
\widehat\mu_t^N:=\frac1N\sum_{i=1}^N\delta_{\hX_t^i}.
\]
For the interaction error of the $i$-th component, write
\begin{align*}
 A_i&:=\frac1N\sum_{k=1}^N\left[K(X^{k,N}-X^{i,N})-K(\hX^k-\hX^i)\right],\\
 \Xi_i&:=\frac1N\sum_{k=1}^NK(\hX^k-\hX^i)
 -\int_{\R^n} K(y-\hX^i)\,\mu_t(\dd y).
\end{align*}
Thus, $\Xi_i$ denotes the empirical interaction-field fluctuation, namely the difference between the conditionally empirical interaction field at $\hX^i$ and its conditional McKean--Vlasov mean.
Monotonicity gives the exact sign
\begin{align}\label{eq:monotone-poc}
 \frac1N\sum_{i=1}^N e_i\cdot A_i
 &=-\frac1{2N^2}\sum_{i,k=1}^N
 \left[(X^{k,N}-X^{i,N})-(\hX^k-\hX^i)\right]\nonumber\\
 &\hspace{22mm}\cdot
 \left[K(X^{k,N}-X^{i,N})-K(\hX^k-\hX^i)\right]\le0,
\end{align}
which follows from the monotonicity of $K$. Thus no polynomial Lipschitz bound for the interacting particle drift is needed.

As $T$ lies strictly inside the lifetime of the mean-field solution, the variance is bounded above on $[0,T]$.  If $s_0>0$, \autoref{lem:mf-nonattainment} also gives $\inf_{t\le T}s_t>0$. If $s_0=0$ in the locally Lipschitz regime $q\ge1/2$, uniqueness gives $s_t\equiv0$.  Put $h=p+2\ge2$.  The attractive drift is dissipative, and the same symmetrized moment calculation used in \autoref{lem:moment-generation}, now retaining the bounded variance in the It\^o correction, gives
\begin{equation}\label{eq:hat-moment-poc}
 \sup_{t\le T}\E|\hX_t|^{h}<\infty.
\end{equation}
We now fix the stopping interval used throughout the coupling argument.  For
\[
 r_N(x):=\Big(\frac1N\sum_{i=1}^N|x^i-\bar x|^2\Big)^{1/2},
 \qquad \widehat S_t^N:=r_N(\hX_t)^2,
\]
choose compact intervals $J_{\rm in}\Subset J_{\rm out}\subset[0,\infty)$, with compact containment understood in the relative topology of $[0,\infty)$, such that
$r_t=s_t^{1/2}\in\operatorname{int}_{[0,\infty)}J_{\rm in}$ for every $t\le T$ and such that
$r\mapsto r^{2q}$ is Lipschitz on $J_{\rm out}$.  In cases~\ref{item:poc-constant}--\ref{item:poc-lipschitz} the intervals may have lower endpoint zero. In case~\ref{item:poc-holder} they are contained in $(0,\infty)$.  Let
\begin{equation}\label{eq:poc-stopping-time}
 \tau_N:=\inf\bigl\{t\ge0:r_N(X_t^N)\notin J_{\rm out}
 \text{ or }r_N(\hX_t)\notin J_{\rm out}\bigr\},
\end{equation}
with the convention $\inf\varnothing=\infty$.

We record the stopped empirical moment estimate rather than merely invoking it.  Define
\[
Y_t^{i,N}:=X_t^{i,N}-\bar X_t^N, \qquad H_h(y)=(1+|y|^2)^{h/2},\quad \text{ and}
\]
\[
 \mathcal H_h^N(t):=\frac1N\sum_{i=1}^N(1+|Y_t^{i,N}|^2)^{h/2} = \frac1N\sum_{i=1}^N H_h(Y^{i,N}_t).
\]
Applying It\^o's formula to $\mathcal H_h^N(t)$ we show that, up to $\tau_N$, the diffusion amplitude is bounded. By symmetrization using $K$ and convexity of $H_h$, the interaction contribution in It\^o's formula is
\[
 -\frac1{2N^2}\sum_{i,k=1}^N
 [\nabla H_h(Y^{i,N})-\nabla H_h(Y^{k,N})]
 \cdot K(Y^{i,N}-Y^{k,N})\le0.
\]
Since $\|D^2H_h(y)\|\le C_h(1+|y|^{h-2})$, the It\^o correction is bounded by
$C_{T,J_{\rm out}}(1+\mathcal H_h^N)$.  Localization, expectation, and Gr\"onwall's inequality therefore give
\begin{equation*}
 \sup_N\sup_{t\le T}\E\mathcal H_h^N(t\wedge\tau_N)<\infty.
\end{equation*}
The center equation has bounded quadratic variation up to $\tau_N$. Jensen's inequality for the empirical initial center and the Burkholder--Davis--Gundy inequality give
\[
 \sup_N\E\sup_{t\le T}|\bar X_{t\wedge\tau_N}^N|^h<\infty.
\]
Using $|X^{i,N}|^h\le C_h(|Y^{i,N}|^h+|\bar X^N|^h)$, we obtain the asserted stopped empirical estimate in the stronger form
\begin{equation}\label{eq:particle-moment-poc}
 \sup_N\sup_{t\le T}\E\Big[\frac1N\sum_{i=1}^N
 |X_{t\wedge\tau_N}^{i,N}|^h\Big]<\infty.
\end{equation}
Together with \eqref{eq:hat-moment-poc}, these are the only moments needed below.

To estimate the empirical interaction-field  fluctuation $\Xi_i$ without assuming a $2p+2$ moment, fix $L>1$ and write
\[
 K=K_L+K_L^{\rm tail},\qquad
 K_L(z):=K(z)\one_{\{|z|\le L\}}.
\]
We define the truncated and tail fluctuations by
\[
 \Xi_i^L
 :=\frac1N\sum_{k=1}^NK_L(\hX^k-\hX^i)
 -\int_{\R^n} K_L(y-\hX^i)\mu_t(\dd y),
\]
\[
 \Xi_i^{L,\mathrm{tail}}
 :=\frac1N\sum_{k=1}^N K_L^{\rm tail}(\hX^k-\hX^i)
 -\int_{\R^n} K_L^{\rm tail}(y-\hX^i)\,\mu_t(\dd y),
 \qquad
 \Xi_i=\Xi_i^L+\Xi_i^{L,\mathrm{tail}}.
\]
Conditionally on $(\cF_t^0,\hX_t^i)$, the off-diagonal summands
\[
 K_L(\hX_t^k-\hX_t^i)
 -\int_{\R^n} K_L(y-\hX_t^i)\,\mu_t(\dd y),
 \qquad k\ne i,
\]
are independent and centered. The $k=i$ term produces the bias
$-N^{-1}\int_{\R^n} K_L(y-\hX_t^i)\mu_t(\dd y)$, whose square is $O(N^{-2})$. Since $K_L$ is bounded,
\begin{equation}\label{eq:Rbound-truncated}
 \frac1N\sum_{i=1}^N\E|\Xi_i^L(t)|^2\le\frac{C_L}{N}.
\end{equation}
Let $h'=h/(h-1)=(p+2)/(p+1)$.  Jensen's inequality for the empirical average and $|K(z)|^{h'}=|z|^h$ give
\begin{equation*}
 \int_0^T\frac1N\sum_{i=1}^N \E|\Xi_i^{L,\mathrm{tail}}(t)|^{h'}\dd t
 \le C\int_0^T\E\left[|\hX_t-\widetilde X_t|^h
 \one_{\{|\hX_t-\widetilde X_t|>L\}}\right]\dd t
 =:\e_L^{h'},
\end{equation*}
where $\widetilde X$ is a conditionally independent copy and $\e_L\to0$ as $L\to\infty$ by dominated convergence in spacetime and \eqref{eq:hat-moment-poc}.  Combining \eqref{eq:particle-moment-poc}, \eqref{eq:hat-moment-poc}, and H\"older's inequality yields
\begin{equation}\label{eq:Rtail-energy-poc}
 \E\int_0^{T\wedge\tau_N}\frac1N\sum_{i=1}^N
 |e_i(t)|\,|\Xi_i^{L,\mathrm{tail}}(t)|\dd t\le C_T\e_L.
\end{equation}
Thus the empirical interaction-field fluctuation is controlled first at fixed $L$ with an $O(N^{-1})$ variance estimate and then by sending $L\to\infty$. No moment beyond $p+2$ is required.

For the diffusion amplitudes, orthogonal projection onto the zero-sum subspace has norm one, so for all 
$t\leq T$,
\begin{equation}\label{eq:rLip}
 |r_N(X_t)-r_N(\hX_t)|
 \le\Big(\frac1N\sum_{i=1}^N|X^i_t-\hX^i_t|^2\Big)^{1/2}
 =\cE_N^{1/2}(t).
\end{equation}
The noise amplitude is $\kappa r_N(X_t^N)^{2q}=\kappa(S_t^N)^q$.  In cases~\ref{item:poc-constant}--\ref{item:poc-lipschitz}, the map $r\mapsto r^{2q}$ is either constant or locally Lipschitz on $[0,\infty)$.  In case~\ref{item:poc-holder}, the limiting radius $r_t=s_t^{1/2}$ stays a positive distance from zero, so the same map is Lipschitz on a compact neighborhood of the limiting radius path.

By the choice of $J_{\rm out}$ in \eqref{eq:poc-stopping-time}, the map $r\mapsto r^{2q}$ has a deterministic Lipschitz constant on the stopped radius range.  Hence, up to $\tau_N$,
\begin{equation}\label{eq:amp-local}
 |(S_t^N)^q-(\widehat S_t^N)^q|^2\le C_T\cE_N(t),
\end{equation}
where $\widehat S_t^N=r_N(\hX_t)^2$.

Conditionally on $\cF_T^0$, the paths $\hX^i(\cdot)$ are i.i.d. in the Polish space
$E_T:=C([0,T];\R^n)$ equipped with the supremum norm.  Moreover, monotonicity gives $y\cdot b_{\nu_t}(y)\le |y|\,|b_{\nu_t}(0)|$, where $|b_{\nu_t}(0)|\le\int_{\R^n}|z|^{p+1}\nu_t(\dd z)$ is uniformly bounded on $[0,T]$. Applying It\^o's formula to $|\widehat Y|^2$, followed by the Burkholder--Davis--Gundy, Young and Gr\"onwall inequalities, controls the centered path supremum. Adding the center martingale therefore gives $\E\|\hX\|_{E_T}^2<\infty.$
Let
\[
 \widehat{\mathbf P}_N:=\frac1N\sum_{i=1}^N\delta_{\hX^i(\cdot)},
 \qquad
 \mathbf P^0:=\Law(\hX(\cdot)\mid\cF_T^0).
\]
By \autoref{thm:translation}, conditionally i.i.d. paths have the form $\hX^i=m+\widehat Y^i$, where the centered paths $\widehat Y^i$ are i.i.d. and independent of the common-noise path $m$.  Translation by the same path is an isometry of $E_T$.  Hence the ordinary strong law for empirical measures gives weak convergence of $\widehat{\mathbf P}_N$ to $\mathbf P^0$ almost surely, and the scalar strong law applied to $\|\widehat Y^i\|_{E_T}^2$ gives convergence of the centered second moments.  The standard characterization of Wasserstein convergence on a Polish space \cite{Villani} therefore yields
\begin{equation}\label{eq:path-W2-conditional-LLN}
 \cW_{2,E_T}(\widehat{\mathbf P}_N,\mathbf P^0)\rightarrow0
 \qquad\text{almost surely conditionally on }\cF_T^0,
\end{equation}
and hence in probability unconditionally.

The evaluation map $\omega\mapsto\omega(t)$ is $1$-Lipschitz from $E_T$ to $\R^n$, so
\[
 \sup_{t\le T}\cW_2(\widehat\mu_t^N,\mu_t)
 \le \cW_{2,E_T}(\widehat{\mathbf P}_N,\mathbf P^0).
\]
The centered $L^2$ radius is itself $1$-Lipschitz under $\cW_2$: for any coupling $(U,V)$,
\[
 \left|\|U-\E U\|_2-\|V-\E V\|_2\right|\le\|U-V\|_2.
\]
Consequently
\begin{equation}\label{eq:samplevariance-uniform}
 \sup_{t\le T}\Big|\sqrt{\widehat S_t^N}-\sqrt{s_t}\Big|
 \le \sup_{t\le T}\cW_2(\widehat\mu_t^N,\mu_t)
 \rightarrow0
 \qquad\text{in probability},
\end{equation}
and therefore also $\sup_{t\le T}|\widehat S_t^N-s_t|\to0$ in probability.  On $\{t<\tau_N\}$ both empirical radii lie in $J_{\rm out}$, so the squared amplitude difference is uniformly bounded. We use a time-integrated sampling error, which is zero if the initial sample lies outside the stopping window and $\tau_N=0$. The preceding convergence and bounded convergence in probability imply
\begin{equation}\label{eq:etaN}
 \eta_N(T):=\E\int_0^{T\wedge\tau_N}
 |(\widehat S_t^N)^q-s_t^q|^2\,\dd t\rightarrow0.
\end{equation}

Apply It\^o's formula to $\cE_N$.  The term containing $A_i$ is nonpositive by \eqref{eq:monotone-poc}.  Split $\Xi_i=\Xi_i^L+\Xi_i^{L,\mathrm{tail}}$.  Young's inequality and \eqref{eq:Rbound-truncated} control the truncated part, while \eqref{eq:Rtail-energy-poc} controls the tail after time integration.  The It\^o correction from the common and idiosyncratic noises is bounded, using \eqref{eq:amp-local} and \eqref{eq:etaN}, by $C_T(\cE_N+|(\widehat S^N)^q-s^q|^2)$.  The same quantities control the quadratic variation of the error martingale.  The Burkholder--Davis--Gundy and Young inequalities therefore yield, for each fixed $L$,
\begin{equation}\label{eq:stoppedE}
 \E\sup_{u\le t\wedge\tau_N}\cE_N(u)
 \le C_T\left(\frac{C_L}{N}+\eta_N(T)+\e_L
 +\int_0^t\E\sup_{v\le s\wedge\tau_N}\cE_N(v)\,\dd s\right).
\end{equation}
Gr\"onwall's inequality gives
\begin{equation}\label{eq:stopped-coupling-final}
 \limsup_{N\to\infty}\E\sup_{t\le T\wedge\tau_N}\cE_N(t)
 \le C_T\e_L.
\end{equation}
Letting $L\to\infty$ proves that the left-hand side vanishes.

The stopping can be removed without a global particle moment estimate.  Put
\[
 d_{\star}:=\operatorname{dist}\bigl(J_{\rm in},[0,\infty)\setminus J_{\rm out}\bigr)>0.
\]
By \eqref{eq:path-W2-conditional-LLN}--\eqref{eq:samplevariance-uniform},
\begin{equation}\label{eq:iid-inner-containment}
 \P\bigl(r_N(\hX_t)\in J_{\rm in}\text{ for every }t\le T\bigr)\rightarrow1.
\end{equation}
If the event in \eqref{eq:iid-inner-containment} occurs but $\tau_N\le T$, then the interacting radius must reach $J_{\rm out}^c$ at the stopping time.  The Lipschitz estimate \eqref{eq:rLip} therefore forces
\[
 \sup_{t\le T\wedge\tau_N}\cE_N(t)\ge d_{\star}^2.
\]
Thus
\begin{align*}
 \P(\tau_N\le T)
 &\le \P\bigl(\exists t\le T:r_N(\hX_t)\notin J_{\rm in}\bigr)
 +d_{\star}^{-2}\E\sup_{t\le T\wedge\tau_N}\cE_N(t),
\end{align*}
and \eqref{eq:stopped-coupling-final}, followed by $L\to\infty$, gives
\begin{equation}\label{eq:stop-prob-vanish}
 \P(\tau_N\le T)\rightarrow0.
\end{equation}
For any fixed bounded outer radius level, a finite particle explosion before $T$ must cross that level first.  While the relative radius is bounded, the center has bounded quadratic variation and hence cannot explode.  Repeating the preceding argument with an outer interval containing that level proves
\eqref{eq:particle-exit-before-mf}.

Finally,
\[
 \cW_2(\mu_t^N,\widehat\mu_t^N)^2\le\cE_N(t).
\]
The conditional i.i.d. law of large numbers gives
\[
 \sup_{t\le T}\cW_2(\widehat\mu_t^N,\mu_t)\rightarrow0
 \qquad\text{in probability},
\]
while \eqref{eq:stopped-coupling-final}--\eqref{eq:stop-prob-vanish} control the first term uniformly in time.  On $\{\tau_N^{\rm expl}>T\}$ the triangle inequality proves \eqref{eq:pocW2}. Together with \eqref{eq:particle-exit-before-mf} this gives the asserted survival-event convergence in $C([0,T];\cP_2)$. 
\end{proof}

On compact variance windows, the same synchronous coupling is quantitative for every $q\ge0$.
Define
\[
 r_t^{N,\mathrm{int}}:=r_N(X_t^N)=\sqrt{S_t^N},
 \qquad
 \widehat r_t^N:=r_N(\hX_t)=\sqrt{\widehat S_t^N}.
\]
These are, respectively, the centered empirical radius of the interacting particle system and the centered empirical radius of the conditionally i.i.d.\ comparison system. 

\begin{proposition}\label{prop:poc-local-rate}
Assume the hypotheses of \autoref{thm:poc} and
\[
 \mu_0\in\cP_{m_{\star}}(\R^n),\qquad
 m_{\star}:=
 \begin{cases}
  2p+2,&q=0,\\
  \max\{4,2p+2\},&q>0.
 \end{cases}
\]
Fix a compact radius interval $J=[r_-,r_+]\subset[0,\infty)$ on which the map $r\mapsto r^{2q}$ is Lipschitz. Let
\[
 \tau_N^J:=\inf\{t\ge0:r_t^{N,\mathrm{int}}\notin J
 \text{ or }\widehat r_t^N\notin J\}.
\]
Then
\begin{equation}\label{eq:poc-local-allq}
 \sup_{0\le t\le T}\E\,\cE_N(t\wedge\tau_N^J)\le\frac{C_T}{N}.
\end{equation}
Thus, the mean-square coupling error is $O(N^{-1})$ on every compact radius window on which the noise amplitude is Lipschitz.
\end{proposition}

\begin{proof}
Write $e_i=X^{i,N}-\hX^i$ and use the decomposition $A_i+\Xi_i$ from the proof of \autoref{thm:poc}.  On the stopped radius interval there is a deterministic constant $L_q$ such that
\begin{equation}\label{eq:local-rate-amplitude-lip}
 |x^{2q}-y^{2q}|\le L_q|x-y|
\end{equation}
for empirical radii in $J$ and limiting radii $r_t$, $0\le t\le T$. If necessary, enlarge $J$ to include the deterministic radius path when choosing this Lipschitz constant. For $0<q<1/2$, both sets stay away from zero. The stopping time itself is unchanged. The radius coupling estimate \eqref{eq:rLip} remains valid.

We first record the two sampling estimates.  Conditionally on $(\cF_t^0,\hX_t^i)$, the variables
$K(\hX_t^j-\hX_t^i)$, $j\ne i$, are independent with conditional mean
$\int_{\R^n} K(y-\hX_t^i)\mu_t(\dd y)$.  Since
$|K(z)|^2=|z|^{2p+2}$ and the assumed $m_{\star}$-th moment controls $2p+2$,
\begin{equation}\label{eq:poc-local-R-bound}
 \sup_{t\le T}\frac1N\sum_{i=1}^N\E|\Xi_i(t)|^2\le\frac{C_T}{N}.
\end{equation}
The missing self-interaction term contributes only an $O(N^{-2})$ squared bias.  Suppose now that $q>0$.  The definition of $m_{\star}$ then also gives a propagated fourth moment.  After subtracting the common center, the $\hX^i$ are i.i.d. with deterministic centered law $\nu_t$.  Writing
\begin{equation}\label{eq:radius-in-y}
 \widehat S_t^N=\frac1N\sum_{i=1}^N|\widehat Y_t^i|^2-|\overline{\widehat Y}_t^N|^2, \mbox{ where } \overline{\widehat Y}_t^N:=\frac1N\sum_{i=1}^N\widehat Y_t^i,
\end{equation}
therefore gives
\begin{equation}\label{eq:poc-local-sample-var}
 \sup_{t\le T}\E|\widehat S_t^N-s_t|^2\le\frac{C_T}{N}.
\end{equation}
If $s_0>0$, \autoref{lem:mf-nonattainment} gives $\inf_{t\le T}s_t>0$, and division by the sum of the radii in \eqref{eq:poc-local-sample-var} yields
\begin{equation*}
 \sup_{t\le T}\E|r_N(\hX_t)-r_t|^2\le\frac{C_T}{N}.
\end{equation*}
If instead $s_0=0$ and $q\ge1/2$, uniqueness at consensus implies that the centered limiting law is $\delta_0$, so $r_t=r_N(\hX_t)=0$ identically.  When $q=0$, the diffusion amplitude is constant, and the entire sample-variance step is omitted. In particular, no fourth-moment assumption is needed.

Combining \eqref{eq:rLip}, \eqref{eq:local-rate-amplitude-lip}, and these sampling estimates gives, up to $\tau^J_N$,
\begin{equation}\label{eq:poc-local-amplitude-error}
\E\left[\one_{\{t<\tau_N^J\}}\left|(S_t^N)^q-s_t^q\right|^2\right]
 \le C_T\Big(\E\cE_N(t\wedge\tau^J_N)+\frac1N\Big).
\end{equation}
Apply It\^o's formula to $\cE_N=N^{-1}\sum_{i=1}^N|e_i|^2$ stopped at $\tau^J_N$.  The interaction-difference term is nonpositive by \eqref{eq:monotone-poc}.  For the empirical interaction-field fluctuation, Young's inequality and \eqref{eq:poc-local-R-bound} give
\[
 \frac2N\sum_{i=1}^N\E[e_i\cdot \Xi_i]\le
 \E\cE_N+\frac{C_T}{N}.
\]
The quadratic variation of the common and idiosyncratic error noises contributes a constant multiple of
$\E|(S_t^N)^q-s_t^q|^2$.  Therefore \eqref{eq:poc-local-amplitude-error} yields
\[
 \frac{\dd}{\dd t}\E\cE_N(t\wedge\tau^J_N)
 \le C_T\E\cE_N(t\wedge\tau^J_N)+\frac{C_T}{N}
\]
for almost every $t$.  Since $\cE_N(0)=0$, Gr\"onwall's inequality proves
\eqref{eq:poc-local-allq}. 
\end{proof}

The exit probability itself can also be quantified before a nonexplosive mean-field horizon, at the cost of two additional moments.  This gives a partial quantitative result in the supercritical range where an ordinary post-explosion $\cW_2$ distance is not defined.

\begin{proposition}\label{prop:poc-exit-rate}
Assume $q\ge1/2$, fix $T$ strictly inside the lifetime of a unique mean-field solution, and suppose
\begin{equation}\label{eq:poc-exit-moment}
 \mu_0\in\cP_{\max\{4,2p+4\}}(\R^n).
\end{equation}
Let $r_t=s_t^{1/2}$ and choose $R>\sup_{t\le T}r_t$.  Let $\widehat r_t^N$ be the empirical centered radius of the conditionally independent comparison system and let
\[
 \tau_N^R:=\inf\{t\ge0:r_N(X_t^N)\ge R\text{ or }\widehat r_t^N\ge R\}.
\]
Then
\begin{equation}\label{eq:poc-exit-rate}
 \P(\tau_N^R\le T)\le\frac{C_{T,R}}{N},
 \qquad
 \P(\tau_N^{\rm expl}\le T)\le\frac{C_{T,R}}{N}.
\end{equation}
\end{proposition}
Together with \autoref{prop:poc-local-rate}, this Proposition gives an $O(N^{-1})$ mean-square coupling error on the safe-radius event $\{\tau_N^R>T\}$, on which both $r_N(X_t^N)$ and $\widehat r_t^N$ remain below $R$ for all $t\le T$, and an $O(N^{-1})$ probability of leaving that regime before $T$.

\begin{proof}
Write $\widehat Y^i$ for the centered conditionally independent copies.  Their law is deterministic by \autoref{thm:translation}.  From \eqref{eq:radius-in-y}, we put
\[
 Q_N(t):=\frac1N\sum_{i=1}^N|\widehat Y_t^i|^2,
 \qquad
 (\widehat r_t^N)^2=\widehat S_t^N = Q_N(t)-|\overline{\widehat Y}_t^N|^2.
\]
As $T$ is strictly before the mean-field lifetime, $s$ is bounded on $[0,T]$.  The dissipative moment estimate propagates the $M$-th moment in \eqref{eq:poc-exit-moment}.  If
\[
 G_t(x):=2x\cdot b_{\nu_t}(x)+c s_t^{2q},
\]
then It\^o's formula gives, for a martingale $\mathcal M_t^N$,
\[
 Q_N(t)-s_t
 =Q_N(0)-s_0
 +\int_0^t\frac1N\sum_{i=1}^N\bigl(G_u(\widehat Y_u^i)-\E G_u(\widehat Y_u^i)\bigr)\,\dd u
 +\mathcal M_t^N.
\]
The growth bound $|G_t(x)|^2\le C_T(1+|x|^{2p+4}+M_{2p+4}(t))$ and conditional independence imply
\[
 \E\sup_{t\le T}|Q_N(t)-s_t|^2\le\frac{C_T}{N}.
\]
Here the drift average is estimated by the Cauchy--Schwarz inequality in time and the martingale by the Doob--Burkholder--Davis--Gundy inequality.  Similarly, since $\int_{\R^n} b_{\nu_t}(x)\nu_t(\dd x)=0$ and $|b_{\nu_t}(x)|^2$ has growth at most $2p+2$,
\[
 \E\sup_{t\le T}|\overline{\widehat Y}_t^N|^2\le\frac{C_T}{N}.
\]
Consequently, for every fixed $\delta>0$,
\begin{equation}\label{eq:comparison-radius-sup}
 \P\Big(\sup_{t\le T}|\widehat r_t^N-r_t|\ge\delta\Big)
 \le\frac{C_{T,\delta}}{N}.
\end{equation}
Indeed, for positive $a,b$, $|a-b|\ge\delta$ implies $|a^2-b^2|\ge\delta^2$, and the sample-variance difference is bounded by
$|Q_N-s|+|\overline{\widehat Y}^N|^2$.

Choose $\delta>0$ so that $\sup_{t\le T}r_t+3\delta<R$.  On the event complementary to \eqref{eq:comparison-radius-sup} with threshold $2\delta$, if the interacting radius reaches $R$ before time $T$, then at the first such time
\[
 |r_N(X^N)-\widehat r^N|\ge\delta.
\]
By the $1$-Lipschitz radius estimate \eqref{eq:rLip}, this implies
$\mathcal E_N(\tau_N^R)\ge\delta^2$.  \autoref{prop:poc-local-rate}, applied on the radius interval $[0,R]$, gives
$\E\mathcal E_N(T\wedge\tau_N^R)\le C_{T,R}/N$.  Markov's inequality and \eqref{eq:comparison-radius-sup} therefore prove the first estimate in \eqref{eq:poc-exit-rate}.  A finite-particle explosion before $T$ must cross the level $R$ first, because while the relative radius is bounded the center has finite quadratic variation on $[0,T]$.  Hence
$\{\tau_N^{\rm expl}\le T\}\subset\{\tau_N^R\le T\}$, proving the second estimate. 
\end{proof}

For $q\le1/2$, a global secant bound comparing the empirical radius with the deterministic limiting radius yields an unstopped finite-horizon rate.

\begin{theorem}\label{thm:poc-quant}
Assume the hypotheses of \autoref{thm:poc}, let $0\le q\le1/2$, and suppose
\begin{equation}\label{eq:poc-quant-moment}
 \mu_0\in\cP_{m_{\star}}(\R^n),\qquad
 m_{\star}:=
 \begin{cases}
  2p+2,&q=0,\\
  \max\{4,2p+2\},&0<q\le1/2.
 \end{cases}
\end{equation}
If $0<q<1/2$, assume $s_0>0$.  At $q=1/2$ no lower-variance assumption is needed: if $s_0>0$, finite-time non-attainment gives $\inf_{t\le T}s_t>0$, while if $s_0=0$ both centered systems remain at consensus.  For every fixed $k\in\N$ there is $C_{T,k}<\infty$, independent of $N\ge\max\{k,2\}$, such that
\begin{equation}\label{eq:poc-fixed-k-rate}
 \sup_{0\le t\le T}
 \E \cW_2^2\left(
 \Law\bigl((X_t^{1,N},\ldots,X_t^{k,N})\mid\cF_t^0\bigr),
 \mu_t^{\otimes k}\right)
 \le \frac{C_{T,k}}{N}.
\end{equation}
Here $\cW_2$ is computed with the Euclidean product metric on $\R^{nk}$. 
\end{theorem}

\begin{proof}
If $q=1/2$ and $s_0=0$, then $\mu_0$ is a Dirac mass.  The i.i.d. particle initial condition is exact consensus almost surely, and both the particle system and the mean-field equation remain at that same constant configuration.  The estimate is therefore trivial.  Hence in the remainder of the proof, whenever $q>0$ we may assume $s_0>0$. \autoref{lem:mf-nonattainment} gives $\inf_{t\le T}s_t>0$.

Use the synchronous coupling from the proof of \autoref{thm:poc}, but do not truncate the empirical interaction-field fluctuation.  The moment assumption \eqref{eq:poc-quant-moment}, the boundedness of $s$ on $[0,T]$, and the dissipative moment calculation give
\begin{equation}\label{eq:poc-quant-moments-propagated}
 \sup_{t\le T}\E|\hX_t|^{m_{\star}}<\infty.
\end{equation}
For $q=0$ globality follows from part~\ref{item:particle-constant} of \autoref{thm:particle-allq}.  For $q=1/2$ it follows from part~\ref{item:particle-lipschitz} of \autoref{thm:particle-allq}.  If $0<q<1/2$, the system is global for $m_N\ge2$. In the exceptional case $m_N=1$, we use the unique absorbing continuation when $q\ge1/4$ and choose the absorbing boundary continuation when $q<1/4$.  Thus the stochastic integrals below may be localized and the localization removed by Fatou's lemma.

For the same  empirical interaction-field fluctuation notation,
\[
 \Xi_i(t)=\frac1N\sum_{j=1}^N K(\hX_t^j-\hX_t^i)
 -\int_{\R^n} K(y-\hX_t^i)\,\mu_t(\dd y),
\]
conditional independence of the off-diagonal variables and $|K(z)|^2=|z|^{2p+2}$ imply
\begin{equation}\label{eq:poc-R-L2-rate}
 \sup_{t\le T}\frac1N\sum_{i=1}^N\E|\Xi_i(t)|^2\le\frac{C_T}{N}.
\end{equation}
The $j=i$ term contributes an $O(N^{-2})$ bias and is covered by the same moment bound.

If $q>0$, as in \eqref{eq:radius-in-y}, we can write $\widehat S_t^N$ in terms of i.i.d. relative variables with deterministic centered law $\nu_t$.  The definition of $m_{\star}$ gives $m_{\star}\ge4$, so \eqref{eq:poc-quant-moments-propagated} yields
\begin{equation}\label{eq:poc-sample-variance-rate}
 \sup_{t\le T}\E|\widehat S_t^N-s_t|^2\le\frac{C_T}{N}.
\end{equation}
As $\inf_{t\le T}s_t>0$, with $r_t=s_t^{1/2}$ and $\widehat r_t^N=(\widehat S_t^N)^{1/2}$, following the proof of Proposition~\ref{prop:poc-local-rate}, \eqref{eq:poc-sample-variance-rate} therefore yields
\begin{equation*}
 \sup_{t\le T}\E|\widehat r_t^N-r_t|^2\le\frac{C_T}{N}.
\end{equation*}
For $0<q<1/2$, put $\alpha=2q\in(0,1)$. The positive lower bound $r_t\ge \underline{r}>0$ and concavity give
\[
 |x^\alpha-r_t^\alpha|\le \underline{r}^{\alpha-1}|x-r_t|,
 \qquad x\ge0.
\]
For $q=1/2$ one simply has $|x^{2q}-r_t^{2q}|=|x-r_t|$, with no lower bound needed.  For $q=0$ the amplitudes are constant and no sample-variance estimate is used.  Combining these observations with the $1$-Lipschitz property of the empirical radius gives
\begin{equation}\label{eq:poc-amp-rate}
 \E\bigl|(S_t^N)^q-s_t^q\bigr|^2
 \le C_T\Big(\E\cE_N(t)+\frac1N\Big).
\end{equation}

Apply It\^o's formula to $\cE_N=N^{-1}\sum_{i=1}^N|X^{i,N}-\hX^i|^2$.  The interaction-difference term is nonpositive by \eqref{eq:monotone-poc}.  Young's inequality, \eqref{eq:poc-R-L2-rate}, and \eqref{eq:poc-amp-rate} give
\[
 \frac{\dd}{\dd t}\E\cE_N(t)
 \le C_T\E\cE_N(t)+\frac{C_T}{N}.
\]
Since $\cE_N(0)=0$, Gr\"onwall's inequality yields
\begin{equation}\label{eq:poc-coupling-rate}
 \sup_{t\le T}\E\cE_N(t)\le\frac{C_T}{N}.
\end{equation}
The localized martingale estimate is uniform in the localization level, so Fatou's lemma removes the localization.

Conditionally on $\cF_t^0$, the vector $(\hX_t^1,\ldots,\hX_t^k)$ has law $\mu_t^{\otimes k}$.  The synchronous construction is therefore a conditional coupling, and exchangeability gives
\begin{align*}
 \E \cW_2^2\left(
 \Law((X_t^{1,N},\ldots,X_t^{k,N})\mid\cF_t^0),
 \mu_t^{\otimes k}\right)
 &\le \sum_{i=1}^k\E|X_t^{i,N}-\hX_t^i|^2 =k\,\E\cE_N(t).
\end{align*}
Taking the supremum over $t\le T$ and using \eqref{eq:poc-coupling-rate} we obtain \eqref{eq:poc-fixed-k-rate}.
\end{proof}

Averaging the conditional limit over the common noise gives the corresponding unconditional mixture limit.

\begin{corollary}\label{cor:poc-unconditional-mixture}
Under the assumptions of \autoref{thm:poc-quant}, fix $k\in\N$ and define the deterministic mixture measure on $(\R^n)^k$ by
\begin{equation*}
 Q_t^{(k)}:=\E\bigl[\mu_t^{\otimes k}\bigr].
\end{equation*}
Then
\begin{equation}\label{eq:unconditional-mixture-rate}
 \sup_{0\le t\le T}
 \cW_2^2\left(
 \Law(X_t^{1,N},\ldots,X_t^{k,N}),Q_t^{(k)}
 \right)
 \le\frac{C_{T,k}}{N}.
\end{equation}
\end{corollary}
If $\rho=0$, the limiting law $\mu_t$ is deterministic and $Q_t^{(k)}=\mu_t^{\otimes k}$, so \eqref{eq:unconditional-mixture-rate} is the usual unconditional propagation of chaos.  If $\rho>0$, the correct unconditional limit is generally the mixture $Q_t^{(k)}$, not the product of the unconditional one-particle marginal, because the common noise leaves nontrivial correlations after averaging over $B^0$.

\begin{proof}
Use the synchronous comparison variables $(\hX^1,\ldots,\hX^k)$ from the proof of \autoref{thm:poc-quant}.  Conditionally on the common noise they have law $\mu_t^{\otimes k}$, and hence their unconditional joint law is exactly $Q_t^{(k)}$.  Therefore this synchronous construction is an unconditional coupling of the two measures in \eqref{eq:unconditional-mixture-rate}, and
\[
 \cW_2^2\left(\Law(X_t^{1,N},\ldots,X_t^{k,N}),Q_t^{(k)}\right)
 \le \sum_{i=1}^k\E|X_t^{i,N}-\hX_t^i|^2
 =k\,\E\cE_N(t).
\]
The bound follows from \eqref{eq:poc-coupling-rate}. 
\end{proof}

\section{Long-time behavior}\label{sec:long-time}

\subsection{A deterministic upper barrier for the mean-field variance}

The sharp mean-field coercivity constant gives the following deterministic variance barrier.

\begin{theorem}\label{thm:barrier}
Assume $c>0$, $2q<\bar{p}$, and let $(\nu_t)_{t\ge0}$ be a global centered mean-field solution with finite $(p+2)$-moment at every time. Assume that its variance $s_t=\int_{\R^n}|x|^2\nu_t(\dd x)$ is absolutely continuous and satisfies \eqref{eq:s-exact} for almost every $t\ge0$. 
Then
\begin{equation*}
 \limsup_{t\to\infty}s_t\le s_{\star}^{\sharp} \quad \text{ with } \quad s_{\star}^{\sharp}:=\Big(\frac{c}{\cA_{p,n}^{\rMF,\sharp}}\Big)^{1/(\bar{p}-2q)}.
\end{equation*}
If $s_0\le s_{\star}^{\sharp}$, then $s_t\le s_{\star}^{\sharp}$ for all $t\ge0$.  If $s_0>s_{\star}^{\sharp}$, there exist $C,\lambda>0$, depending on the parameters and $s_0$, such that
\begin{equation}\label{eq:onesided-general}
 (s_t-s_{\star}^{\sharp})_+\le Ce^{-\lambda t}.
\end{equation}
\end{theorem}

\begin{proof}
Let
\[
 f(s)=-\cA_{p,n}^{\rMF,\sharp}s^{\bar{p}}+cs^{2q}.
\]
The positive equilibrium of $\dot u=f(u)$ is exactly $s_{\star}^{\sharp}$, and $f(s)>0$ for $0<s<s_{\star}^{\sharp}$, while $f(s)<0$ for $s>s_{\star}^{\sharp}$.  Since $s'\le f(s)$, scalar comparison with the solution starting from $\max\{s_0,s_{\star}^{\sharp}\}$ gives the first two conclusions.  Moreover,
\[
 f'(s_{\star}^{\sharp})=-\cA_{p,n}^{\rMF,\sharp}(\bar{p}-2q)(s_{\star}^{\sharp})^{\bar{p}-1}<0.
\]
The comparison solution therefore approaches $s_{\star}^{\sharp}$ exponentially once it enters a fixed neighborhood of $s_{\star}^{\sharp}$, which proves \eqref{eq:onesided-general}. 
\end{proof}

\begin{remark}\label{rem:general-variance-rate}
For $p>0$, \autoref{thm:barrier} controls the excess
$(s_t-s_{\star}^{\sharp})_+$ but does not assert convergence of $s_t$ to
$s_{\star}^{\sharp}$. At positive variance, the exact evolution also depends on
the normalized dissipation $D_p(\nu_t)/s_t^{\bar{p}}$, whose long-time behavior
is not determined by the scalar comparison.
\end{remark}

At critical homogeneity, strict coercivity yields exponential consensus for $p=0$ and the polynomial upper bound \eqref{eq:critical-alignment-rate} for $p>0$.

\begin{theorem}\label{thm:critical-align}
Let $p\ge0$ and $4q=p+2$. Let $\mu_0\in\cP_{p+2}(\R^n)$ be deterministic 
and let $\nu$ be the unique
maximal centered solution given by \autoref{thm:MV-allq}.

If $s_0=0$, then $\nu_t=\delta_0$ for all $t\ge0$.
Now suppose $s_0>0$. 

If $c\le\cA_{p,n}^{\rMF,\sharp}$, the solution is
global. More precisely, if $c<\cA_{p,n}^{\rMF,\sharp}$ and
$\delta:=\cA_{p,n}^{\rMF,\sharp}-c>0$, then, for every $t\ge0$,
\begin{equation}\label{eq:critical-alignment-rate}
 s_t\le
 \begin{cases}
 s_0e^{-\delta t},&p=0,\\[1mm]
 \left(s_0^{-p/2}+\dfrac p2\delta t\right)^{-2/p},&p>0.
 \end{cases}
\end{equation}
In particular, $s_t\to 0$ as $t\to\infty$.

At the boundary $c=\cA_{p,n}^{\rMF,\sharp}$, the variance $s_t$ is
nonincreasing. For $p=0$, one has $\cA_{0,n}^{\rMF,\sharp}=2$ and
$s_t\equiv s_0$. For $p>0$, define
\[
 \eta_t=(\mathsf{Dil}_{s_t^{-1/2}})\push\nu_t
 \qquad\text{whenever }s_t>0.
\]
Then, for almost every such $t$,
\begin{equation}\label{eq:critical-boundary-exact-dissipation}
 s_t'=-s_t^{2q}\bigl(D_p(\eta_t)-\cA_{p,n}^{\rMF,\sharp}\bigr)\le0.
\end{equation}
Consequently, $s_t\to s_\infty$ for some $s_\infty\in[0,s_0]$.
If $s_\infty>0$, then
\begin{equation}\label{eq:critical-boundary-integrable-defect}
 \int_0^\infty
 \bigl(D_p(\eta_t)-\cA_{p,n}^{\rMF,\sharp}\bigr)\,\dd t<\infty.
\end{equation}
For $p>0$, scalar coercivity alone does not determine whether
$s_\infty=0$.
\end{theorem}

\begin{proof}

Since $q=(p+2)/4\ge1/2$, part~\ref{item:mf-consensus} of
\autoref{thm:MV-allq} implies that $s_0=0$ gives the unique global
centered solution $\nu_t=\delta_0$. Now take $s_0>0$.
Under $c\le\cA_{p,n}^{\rMF,\sharp}$, the relation $2q=(p+2)/2=\bar{p}$ allows us to
apply part~\ref{item:mf-lipschitz} of \autoref{thm:MV-allq}, which
gives global existence.

The variance is locally absolutely continuous, and
\eqref{eq:s-exact}--\eqref{eq:sbarrier-sharp} give, for almost every
$t\ge0$,
\[
 s_t'=-D_p(\nu_t)+cs_t^{\bar{p}}
 \le-\bigl(\cA_{p,n}^{\rMF,\sharp}-c\bigr)s_t^{\bar{p}}.
\]
In particular, $s_t$ is nonincreasing whenever
$c\le\cA_{p,n}^{\rMF,\sharp}$.

Suppose first that $c<\cA_{p,n}^{\rMF,\sharp}$.
For $p=0$, one has $\bar{p}=1$, and Gr\"onwall's inequality gives
$s_t\le s_0e^{-\delta t}$. For $p>0$, on any interval where $s_t>0$,
the chain rule yields
\[
 \frac{\dd}{\dd t}s_t^{-p/2}
 =-\frac p2s_t^{-\bar{p}}s_t'
 \ge\frac p2\delta
 \qquad\text{for almost every }t.
\]
Integrating from $0$ to $t$ and rearranging gives the second bound in
\eqref{eq:critical-alignment-rate}, as long as $s_t>0$.
If the variance reaches zero, its nonnegativity and monotonicity imply
that it remains zero, so the same bound holds thereafter.
This proves \eqref{eq:critical-alignment-rate} for every $t\ge0$
and implies $s_t\to 0$.

Now let $c=\cA_{p,n}^{\rMF,\sharp}$.
Since $s_t$ is nonnegative and nonincreasing, it converges to some
$s_\infty\in[0,s_0]$. When $p=0$, the identity $D_0(\nu_t)=2s_t$
and $\cA_{0,n}^{\rMF,\sharp}=2$ reduce \eqref{eq:s-exact} to
$s_t'=0$ almost everywhere. Absolute continuity then gives
$s_t\equiv s_0$.

For $p>0$ and $s_t>0$, the law $\eta_t$ is centered and has unit
variance. Homogeneity of $D_p$ and the sharp coercivity bound \eqref{eq:D-MV-sharp-lower} give
\[
 D_p(\nu_t)=s_t^{\bar{p}}D_p(\eta_t),
 \qquad
 D_p(\eta_t)\ge\cA_{p,n}^{\rMF,\sharp}.
\]
Substitution into \eqref{eq:s-exact} proves
\eqref{eq:critical-boundary-exact-dissipation}.

Finally, let $p>0$ and suppose $s_\infty>0$.
Then $s_t\ge s_\infty$ for every $t\ge0$, so the chain rule applies
to $s_t^{-p/2}$ on every finite time interval.
By \eqref{eq:critical-boundary-exact-dissipation},
\[
 \frac{\dd}{\dd t}s_t^{-p/2}
 =\frac p2\bigl(D_p(\eta_t)-\cA_{p,n}^{\rMF,\sharp}\bigr)
 \qquad\text{for almost every }t\ge0.
\]
Integrating over $[0,T]$ and letting $T\to\infty$, using monotone
convergence for the nonnegative integrand, yields
\[
 \int_0^\infty
 \bigl(D_p(\eta_t)-\cA_{p,n}^{\rMF,\sharp}\bigr)\,\dd t
 =\frac2p\bigl(s_\infty^{-p/2}-s_0^{-p/2}\bigr)<\infty.
\]
This proves \eqref{eq:critical-boundary-integrable-defect}.
\end{proof}

For $p=2$, the equation \eqref{eq:p2-defect} in
\autoref{cor:p2-defect}, together with \autoref{thm:critical-align},
yields the following sufficient condition for consensus, with the optimal
constant in the dissipation bound based only on variance and dimension. 

\begin{corollary}\label{cor:p2-critical}
Let $\mu_0\in\cP_4(\R^n)$ be deterministic, and let $p=2$, $q=1$. For the unique maximal solution of \autoref{thm:MV-allq},
\begin{equation}\label{eq:p2-critical-condition}
 n\kappa^2(1-\rho^2)<4\left(1+\frac1n\right)
\end{equation}
implies global existence and consensus, with
\begin{equation*}
 s_t\le\frac{s_0}{1+\left[4(1+n^{-1})-n\kappa^2(1-\rho^2)\right]s_0t}.
\end{equation*}
If there is equality in \eqref{eq:p2-critical-condition}, \autoref{cor:p2-defect} sharpens mere monotonicity to
\[
 s_t'=-2\operatorname{Var}_{\nu_t}(|x|^2)
 -4\left\|\Sigma_t-\frac{s_t}{n}\Id_n\right\|_{\rm HS}^2\le0.
\]
\end{corollary}
The constant $4(1+1/n)$ cannot be increased in a dissipation bound using only variance and dimension, because equality in \eqref{eq:D2-MV-sharp} is attained by centered isotropic spherical laws. This optimality concerns the static inequality, not a necessary condition for consensus of the evolving noisy law.

\subsection{Pure common noise: consensus and motion of the cluster}

The case $\rho=1$ is special because the relative dynamics contains no stochastic term at all, regardless of the regularity of $\psi$.

When the noise is purely common, the relative configuration is deterministic, and the center carries all 
randomness. The notation $f(t)\simeq g(t)$ means that there exist constants $c_1,c_2 >0$ and $t_0\geq 0$ such that $c_1g(t)\leq f(t)\leq c_2 g(t)$ for all $t\geq t_0$.  

\begin{theorem}\label{thm:rho1}
Assume $\rho=1$
and $S_0^N>0$.  Then the relative system is deterministic and has a unique global solution for every $p,q\ge0$.  Moreover:
\begin{enumerate}[label=\textnormal{(\roman*)}]
\item If $p=0$,
\begin{equation}\label{eq:rho1-p0}
 S_t^N=S_0^Ne^{-2t}.
\end{equation}
\item If $p>0$,
\begin{equation}\label{eq:rho1-twosided}
 \left((S_0^N)^{-p/2}+\frac p2\cB_{p,n,N}^{\sharp}t\right)^{-2/p}
 \le S_t^N\le
 \left((S_0^N)^{-p/2}+\frac p2\cA_{p,n,N}^{\sharp}t\right)^{-2/p}.
\end{equation}
Hence $S_t^N\simeq t^{-2/p}$ as $t\to\infty$.
\end{enumerate}
The center satisfies
\begin{equation}\label{eq:rho1-center}
 \barX_t^N=\barX_0^N+\kappa\int_0^t(S_s^N)^q\,\dd B_s^0.
\end{equation}
Consequently, for $p\geq 0$,
\begin{equation}\label{eq:center-phase}
 \int_0^\infty(S_s^N)^{2q}\,\dd s
 \begin{cases}
 <\infty,&q>p/4,\\
 =\infty,&0\le q\le p/4.
 \end{cases}
\end{equation}

Therefore the particles reach consensus.  The common cluster converges to a finite random point exactly when $q>p/4$ for $p>0$, and exactly when $q>0$ for $p=0$.  For $p>0$, the center has logarithmically divergent quadratic variation at $q=p/4$ and polynomially divergent quadratic variation when $0\le q<p/4$. In the separate case $p=q=0$, the quadratic variation in each coordinate is exactly $\kappa^2t$.
\end{theorem}

\begin{proof}
When $\rho=1$, \autoref{prop:decomposition} gives a deterministic relative system and \autoref{prop:variance-general} reduces to
\[
 \frac{\dd}{\dd t}S_t^N=-D_{p,N}(Y_t).
\]
For $p=0$, \autoref{lem:D-bounds} gives $D_{0,N}=2S^N$, proving \eqref{eq:rho1-p0}.  For $p>0$, homogeneity and \eqref{eq:shape-extrema-global} give the optimal uniform bounds
\[
 -\cB_{p,n,N}^{\sharp}(S_t^N)^{\bar{p}}
 \le (S_t^N)'
 \le-\cA_{p,n,N}^{\sharp}(S_t^N)^{\bar{p}}.
\]
Integrating both inequalities yields \eqref{eq:rho1-twosided}.  The fully explicit outer bounds follow by \eqref{eq:shape-extrema-explicit-bracket}.

Equation \eqref{eq:rho1-center} follows from \eqref{eq:centerSDE}.  Its quadratic variation is $\kappa^2\int_0^t(S_s^N)^{2q}\dd s$ in each coordinate.  Formula \eqref{eq:rho1-p0} gives the assertion when $p=0$.  If $p>0$, the two-sided estimate \eqref{eq:rho1-twosided} shows that $(S_t^N)^{2q}\simeq t^{-4q/p}$, proving \eqref{eq:center-phase}. For deterministic initial data the finite total quadratic variation here is deterministic, so the center is bounded in $L^2$ and converges almost surely and in $L^2$. For random initial data the assertion holds conditionally on the initial configuration, unconditional $L^2$ convergence additionally requires integrability of the total quadratic variation and of the initial center.  If the quadratic variation is infinite, the Dambis--Dubins--Schwarz representation 
identifies each coordinate with a Brownian motion run for an unbounded amount of time, so it cannot converge to a finite limit.  Finally, $\max_i|Y_t^i|^2\le NS_t^N\to0$, which transfers the center behavior to every particle. 
\end{proof}

For $\rho=1$ and deterministic initial data, the center process can be written as 
\begin{equation*}\label{eq:rho1-timechanged-bm}
 \bigl(\bar X_t^N\bigr)_{t\ge0}
 \stackrel{\rm law}{=} 
 \bigl(\bar X_0^N+W_{A_N(t)}\bigr)_{t\ge0},
\end{equation*}
where $W$ is a Brownian motion and for $p>0$ and $0\le q\le p/4$, its deterministic clock has the asymptotics
\[ A_N(t):=\kappa^2\int_0^t(S_s^N)^{2q}\,\dd s \simeq
 \begin{cases}
  \log t,&q=p/4,\\[1mm]
  t^{1-4q/p},&0\le q<p/4,
 \end{cases} \]
using \eqref{eq:rho1-twosided}. Consequently, from Brownian scaling, the natural spatial scale of the wandering consensual cluster is $\sqrt{\log t}$ at $q=p/4$ and $t^{(1-4q/p)/2}$ for $q<p/4$.
\section{Linear interaction: explicit results}\label{sec:p0}

In this section we set $p=0$ and hence $K(z)=z$. This case admits exact formulas that settle the behavior of the system for all $q\ge0$.

Recall that $c=n\kappa^2(1-\rho^2)$, and throughout this section we set $s_{\star}:=\left(c/2\right)^{1/(1-2q)}$ when $c>0$ and $q\ne 1/2$.

\subsection{The mean-field equation}
By \autoref{prop:variance-MV}, the mean-field variance satisfies
\begin{equation}\label{eq:sODEp0}
 s'=-2s+cs^{2q}.
\end{equation}
When $q=0$, the convention is $s^0=1$, including at $s=0$.

Linearity of the interaction yields an explicit representation of the conditional law.

\begin{theorem}\label{thm:p0-explicit}
Assume $p=0$ and $\mu_0\in\cP_2(\R^n)$. Let $s$ be a deterministic nonnegative solution of \eqref{eq:sODEp0} on an interval $[0,T)$, with $s(0)=s_0:=\int_{\R^n}|x-m_0|^2\mu_0(\dd x)$.   Let $m_0=\int_{\R^n} x\mu_0(\dd x)$, let $Y_0$ have the centered law $(\tau_{-m_0})\push\mu_0$, and set
\begin{align}
 m_t&=m_0+\kappa\rho\int_0^t s_u^q\,\dd B_u^0,\label{eq:p0-m}\\
 Y_t&=e^{-t}Y_0+\kappa\sqrt{1-\rho^2}
 \int_0^t e^{-(t-u)}s_u^q\,\dd B_u.\label{eq:p0-Y}
\end{align}
Then $X_t=m_t+Y_t$ solves the McKean--Vlasov equation and
\begin{equation}\label{eq:p0-law}
 \mu_t=(\tau_{m_t})\push\left[
 (\mathsf{Dil}_{e^{-t}})\push\nu_0 * \cN(0,\theta_t\Id_n)
 \right],
\end{equation}
where $\nu_0=(\tau_{-m_0})\push\mu_0$,  and
\begin{equation*}
 \theta_t=\kappa^2(1-\rho^2)\int_0^t e^{-2(t-u)}s_u^{2q}\,\dd u.
\end{equation*}
The variance satisfies
\begin{equation}\label{eq:variance-theta}
 s_t=e^{-2t}s_0+n\theta_t.
\end{equation}
Conversely, every mean-field solution whose centered conditional law is deterministic has this form.  In particular, when $s_0>0$, the scalar ODE is locally Lipschitz along its positive trajectory, so the representation covers every solution up to its maximal lifetime.
\end{theorem}

\begin{proof}
For $p=0$, the centered drift in \eqref{eq:centeredMV} is
\[
 \int_{\R^n}(z-Y_t)\,\nu_t(\dd z)=-Y_t
\]
because $\nu_t$ is centered.  Thus \eqref{eq:centeredMV} is a linear Ornstein--Uhlenbeck equation with deterministic time-dependent diffusion coefficient once $s$ is fixed.  Variation of constants gives \eqref{eq:p0-Y}, while the center is \eqref{eq:p0-m}.  The Gaussian convolution formula \eqref{eq:p0-law} follows from the same representation.  Taking second moments in \eqref{eq:p0-Y} gives \eqref{eq:variance-theta}. Differentiating \eqref{eq:variance-theta} reproduces \eqref{eq:sODEp0}.  Conversely, \autoref{thm:translation} reduces every solution with deterministic centered conditional law to this centered linear equation, so the representation is necessary.  If $s_0>0$, uniqueness of the positive scalar ODE forces the variance path, and hence the centered conditional law, to be deterministic. 
\end{proof}

The scalar variance equation is explicit.  If $c>0$, $s_0>0$, and $q\ne1/2$, the change of variables $z=s^{1-2q}$ gives
\begin{equation}\label{eq:s-explicit-qneq}
 s_t=
 \left[
 \frac c2+\left(s_0^{1-2q}-\frac c2\right)e^{-2(1-2q)t}
 \right]^{1/(1-2q)}
\end{equation}
for as long as the bracket is positive, whereas for $q=1/2$,
\begin{equation}\label{eq:s-explicit-half}
 s_t=s_0e^{(c-2)t}.
\end{equation}
If $c=0$, then $s_t=s_0e^{-2t}$ for every $q\ge0$.

\begin{theorem}\label{thm:p0-phase}
Assume $\mu_0\in\cP_2(\R^n)$ and $s_0>0$.  If $c>0$, the following trichotomy holds.
\begin{enumerate}[label=\textnormal{(\roman*)}]
\item \textnormal{(Subcritical noise.)}
If $0\le q<1/2$, then $s_t\to s_{\star}$ as $t\to \infty$, and the centered law converges in $\cW_2$,
\begin{equation*}
 \nu_t\rightarrow \cN\left(0,\frac{s_{\star}}{n}\Id_n\right) \quad \text{ as } t\to\infty.
\end{equation*}
For $q=0$, $m_t=m_0+\kappa\rho B_t^0$.  For $0<q<1/2$,
\begin{equation*}
 \frac1t\langle m\rangle_t
 \rightarrow \kappa^2\rho^2s_{\star}^{2q}\Id_n.
\end{equation*}
Hence, for $\rho<1$, there is no consensus: the centered cloud converges to a nondegenerate Gaussian profile, with fixed center for $\rho=0$ and diffusive center for $\rho>0$.

\item \textnormal{(Critical noise.)}
Let $q=1/2$. If $c<2$, then $s_t\to0$ exponentially, $\nu_t\to\delta_0$ in $\cW_2$, and
\begin{equation*}
 m_t\rightarrow m_\infty
 :=m_0+\kappa\rho\int_0^\infty\sqrt{s_0}\,e^{-(2-c)u/2}\,\dd B_u^0 \quad \text{ as } t\to\infty
\end{equation*}
almost surely and in $L^2$. The limit $m_\infty$ is Gaussian with mean $m_0$ and covariance
\begin{equation*}
 \Cov(m_\infty)=\frac{\kappa^2\rho^2s_0}{2-c}\Id_n.
\end{equation*}
Hence $\mu_t\to\delta_{m_\infty}$ almost surely in $\cW_2$.

If $c=2$, then $s_t\equiv s_0$,
\begin{equation*}
 \nu_t\rightarrow\cN\left(0,\frac{s_0}{n}\Id_n\right),
\end{equation*}
and $m_t$ is Brownian with covariance matrix $\kappa^2\rho^2s_0\Id_n$.

If $c>2$, then $s_t=s_0e^{(c-2)t}\to\infty$ as $t\to\infty$.  The solution is 
global for every finite time.

\item \textnormal{(Supercritical noise.)}
Let $q>1/2$. 
If $0<s_0<s_{\star}$, then $s_t\to 0$, $\nu_t\to\delta_0$, and $m_t$ converges almost surely and in $L^2$ to a finite random vector $m_\infty$.  Consequently $\mu_t\to\delta_{m_\infty}$ almost surely in $\cW_2$.

If $s_0=s_{\star}$, then $s_t\equiv s_{\star}$,
\begin{equation*}
 \nu_t\rightarrow\cN\left(0,\frac{s_{\star}}{n}\Id_n\right),
\end{equation*}
and the center is Brownian with covariance matrix $\kappa^2\rho^2s_{\star}^{2q}\Id_n$.

If $s_0>s_{\star}$, then the variance blows up at the finite deterministic time
\begin{equation}\label{eq:Tstar}
 T_{\star}=
 \frac1{2(2q-1)}
 \log\Big(
 \frac{c/2}{c/2-s_0^{1-2q}}
 \Big),\quad \text{ and }
\end{equation}

\begin{equation}\label{eq:blowup-asymptotic}
 s_t\sim\bigl(c(2q-1)(T_{\star}-t)\bigr)^{-1/(2q-1)}
 \qquad (\text{as } t\to T_{\star}).
\end{equation}
\end{enumerate}
If $c=0$, then $s_t=s_0e^{-2t}$ for every $q\ge0$. Hence the centered cloud reaches consensus exponentially quickly, and the center dynamics are those of \autoref{thm:rho1}.
\end{theorem}

\begin{proof}
The variance assertions follow directly from \eqref{eq:s-explicit-qneq}--\eqref{eq:s-explicit-half}.  Formula \eqref{eq:variance-theta} then determines the Gaussian covariance in \eqref{eq:p0-law}: when $s_t$ converges to a positive limit $s_\infty$, one has $\theta_t\to s_\infty/n$, while $s_t\to0$ forces the centered law to collapse because the contracted initial part also vanishes.  The center has quadratic variation
\[
 \langle m\rangle_t=\kappa^2\rho^2\int_0^t s_u^{2q}\,\dd u\,\Id_n,
\]
which gives the stated convergence or diffusive behavior.  For $q>1/2$, the bracket in \eqref{eq:s-explicit-qneq} stays positive exactly for $s_0\le s_{\star}$. When $s_0>s_{\star}$ its first zero is \eqref{eq:Tstar}, and a first-order expansion there gives \eqref{eq:blowup-asymptotic}.  The case $c=0$ is immediate.
\end{proof}

\subsection{Nonuniqueness at consensus in the sub-Lipschitz regime}

The Hölder singularity at zero variance is not a technical artifact.
The singular consensus state supports a genuine nonuniqueness mechanism.

\begin{theorem}\label{thm:nonunique}
Assume $p=0$, $0<q<1/2$, $\rho<1$, and $\mu_0=\delta_{x_0}$.  Then the McKean--Vlasov equation has infinitely many solutions in law.

More precisely, for every waiting time $\tau\in[0,\infty]$, define
\begin{equation}\label{eq:waiting-s}
 s^{(\tau)}_t=
 \begin{cases}
 0,&0\le t\le\tau,\\[1mm]
 \left[\dfrac c2\left(1-e^{-2(1-2q)(t-\tau)}\right)\right]^{1/(1-2q)},&t>\tau,
 \end{cases}
\end{equation}
with the convention that $s^{(\infty)}\equiv0$.  Each $s^{(\tau)}$ solves \eqref{eq:sODEp0} with $s_0=0$, and \autoref{thm:p0-explicit} constructs a distinct McKean--Vlasov solution.  The choice $\tau=\infty$ is the absorbing consensus solution.
\end{theorem}

\begin{proof}
For $t>\tau$, formula \eqref{eq:waiting-s} is exactly the solution of \eqref{eq:sODEp0} obtained by starting from zero at time $\tau$.  Since $1/(1-2q)>1$, its right derivative at $t=\tau$ is zero, so the function is continuously differentiable there and satisfies the ODE on all of $[0,\infty)$.  \autoref{thm:p0-explicit} then gives a solution for each variance path.  Distinct waiting times give distinct variances and therefore distinct laws. 
\end{proof}

The zero-variance boundary can therefore be classified completely in the linear case.

\begin{corollary}\label{cor:zero-variance-p0}
Assume $c>0$, and $s_0=0$.
\begin{enumerate}[label=\textnormal{(\roman*)}]
\item If $q=0$, the variance is uniquely determined and
\[
 s_t=\frac c2(1-e^{-2t}).
\]
The initially concentrated law immediately spreads and its centered profile converges to $\cN(0,(c/2n)\Id_n)$.
\item If $0<q<1/2$, there are infinitely many solutions, including the waiting-time family in \autoref{thm:nonunique}.
\item If $q\ge1/2$, the only finite-second-moment solution has $s_t\equiv0$.  Hence $X_t\equiv x_0$ when $\mu_0=\delta_{x_0}$.
\end{enumerate}
\end{corollary}

\begin{proof}
For $q=0$, the scalar equation is the linear ODE $s'=-2s+c$, which gives the displayed formula and uniqueness.  The Gaussian limit follows from \autoref{thm:p0-explicit}.  The second assertion is \autoref{thm:nonunique}.  If $q\ge1/2$, the vector field $f(s)=-2s+cs^{2q}$ is locally Lipschitz at zero, so the scalar initial value problem with $s_0=0$ has the unique solution $s\equiv0$.  By \eqref{eq:p0-m} 
the centered law remains as $\delta_0$. Since $q>0$ in this case, the center noise amplitude also vanishes. 
\end{proof}

\begin{remark}\label{rem:selection}
The nonuniqueness in \autoref{thm:nonunique} is sensitive to how consensus is approximated.  If every particle starts at the same deterministic point, then $S_0^N=0$ and the identically constant configuration is a finite-particle solution because both the interaction and the noise amplitude vanish.  The absorbing particle approximation therefore selects $s\equiv0$.

By contrast, positive-variance regularization selects the immediate-departure solution.  Let $s^{(\e)}$ be the unique positive solution of \eqref{eq:sODEp0} with $s_0^{(\e)}=\e>0$.  The explicit formula gives
\begin{equation*}
 s_t^{(\e)}
 =\left[\frac c2\bigl(1-e^{-2(1-2q)t}\bigr)
 +\e^{1-2q} e^{-2(1-2q)t}\right]^{1/(1-2q)}
 \rightarrow
 \left[\frac c2\bigl(1-e^{-2(1-2q)t}\bigr)\right]^{1/(1-2q)}
 =s_t^{(0)}
\end{equation*}
uniformly on every bounded time interval, where $s^{(0)}$ is the waiting-time solution \eqref{eq:waiting-s} with $\tau=0$.  Thus there is no approximation-independent selection principle at the singular boundary: exact consensus with absorbing microscopic dynamics selects $\tau=\infty$, whereas vanishing strictly positive variance selects $\tau=0$.
\end{remark}

\subsection{Finite-particle variance diffusion in the linear case}

Assume from now on in this subsection that $\rho<1$.  We retain the relative dimension
$m_N=n(N-1)$ and recall the notation 
\begin{equation*}
 c_N=n\kappa^2(1-\rho^2)\Bigl(1-\frac1N\Bigr)=\frac{m_N}{4}a_N^2,
 \qquad
 a_N=2\kappa\sqrt{\frac{1-\rho^2}{N}}.
\end{equation*}

At finite $N$, the empirical variance itself is an autonomous one-dimensional diffusion.

\begin{proposition}\label{prop:scalarSN}
Suppose $S_0^N>0$.  Up to the first hitting time of $0$ or explosion, there exists a one-dimensional Brownian motion $\beta$ such that
\begin{equation}\label{eq:scalarSN}
 \dd S_t^N
 =\left[-2S_t^N+c_N(S_t^N)^{2q}\right]\dd t
 +a_N(S_t^N)^{q+1/2}\,\dd\beta_t.
\end{equation}
\end{proposition}

\begin{proof}
For $p=0$, \autoref{lem:D-bounds} gives $D_{0,N}=2S^N$.  \autoref{prop:variance-general} therefore gives the drift in \eqref{eq:scalarSN}, while \eqref{eq:bracket-power} gives
\[
 \dd\langle M^N\rangle_t=a_N^2(S_t^N)^{2q+1}\,\dd t.
\]
On the set $S_t^N>0$, 
L\'evy's characterization gives $\beta$, a Brownian motion, by normalizing $M^N$. 
\end{proof}

One-dimensional diffusion theory then gives the complete finite-particle classification for the linear interaction.

\begin{theorem}\label{thm:finiteN-phase}
Assume $S_0^N>0$, and $m_N\ge 1$.

\begin{enumerate}[label=\textnormal{(\roman*)}]
\item\label{item:finite-linear-subcritical} \textnormal{(Subcritical noise.)} Let $0\le q<1/2$.  If $m_N\ge2$, the variance diffusion is nonexplosive, remains strictly positive, and is positive recurrent.  If $m_N=1$ and $q=0$, it is positive recurrent but hits zero.  In these recurrent cases the invariant density is
\begin{equation}\label{eq:invariant-density}
 \pi_N(s)=\frac1{Z_N^{\rm var}}
 s^{m_N/2-2q-1}
 \exp\left[-\frac{N}{\kappa^2(1-\rho^2)(1-2q)}s^{1-2q}\right],
 \qquad s>0,
\end{equation}
where $Z_N^{\rm var}$ is the normalizing constant so that $\int_0^\infty \pi_N(s)\dd s = 1$.
Thus, starting from $\pi_N$,
\begin{equation}\label{eq:gamma-stationary}
 \lambda_N(S^N)^{1-2q}\sim\mathrm{Gamma}(k_N,1),
\end{equation}
where 
\begin{equation*}
  \lambda_N:=\frac{N}{\kappa^2(1-\rho^2)(1-2q)},
 \qquad
 k_N=\frac{m_N-4q}{2(1-2q)}.
\end{equation*}
If $m_N=1$ and $0<q<1/2$, the variance reaches zero in finite time almost surely.  The absorbing continuation is uniquely selected when $q\ge1/4$, while for $0<q<1/4$ the post-hitting continuation is nonunique. 

\item\label{item:finite-linear-critical} \textnormal{(Critical noise.)} Let  $q=1/2$. The variance is geometric Brownian motion:
\begin{equation}\label{eq:GBM}
 S_t^N=S_0^N\exp\left(\Lambda_Nt+a_N\beta_t\right),
\end{equation}
where
\begin{equation}\label{eq:LambdaN}
 \Lambda_N=-2+c_N-\frac{a_N^2}{2}
 =-2+\kappa^2(1-\rho^2)\left(n-\frac{n+2}{N}\right).
\end{equation}
Hence $S_t^N\to0$ almost surely if $\Lambda_N<0$, $S_t^N\to\infty$ almost surely if $\Lambda_N>0$, and if $\Lambda_N=0$,
\begin{equation}\label{eq:oscillatory}
 \liminf_{t\to\infty}S_t^N=0,
 \qquad
 \limsup_{t\to\infty}S_t^N=\infty
 \quad\text{almost surely}.
\end{equation}

\item\label{item:finite-linear-supercritical} \textnormal{(Supercritical noise.)} Let $q>1/2$. If $m_N\le2$, finite-time explosion has probability zero and $S_t^N\to0$ almost surely.  If $m_N>2$, the explosion probability is
\begin{equation}\label{eq:explosion-prob}
 \P_{S_0^N}(\tau_N^{\rm expl}<\infty)
 =Q\left(k_N^+, \lambda_N^+(S^N_0)^{1-2q}\right),
\end{equation}
where $Q(k,x):=\Gamma(k,x)/\Gamma(k)$ is the regularized upper incomplete gamma function and
\begin{equation*}
 k_N^+=\frac{m_N-2}{2(2q-1)},
 \qquad
  \lambda_N^+=\frac{N}{(2q-1)\kappa^2(1-\rho^2)}.
\end{equation*}
On the complementary nonexplosion event, $S_t^N\to0$ as $t\to\infty$.
\end{enumerate}
\end{theorem}

\begin{proof}
If $m_N=1$ and $0<q<1/2$, part~\ref{item:finite-linear-subcritical} is exactly \autoref{cor:two-particle}, so for the recurrent calculation suppose either $m_N\ge2$ or $q=0$.  In the case $m_N=1$, $q=0$, the relative coordinate is a one-dimensional Ornstein--Uhlenbeck process. Its square is the unique instantaneously reflecting nonnegative solution of \eqref{eq:scalarSN}.  For \eqref{eq:scalarSN},
\[
 \mu(s)=-2s+c_Ns^{2q},\qquad \sigma(s)^2=a_N^2s^{2q+1},
 \qquad \frac{2c_N}{a_N^2}=\frac{m_N}{2}.
\]
When $q<1/2$, the standard speed-density formula
\[
 \pi(s)\ \propto \ \sigma(s)^{-2}
 \exp\left(\int^s\frac{2\mu(u)}{\sigma(u)^2}\,\dd u\right)
\]
gives \eqref{eq:invariant-density}. It is integrable at zero in the recurrent cases under consideration ($m_N>4q$) and at infinity because of the stretched exponential.   The scale--speed test gives positive recurrence and, when $m_N\ge2$, strict positivity. The case $m_N=1$, $q=0$ is instantaneously reflecting at zero.  The substitution $x=\lambda_N s^{1-2q}$ yields \eqref{eq:gamma-stationary}.

At $q=1/2$, \eqref{eq:scalarSN} is geometric Brownian motion, so It\^o's formula for $\log S^N$ gives \eqref{eq:GBM},\eqref{eq:LambdaN}. The sign of $\Lambda_N$ gives the transient alternatives, while Brownian recurrence gives \eqref{eq:oscillatory} at $\Lambda_N=0$.

Finally, for $q>1/2$ put $\alpha=q-1/2$ and $R=(\alpha a_N)^{-1}(S^N)^{-\alpha}$.  It\^o's formula with \eqref{eq:scalarSN} gives
\begin{equation*}
 \dd R_t=2\alpha R_t\,\dd t+\frac{\delta_N-1}{2R_t}\,\dd t+\dd \beta_t,
 \qquad \delta_N=\frac{4q-m_N}{2q-1}.
\end{equation*}
The explosion of $S^N$ corresponds to the hitting of zero by $R$, whose scale density is
\begin{equation*}
 \mathfrak s'(r)=r^{1-\delta_N}e^{-2\alpha r^2}.
\end{equation*}
Thus zero is accessible exactly when $m_N>2$.  In that case, writing $T_0=\inf\{t\geq 0:R_t=0\}$,
\begin{equation*}
 \P_r(T_0<\infty)
 =\frac{\int_r^\infty u^{1-\delta_N}e^{-2\alpha u^2}\,\dd u}
 {\int_0^\infty u^{1-\delta_N}e^{-2\alpha u^2}\,\dd u}.
\end{equation*}
The substitution $x=2\alpha u^2$, together with $2\alpha R_0^2= \lambda_N^+(S_0^N)^{1-2q}$, gives \eqref{eq:explosion-prob}.  If $T_0=\infty$, standard one-dimensional boundary classification gives $R_t\to\infty$, and hence $S_t^N\to0$, see \cite[Section~5.5]{KaratzasShreve1991}.
\end{proof}

The scalar classification immediately yields the corresponding consensus alternatives.

\begin{corollary}\label{cor:finiteN-alignment}
Under the assumptions of \autoref{thm:finiteN-phase}:
\begin{enumerate}[label=\textnormal{(\roman*)}]
\item\label{item:finite-align-subcritical} If $0\le q<1/2$ and $m_N\ge2$, consensus does not occur: the variance is positive recurrent and cannot converge to zero.  If $m_N=1$ and $q=0$, the nondegenerate one-dimensional Ornstein--Uhlenbeck relative coordinate does not converge to consensus. For $0<q<1/2$ the exceptional case is governed by \autoref{cor:two-particle}.
\item\label{item:finite-align-critical} If $q=1/2$, consensus occurs almost surely exactly when $\Lambda_N<0$.  In that case every particle converges to the same finite random point.
\item\label{item:finite-align-supercritical} If $q>1/2$, on the non-explosion event every particle converges to the same finite random point. The complementary event is finite-time explosion with probability given by \eqref{eq:explosion-prob} when $m_N>2$.
\end{enumerate}
\end{corollary}

\begin{proof}
Part~\ref{item:finite-align-subcritical} is immediate.  For $q>1/2$, \autoref{thm:finiteN-phase} gives $S_t^N\to0$ on the nonexplosion event. It\^o's formula reads
\[
 \log S_t^N=\log S_0^N-2t+
 \Bigl(c_N-\frac12a_N^2\Bigr)\int_0^t(S_u^N)^{2q-1}\dd u+M_t,
 \qquad
 \langle M\rangle_t=a_N^2\int_0^t(S_u^N)^{2q-1}\dd u.
\]
On this event, Ces\`aro averaging gives $\int_0^t(S_u^N)^{2q-1}\dd u=o(t)$. If the bracket stays finite, $M$ converges. Otherwise, the continuous-martingale strong law gives $M_t/\langle M\rangle_t\to0$. In either case $M_t/t\to0$, and therefore $t^{-1}\log S_t^N\to-2$.  Hence $\int_0^\infty(S_t^N)^{2q}\,\dd t<\infty$, the center has finite quadratic variation and converges, and $\max_i|Y_t^i|^2\le NS_t^N\to0$.  For $q=1/2$ the same conclusion follows from the explicit rate $\Lambda_N<0$ in \eqref{eq:GBM}.
\end{proof}

At the mean-field critical point, the finite-particle and long-time limits fail to commute.

\begin{corollary}\label{cor:noncommuting-limits}
Assume $q=1/2$, and $n\kappa^2(1-\rho^2)=2$.
Then, for every fixed $N\ge2$, on the event $\{S_0^N>0\}$,
\begin{equation}\label{eq:finite-critical-decay-rate}
 \lim_{t\to\infty}\frac1t\log S_t^N
 =-\frac{2(n+2)}{nN}<0
 \qquad\text{almost surely}.
\end{equation}

On the event $\{S_0^N=0\}$, one has $S_t^N=0$ for all $t\ge0$. Consequently, for every fixed $N\ge2$,
\[
 S_t^N\rightarrow0
 \qquad\text{almost surely as }t\to\infty.
\]
Suppose, in addition, that the initial particles are i.i.d. with law
$\mu_0\in\cP_4(\R^n)$ and variance $s_0>0$.
Then the mean-field variance satisfies $s_t\equiv s_0$, and for every
fixed $t\ge0$,
\[
 S_t^N\xrightarrow[]{\P}s_t=s_0 \qquad \text{ as } N\to\infty.
\]
Thus the large-population and long-time limits do not commute, that is,
\[
 \lim_{N\to\infty}\Bigl(\lim_{t\to\infty}S_t^N\Bigr)=0
 \quad\text{almost surely},
 \qquad
 \lim_{t\to\infty}s_t=s_0>0.
\]
\end{corollary}

\begin{proof}
At $q=\frac12$, \eqref{eq:scalarSN} is geometric Brownian motion and, as $n\kappa^2(1-\rho^2)=2$, the coefficients give
\[
 \dd S_t^N=-\frac2N S_t^N\dt
       +\sqrt{\frac{8}{nN}}\,S_t^N\dd\beta_t^N,
\]
where $\beta^N$ is a standard one-dimensional Brownian motion. Thus the variance can be written
\begin{equation}\label{eq:sgbm}
 S_t^N=S_0^N\exp\Big(
 -\frac{2(n+2)}{nN}\,t
 +\sqrt{\frac{8}{nN}}\,\beta_t^N
 \Big),
\end{equation}
including the absorbing solution when $S_0^N=0$.
For fixed $N$, the Brownian strong law gives
$\beta_t^N/t\to0$ almost surely as $t\to\infty$.
Taking logarithms on $\{S_0^N>0\}$ and dividing by $t$ proves
\eqref{eq:finite-critical-decay-rate}. Together with absorption at zero,
this yields $S_t^N\to0$ almost surely for every fixed $N$.

The mean-field variance equation \eqref{eq:sODEp0} reduces to
\[
 s_t'=\bigl[-2+n\kappa^2(1-\rho^2)\bigr]s_t=0,
\]
so $s_t\equiv s_0$. For i.i.d. initial particles, the law of large
numbers applied to the empirical first and second moments gives
$S_0^N\xrightarrow[]{\mathbb P}s_0$ as $N\to\infty$.
Moreover, for every fixed $t\ge0$,
\[
 \E\Big[
 \Big|\sqrt{\frac{8}{nN}}\,\beta_t^N\Big|^2
 \Big]=\frac{8t}{nN}\rightarrow 0 \quad \text{ as } N \to \infty.
\]
Hence the exponential factor in \eqref{eq:sgbm} converges
to $1$ in probability, and thus
$S_t^N\xrightarrow[]{\mathbb P}s_0$ as $N\to\infty$.

On a common probability space, the almost-sure limits for all $N\ge2$
hold simultaneously by taking a countable intersection of
probability-one events. This proves the first iterated-limit statement.
The second follows from $s_t\equiv s_0>0$.
\end{proof}

\subsection{Large-population limits of the finite-particle phase transitions}

The exact Gamma representation gives both the law of large numbers and the central-limit fluctuations of the invariant variance.

\begin{theorem}\label{thm:stationary-concentration}
Assume $\rho<1$, and $0\le q<1/2$.  For each $N$ in the recurrent regime of \autoref{thm:finiteN-phase}, let $S^N$ have density $\pi_N$.  Then, as $N\to\infty$, one has $S^N\xrightarrow{\P}s_{\star},$
and
\begin{equation}\label{eq:stationary-clt}
 \sqrt N\,(S^N-s_{\star})\xrightarrow{\rm law}
 \cN\Big(0,\frac{2s_{\star}^2}{n(1-2q)}\Big).
\end{equation}
\end{theorem}

\begin{proof}
By \eqref{eq:gamma-stationary}, set
\[
 G_N:=\lambda_N(S^N)^{1-2q}\sim\mathrm{Gamma}(k_N,1),
 \qquad
 \theta:=s_{\star}^{1-2q}=\frac{n\kappa^2(1-\rho^2)}2.
\]
Since a unit-rate Gamma variable with shape $k_N$ has both mean and variance equal to $k_N$,
\[
 \E\Big|\frac{G_N}{k_N}-1\Big|^2=\frac1{k_N}\rightarrow0,
 \qquad
 \frac{k_N}{\lambda_N}\rightarrow\theta.
\]
Hence $G_N/\lambda_N\to\theta$ in probability and therefore
$S^N=(G_N/\lambda_N)^{1/(1-2q)}\to s_{\star}$ in probability.

For the fluctuation statement, put
\[
 Z_N:=\frac{G_N-k_N}{\sqrt{k_N}}.
\]
For arbitrary positive $k_N\to\infty$, a standard argument gives 
$Z_N\Rightarrow\cN(0,1)$.  
Moreover,
\[
 \sqrt N\Big(\frac{G_N}{\lambda_N}-\theta\Big)
 =\frac{\sqrt{Nk_N}}{\lambda_N}Z_N
 +\sqrt N\Big(\frac{k_N}{\lambda_N}-\theta\Big),
\]
where
\[
 \frac{\sqrt{Nk_N}}{\lambda_N}\rightarrow
 \theta\sqrt{\frac{2(1-2q)}{n}},
 \qquad
 \sqrt N\Big(\frac{k_N}{\lambda_N}-\theta\Big)\rightarrow0.
\]
Applying the delta method to $h(x)=x^{1/(1-2q)}$ at $\theta$ gives
\[
 \sqrt N(S^N-s_{\star})
 \rightarrow
 \cN\Big(
 0,\,
 [h'(\theta)]^2\theta^2\frac{2(1-2q)}{n}
 \Big)
 =\cN\Big(0,\frac{2s_{\star}^2}{n(1-2q)}\Big),
\]
which is exactly \eqref{eq:stationary-clt}.
\end{proof}

The probability of the explosion event is dependent on the deterministic mean-field threshold.

\begin{theorem}\label{thm:explosion-limit}
Assume $\rho<1$, $q>1/2$, and let deterministic initial variances satisfy $S_0^N\to s_0>0$.  
If $s_0\ne s_{\star}$, then as $N\to\infty$
\begin{equation}\label{eq:explosion-limit}
 \P(\tau_N^{\rm expl}<\infty)\rightarrow
 \begin{cases}
 0,&s_0<s_{\star},\\
 1,&s_0>s_{\star}.
 \end{cases}
\end{equation}
If $S_0^N=s_{\star}$ for every $N$, then
\begin{equation*}
 \P(\tau_N^{\rm expl}<\infty)\rightarrow\frac12.
\end{equation*}
More generally, if
\begin{equation}\label{eq:critical-window-assumption}
 \sqrt N\left((S_0^N)^{1-2q}-\frac{n\kappa^2(1-\rho^2)}2\right)\rightarrow\zeta,
\end{equation}
then
\begin{equation}\label{eq:critical-window-limit}
 \P(\tau_N^{\rm expl}<\infty)\rightarrow
 1-\Phi\Big(
 \frac{\sqrt2\,\zeta}{\kappa^2(1-\rho^2)\sqrt{n(2q-1)}}
 \Big),
\end{equation}
where $\Phi$ is the standard normal distribution function.  Hence the finite-$N$ transition converges to the deterministic mean-field threshold, with a Gaussian critical window of width $N^{-1/2}$.
\end{theorem}

\begin{proof}
For all large $N$, \autoref{thm:finiteN-phase} gives $\P(\tau_N^{\rm expl}<\infty)=\P(G_N\ge x_N)$ with $G_N\sim\mathrm{Gamma}(k_N^+,1)$ and $x_N=\lambda_N^+(S_0^N)^{1-2q}$.  Since
\[
 \frac{x_N}{k_N^+}\rightarrow
 \frac{2s_0^{1-2q}}{n\kappa^2(1-\rho^2)},
\]
the law of large numbers from the previous proof gives \eqref{eq:explosion-limit}.  
Under \eqref{eq:critical-window-assumption},
\[
 \frac{x_N-k_N^+}{\sqrt{k_N^+}}
 \longrightarrow
 \frac{\sqrt2\,\zeta}{\kappa^2(1-\rho^2)\sqrt{n(2q-1)}}.
\]
By the CLT from the preceding proof and continuity of the normal distribution function, we have \eqref{eq:critical-window-limit}. $S_0^N=s_{\star}$ corresponds to $\zeta=0$.
\end{proof}

\begin{corollary}\label{cor:explosion-LDP}
Under the assumptions of \autoref{thm:explosion-limit}, suppose $s_0\ne s_{\star}$ and define
\[
 a_0:=\frac{2s_0^{1-2q}}{n\kappa^2(1-\rho^2)},\qquad I(a):=a-1-\log a.
\]
Then
\begin{align}
 -\frac1N\log\P(\tau_N^{\rm expl}<\infty)
 &\rightarrow \frac{n}{2(2q-1)}I(a_0),\qquad s_0<s_{\star},\label{eq:LDP-sub}\\
 -\frac1N\log\P(\tau_N^{\rm expl}=\infty)
 &\rightarrow \frac{n}{2(2q-1)}I(a_0),\qquad s_0>s_{\star}.\label{eq:LDP-super}
\end{align}
Hence the event on the wrong side of the limiting phase transition is exponentially unlikely in $N$.
\end{corollary}

\begin{proof}
Recall that $k_N^+= \frac{m_N-2}{2(2q-1)}$ and $G_N\sim\mathrm{Gamma}(k_N^+,1)$. By \autoref{thm:finiteN-phase}, the explosion probability is $\P(G_N\ge x_N)$, where $x_N/k_N^+\to a_0$ and $k_N^+/N\to n/[2(2q-1)]$. For $\theta<1$,
\[
 \frac1{k_N^+}\log\E\exp\Big(k_N^+\theta\frac{G_N}{k_N^+}\Big)
 =-\log(1-\theta),
 \qquad
 \sup_{\theta<1}\{\theta a+\log(1-\theta)\}=a-1-\log a.
\]
The supremum is achieved at $\theta=1-1/a$. Chernoff's inequality therefore gives the upper tail bound for $a>1$ and the lower tail bound for $0<a<1$. For matching exponential lower bounds, tilt by this parameter: under the tilted law $G_N$ has Gamma shape $k_N^+$ and rate $1/a$, so $G_N/k_N^+\to a$ in probability. Restricting the change-of-measure identity to $|G_N/k_N^+-a|<\e$ and then sending $\e\to 0$ gives exponential rate $I(a)$. Approaching a tail threshold from its interior gives the same rate by continuity of $I$.

Finally, bound the moving threshold $x_N/k_N^+$ between $a_0-\e$ and $a_0+\e$. If $s_0<s_\star$, then $a_0>1$ and explosion is the upper-tail event. If $s_0>s_\star$, then $0<a_0<1$ and nonexplosion is the lower-tail event. The fixed-threshold bounds, followed by $\e\to 0$ and multiplication by $k_N^+/N$, prove \eqref{eq:LDP-sub}--\eqref{eq:LDP-super}.
\end{proof}

\subsection{Quantitative propagation of chaos in the linear case}

In the linear case, the empirical variance $S^N$ satisfies the autonomous one-dimensional diffusion equation \eqref{eq:scalarSN}, while the limiting variance $s$ solves the autonomous ODE \eqref{eq:sODEp0}.  Fix $T$ before the lifetime of $s$ and stop $S^N$ at its first exit from a compact interval that contains the deterministic path $\{s_t:0\le t\le T\}$ in its interior. We show the expected squared supremum of the variance error up to this exit is $O(N^{-1})$.  
This stopped estimate controls the difference of the noise amplitudes and yields the finite-horizon particle coupling.

\begin{theorem}\label{thm:p0-poc}
Assume  $\rho<1$, $q\ge0$, and let the initial particles be i.i.d. with law $\mu_0\in\cP_4(\R^n)$ and variance $s_0>0$.  When $0<q<1/2$, we use the absorbing solution on the exceptional event $S_0^N=0$.  Let $s$ be the unique positive solution of \eqref{eq:sODEp0}, and let $T_{\star}\in(0,\infty]$ be its maximal lifetime.  For every $T<T_{\star}$, choose constants $0<\underline s<\overline s<\infty$ with $\underline s\le s_t\le\overline s$ on $[0,T]$, and define
\begin{equation*}
 \tau_N:=\inf\left\{t\ge0:S_t^N\notin[\underline s/2,2\overline s]\right\}.
\end{equation*}
Then there exists $N_0=N_0(T,\underline s,\overline s,n,q,\kappa,\rho,\mu_0)$ such that, for every $N\ge N_0$,
\begin{equation}\label{eq:SNrate-p0}
 \E\sup_{t\le T}|S_{t\wedge\tau_N}^N-s_{t\wedge\tau_N}|^2\le\frac{C_T}{N},
 \qquad
 \P(\tau_N\le T)\le\frac{C_T}{N}.
\end{equation}
In particular, $\P(\tau_N^{\rm expl}\le T)=O(N^{-1})$.

Moreover, if $\mu_t$ is the explicit conditional law from \autoref{thm:p0-explicit}, then for every $\e>0$,
\begin{equation}\label{eq:p0-pocW}
 \P\left(
 \tau_N^{\rm expl}>T,
 \ \sup_{t\le T}\cW_2(\mu_t^N,\mu_t)>\e
 \right)\rightarrow0.
\end{equation}
For $q\le1/2$ the particle system is global, so \eqref{eq:p0-pocW} is ordinary convergence in probability.  For every $q\ge0$, conditional propagation of chaos therefore holds up to the mean-field blow-up time.
\end{theorem}

\begin{proof}
On $[0,T]$ the deterministic variance remains in a compact interval $[\underline s,\overline s]\Subset(0,\infty)$.  Up to $\tau_N$, the coefficients of the scalar diffusion \eqref{eq:scalarSN} are therefore uniformly Lipschitz.  Subtracting \eqref{eq:sODEp0}, applying the Burkholder--Davis--Gundy and Gr\"onwall inequalities, and using
\[
 |c_N-c|=O(N^{-1}),\qquad a_N^2=O(N^{-1}),\qquad
 \E|S_0^N-s_0|^2=O(N^{-1}),
\]
gives the first estimate in \eqref{eq:SNrate-p0}.  
As $s_t\in[\underline s,\overline s]$ for $0\le t\le T$, continuity and the definition of $\tau_N$ imply, as $\underline{s}<\overline{s}$,
\[
 \{\tau_N\le T\}
 \subset
 \left\{\sup_{t\le T}
 |S_{t\wedge\tau_N}^N-s_{t\wedge\tau_N}|\ge \underline{s}/2\right\}.
\]
Squaring the supremum, applying Markov's inequality and using the first estimate in \eqref{eq:SNrate-p0} gives
\[
 \P(\tau_N\le T)\le 4\underline{s}^{-2}
 \E\sup_{t\le T}|S_{t\wedge\tau_N}^N-s_{t\wedge\tau_N}|^2
 \le \frac{C_T}{N},
\]
which is the second estimate. Moreover, any path that explodes before $T$ must first cross the finite upper boundary $2\overline s$, so
$\{\tau_N^{\rm expl}\le T\}\subset\{\tau_N\le T\}$. Hence $\P(\tau_N^{\rm expl}\le T)=O(N^{-1})$.

For the particle positions, couple $X^{i,N}$ synchronously with the explicit conditionally i.i.d. copies from \autoref{thm:p0-explicit}.  The stopped variance estimate gives
\[
 \E\int_0^{T\wedge\tau_N}|(S_t^N)^q-s_t^q|^2\,\dd t=O(N^{-1}).
\]
The linear centered equations, the Burkholder--Davis--Gundy inequality, and $\langle\bar B^N\rangle_t=t/N$ then give an averaged mean-square coupling error $O(N^{-1})$ up to $T\wedge\tau_N$, with the same bound for the centers.  Conditional empirical-measure convergence of the independent copies, together with $\P(\tau_N\le T)=O(N^{-1})$, yields \eqref{eq:p0-pocW}.
\end{proof}

In the linear nonexplosive regime, the same coupling also gives a quantitative path-space rate.

\begin{corollary}\label{cor:p0-path-rate}
Assume $\rho<1$, $0\le q\le1/2$, $s_0>0$, and $\mu_0\in\cP_4(\R^n)$.  Let
\[
 E_T:=C([0,T];\R^n),
 \qquad
 \mathbf P_T^0:=\Law(X_{[0,T]}\mid\cF_T^0)
\]
denote the conditional path law of the McKean--Vlasov solution.  For every fixed $k\in\N$ and $T<\infty$, there is $C_{T,k}<\infty$ such that, for all sufficiently large $N$,
\begin{equation}\label{eq:p0-path-rate}
 \E \cW_{2,E_T^k}^2\left(
 \Law\bigl((X^{1,N}_{[0,T]},\ldots,X^{k,N}_{[0,T]})\mid\cF_T^0\bigr),
 (\mathbf P_T^0)^{\otimes k}
 \right)
 \le\frac{C_{T,k}}{N}.
\end{equation}
Hence the linear nonexplosive regime admits an $O(N^{-1})$ squared-$\cW_2$ rate on path space.
\end{corollary}

\begin{proof}
Use the synchronous conditionally i.i.d. copies from \autoref{thm:p0-poc}.  For $q\le1/2$ the system is global, and the quantitative coupling argument of \autoref{thm:poc-quant}, specialized to $p=0$, yields
\[
 \int_0^T\E|(S_t^N)^q-s_t^q|^2\,\dd t\le C_TN^{-1}.
\]
The linear relative and center equations, the Burkholder--Davis--Gundy inequality, and $\langle\bar B^N\rangle_t=t/N$ therefore imply
\[
 \E\sup_{t\le T}|X_t^{i,N}-\hX_t^i|^2\le C_TN^{-1}.
\]
Conditionally on $\cF_T^0$, the paths $\hX^1,\ldots,\hX^k$ are independent with law $\mathbf P_T^0$.  Their synchronous coupling with the particle paths has quadratic transport cost at most $C_{T,k}/N$, which proves \eqref{eq:p0-path-rate}.
\end{proof}

\end{document}